%% file: SaturatedOC-OL-arx.tex
\documentclass[a4paper,reqno]{amsart}

\usepackage{amscd,amsthm,amsmath,amssymb,mathtools}
\usepackage{verbatim}
\usepackage{color}
\usepackage{url}
\usepackage{enumerate,enumitem}
\usepackage{graphicx}
\usepackage{epstopdf,epsfig,subfigure}
\usepackage{curve2e}
\usepackage{array}
\usepackage{stackrel}

\usepackage[numbers,sort&compress]{natbib}

\usepackage{hyperref}

\theoremstyle{plain}
 \newtheorem{mainresult}{Main Result}

\newtheorem{theorem}{Theorem}[section]

\newtheorem{lemma}[theorem]{Lemma}
\newtheorem{corollary}[theorem]{Corollary}
\newtheorem{problem}[theorem]{Problem}
\newtheorem{assumption}[theorem]{Assumption}

\theoremstyle{definition}
\newtheorem{definition}[theorem]{Definition}
\newtheorem{remark}[theorem]{Remark}

\numberwithin{equation}{section}

\usepackage{geometry}
\input{Mathcommands}
\begin{document}
\title{Global stabilizability to trajectories
for the Schl\"{o}gl  equation in a Sobolev norm}
\author{Karl Kunisch$^{\tt1,2}$,
\quad S\'ergio S.~Rodrigues$^{\tt1,2}$ }
\thanks{
\vspace{-1em}\newline\noindent
\newline\noindent
$^{\tt1}$ Johann Radon Institute for Computational and Applied Mathematics,
  \"OAW, Altenbergerstrasse~69, 4040~Linz, Austria.
\newline\noindent
$^{\tt2}$ Institute for Mathematics and Scientific Computing, Heinrichstrasse~36, 8010~Graz, Austria.
\newline\noindent
{\sc Emails}:
{\small\tt karl.kunisch@uni-graz.at,\quad
sergio.rodrigues@ricam.oeaw.ac.at}%
}

\begin{abstract}
The stabilizability to trajectories of the Schl\"ogl model is
investigated in the norm of the natural state space for strong solutions, which is strictly contained in the standard pivot space of square integrable functions.
 As actuators a finite number of indicator functions are used  and the control input is subject to a bound constraint.
A stabilizing  saturated explicit feedback control is proposed, where the set of actuators and the input bound are independent of the targeted trajectory. Further, the existence of open-loop optimal stabilizing constrained controls and related first-order optimality conditions are investigated. These conditions are then used to compute  stabilizing receding horizon based controls. Results of numerical simulations are presented comparing their stabilizing performance with that of saturated explicit feedback controls.
\end{abstract}

\maketitle

\pagestyle{myheadings} \thispagestyle{plain} \markboth{\sc K. Kunisch and  S. S. Rodrigues}
{\sc stabilizability to trajectories for the Schl\"ogl equation}

%%%%%%%%%%%%%%%%%%%%%%%%%%%%%
%%%%%%%%%%%%%%%%%%%%%%%%%%%%%
\section{Introduction}
We investigate the controlled Schl\"{o}gl system,
\begin{subequations}\label{sys-y-u-Cu}
 \begin{align}
& \tfrac{\p}{\p t} y_\ttc -\nu\Delta y_\ttc +(y_\ttc-\zeta_1)(y_\ttc-\zeta_2)(y_\ttc-\zeta_3)=h+{\textstyle\sum\limits_{i=1}^{M_\sigma} u_i\indf_{\omega^M_{i}}},\label{sys-y-u}\\
&y_\ttc(0)=y_{\ttc0}\in W^{1,2}(\Omega),\qquad \tfrac{\p}{\p\bfn}y_\ttc\rest{\p\Omega}=0,\label{y-icH1}\\
&\dnorm{u(t)}{}\le C_u,\quad t\ge0,\label{normu<Cu}
\end{align}
\end{subequations}
with state~$y_\ttc=y_\ttc(t,x)$, defined for~$(t,x)\in[0,+\infty)\times\Omega$, evolving in the Hilbert Sobolev space of square integrable functions with square integrable first-order derivatives, i.e. $y_\ttc(t,\Bigcdot)\in W^{1,2}(\Omega)\subset L^2(\Omega)$. Here~$\Omega\subset\bbR^d$ is a bounded rectangular domain, with~$d\in\{1,2,3\}$. The diffusion coefficient~$\nu>0$ is a positive constant and~$\zeta=(\zeta_1,\zeta_2,\zeta_3)\in\bbR^3$ is a constant vector. Further,~$M$ and~$M_\sigma$  are positive integers and $U_M\coloneqq\{\indf_{\omega^M_{i}}\mid 1\le i\le M_\sigma\}\subseteq L^2(\Omega)$ is a given family of $M_\sigma$~actuators, which are indicator functions of open subdomains~$\omega^M_{i}\subseteq\Omega$ depending on the parameter index~$M$, whose role will be clarified below.
Finally,~$h\in L^6_{\rm loc}(\bbR_+,L^2(\Omega))$ is a given external force satisfying, for some~$\tau_h>0$,
\begin{equation}\label{bound-h-intro}
\sup_{s\ge0}\norm{h}{L^6((s,s+\tau_h),L^2(\Omega))}<+\infty,
\end{equation}
and~$u=u(t)=(u_i(t),u_2(t),\dots,u_{M_\sigma}(t))$,  is a vector of scalar  input  controls (tuning parameters), the choice of which is at our disposal.

We consider the
case where the control~$u=u(t)$ is subject to constraints as in~\eqref{normu<Cu} for an a priori given constant~$C_u\in[0,+\infty]$ and an a priori given norm~$\dnorm{\Bigcdot}{}$ in~$\bbR^{M_\sigma}$. The usual Euclidean norm in~$\bbR^{M_\sigma}$ shall be denoted by~$\norm{\Bigcdot}{\bbR^{M_\sigma}}$.

We are particularly interested in the case~$C_u\in(0,+\infty)$, but we allow the extremal values~$C_u\in\{0,+\infty\}$ to include the cases of the free (uncontrolled) dynamics~$C_u=0$ and of unconstrained controls~$C_u=+\infty$.

 The total volume~$\vol(\bigcup_{i=1}^{M_\sigma}\omega^M_{i})$
covered by the actuators can be chosen a priori.

\subsection{Exponential stabilization to trajectories}
The  problem under investigation is as follows. We are given a trajectory/solution~$y_\ttt$ of the free dynamics,
\begin{subequations}\label{sys-haty}
 \begin{align}
& \tfrac{\p}{\p t} y_\ttt -\nu\Delta  y_\ttt +(y_\ttt-\zeta_1)(y_\ttt-\zeta_2)(y_\ttt-\zeta_3)=h,\\
&y_\ttt(0,\Bigcdot)=y_{\ttt0}\in W^{1,2}(\Omega),\qquad \tfrac{\p}{\p\bfn}y_\ttt\rest{\p\Omega}=0,
\end{align}
\end{subequations}
which we would like to track.

We are also given another initial state~$y_0\in W^{1,2}(\Omega)$. It turns out that the corresponding solution~$y$ of the free dynamics, with~$y(0,\Bigcdot)=y_0\in W^{1,2}(\Omega)$ may present an asymptotic behavior different from the targeted behavior of~$y_\ttt$. For example, in the case~$h=0$, and~$\zeta_1<\zeta_2<\zeta_3$ we could consider the free dynamics equilibrium
$y_\ttt(t,x)=\zeta_2$, with initial state~$y_\ttt(0,x)=y_{\ttt0}(x)=\zeta_2$, as our desired targeted behavior. Since~$y_\ttt(t,x)=\zeta_2$ is not a stable equilibrium,  if~$y_0(x)\coloneqq c\ne\zeta_2$ is a constant, then the  solution~$y(t,x)$  of the free dynamics, corresponding to the initial state~$y(0,x)=c$, does not converge to the targeted $y_\ttt(t,x)$  as time increases.

Our goal is  to design the control input~$u$ such that the controlled state~$y_\ttc(t,\Bigcdot)$ of the solution of system~\eqref{sys-y-u-Cu} converges exponentially to the given~$y_\ttt(t,\Bigcdot)$,  so that for some constants~$\varrho\ge1$ and~$\mu>0$, and all~$t\ge s\ge0$,
\begin{equation}\label{goal-exp}
\norm{ y_\ttc(t,\Bigcdot)-y_\ttt(t,\Bigcdot)}{W^{1,2}(\Omega)}^2\le \varrho\rme^{-\mu (t-s)}\norm{ y_\ttc(s,\Bigcdot)-y_\ttt(s,\Bigcdot)}{W^{1,2}(\Omega)}^2.
\end{equation}

We shall construct the stabilizing constrained control~$u$ by saturating a suitable unconstrainted stabilizing linear feedback control~$\clK\colon W^{1,2}(\Omega)\to\bbR^{M_\sigma}$ through a radial projection as follows
\begin{subequations}\label{rad.proj}
\begin{align}
&u=\overline\clK(y-y_\ttt)\coloneqq \fkP^{\dnorm{\Bigcdot}{}}_{C_u}(\clK (y-y_\ttt)),
&\intertext{where}
&\fkP^{\dnorm{\Bigcdot}{}}_{C_u}(v)\coloneqq\begin{cases}v,&\mbox{ if }\dnorm{v}{}\le C_u,\\
\frac{C_u}{\dnorm{v}{}}v,&\mbox{ if }\dnorm{v}{}> C_u,
\end{cases}\qquad v\in\bbR^{M_\sigma}
\end{align}
\end{subequations}
Note that we have
\begin{equation}\label{rad.proj.min}
\fkP^{\dnorm{\Bigcdot}{}}_{C_u}(0)=0\quad\mbox{and}\quad\fkP^{\dnorm{\Bigcdot}{}}_{C_u}(v)=\min\left\{1,\tfrac{C_u}{\dnorm{v}{}}\right\}v\quad\mbox{for}\quad v\ne0.
\end{equation}
and also that, for all $(v,C_u)\in\bbR^{M_\sigma}\times[0,+\infty]$,
\[\dnorm{\fkP^{\dnorm{\Bigcdot}{}}_{C_u}(v)}{}\le C_u,\qquad \fkP^{\dnorm{\Bigcdot}{}}_{0}(v)=0,\quad\mbox{and}\quad\fkP^{\dnorm{\Bigcdot}{}}_{+\infty}(v)=v.
\]
In particular, the saturated feedback control~$u(t)= \overline\clK(y(t)-y_\ttt(t))$ satisfies~$\dnorm{u(t)}{}\le C_u$.

The stabilizability of dynamical systems such as~\eqref{sys-y-u} is an important problem for applications, even in the case that the ``magnitude''~$\dnorm{u(t)}{}$ of the control is allowed to take arbitrary large values (i.e., in the case $C_u=+\infty$).

In applications we may be faced with physical constraints, for example, with an upper bound for the magnitude of the acceleration/forcing provided by an engine, or with an upper bound for the temperature provided by a heat radiator. For this reason it is also important to investigate the case of bounded controls (i.e., the case $C_u<+\infty$).

\begin{remark}\label{R:Msigma}
The stabilizability shall be proven to hold for a large number~$M_\sigma$ of actuators. That is why
we shall consider a sequence~$(U_M)_{M\in\bbN_+}$ of families of indicator functions~$U_M=\{\indf_{\omega^M_{i}}\mid 1\le i\le M_\sigma\}\subseteq L^2(\Omega)$ with supports~$\overline{\omega^M_{i}}$ depending on the sequence index~$M$.  Such dependence on~$M$ is also convenient to construct a sequence of families where the total volume covered by the actuators in each family is fixed a priori: given~$r\in(0,1)$, we construct~$(U_M)_{M\in\bbN_+}$ so that $\vol(\bigcup_{i=1}^{M_\sigma}\omega^M_{i})=r\vol(\Omega)$ for every~$M$.
\end{remark}

%%%%%%%%%%%%%%%%%%%%%%%%%%%
\subsection{The sequence of families of actuators}\label{sS:actuators}
For our rectangular spatial domain
\begin{equation}\label{Omegaxd}
\Omega=\Omega^\times=(0,L_1)\times (0,L_2)\times \cdots\times (0,L_d)\subset\bbR^d,\qquad d\in\{1,2,3\},
\end{equation}
we consider the set~$U_M$ of actuators as in
~\cite[sect.~4.8]{KunRod19-cocv} and~\cite[sect.~5]{KunRodWalter21},
\begin{subequations}\label{U_M}
\begin{align}
&U_M\coloneqq\{\indf_{\omega_{j}^M}\mid 1\le j\le M_\sigma\}\subseteq L^2(\Omega),\qquad\clU_M\coloneqq\linspan U_M,\qquad \dim\clU_M=M_\sigma,
\intertext{where, for a fixed~$M$, $M_\sigma=M^d$ and}
 &\omega_j^M\coloneqq{\bigtimes\limits_{n=1}^d}((c_n)_{j}^M-\tfrac{rL_n}{2M},(c_n)_{j}^M+\tfrac{rL_n}{2M}),\quad 1\le j\le M_\sigma,
\intertext{with centers}
 &\{c=(c)^M_j\mid 1\le j\le M_\sigma\}={\bigtimes\limits_{n=1}^d}\{(2k-1)\tfrac{L_n}{2M}\mid 1\le k\le M\}.
\end{align}
\end{subequations}
See Figure~\ref{fig.suppActSens} for an illustration for the case~$d=2$. See also~\cite[sect.~5.2]{Rod20-eect}
~\cite[sect.~6]{Rod21-aut} where an
analogue placement of the actuators and/or sensors has been used.
%%%%%%%%%%%%%%%%%%%%%%%%
%%%%%%   TIKZ BOXES   %%%%%%%%%
%%%%%%%%%%%%%%%%%%%%%%%%
\setlength{\unitlength}{.0018\textwidth}
\newsavebox{\Rectfw}%
\savebox{\Rectfw}(0,0){%
\linethickness{2pt}
{\color{black}\polygon(0,0)(120,0)(120,80)(0,80)(0,0)}%
}%
\newsavebox{\RectRef}%
\savebox{\RectRef}(0,0){%
\linethickness{1.5pt}
{\color{LightGray}\polygon*(40,30)(80,30)(80,50)(40,50)(40,30)}%
}%
 \begin{figure}[h!]
\begin{center}
\begin{picture}(500,100)%(0,0)
% Rect1
 \put(0,0){\usebox{\Rectfw}}
\put(0,0){\usebox{\RectRef}}
% Rect2
  \put(190,0){\usebox{\Rectfw}}
% %  %
   \put(190,0){\scalebox{.5}[.5]{\usebox{\RectRef}}}
    \put(250,0){\scalebox{.5}[.5]{\usebox{\RectRef}}}
   \put(190,40){\scalebox{.5}[.5]{\usebox{\RectRef}}}
  \put(250,40){\scalebox{.5}[.5]{\usebox{\RectRef}}}
%
% % Rect3
 \put(380,0){\usebox{\Rectfw}}
%  %
  \put(380,0){\scalebox{.333333}[.333333]{\usebox{\RectRef}}}
 \put(420,0){\scalebox{.333333}[.333333]{\usebox{\RectRef}}}
 \put(460,0){\scalebox{.333333}[.333333]{\usebox{\RectRef}}}
%  %
  \put(380,26.66666){\scalebox{.333333}[.333333]{\usebox{\RectRef}}}
 \put(420,26.66666){\scalebox{.333333}[.333333]{\usebox{\RectRef}}}
 \put(460,26.66666){\scalebox{.333333}[.333333]{\usebox{\RectRef}}}
%  %
  \put(380,53.33333){\scalebox{.333333}[.333333]{\usebox{\RectRef}}}
 \put(420,53.33333){\scalebox{.333333}[.333333]{\usebox{\RectRef}}}
 \put(460,53.33333){\scalebox{.333333}[.333333]{\usebox{\RectRef}}}
 \put(40,85){$M=1$}
 \put(230,85){$M=2$}
 \put(420,85){$M=3$}
\put(60,35){$\omega_1^1$}
\put(222,17){$\omega_1^2$}
\put(222,57){$\omega_2^2$}
\put(282,17){$\omega_3^2$}
\put(282,57){$\omega_4^2$}
\end{picture}
\end{center}
 \caption{Supports of actuators in a rectangle~$\Omega^\times\subset\bbR^2$.} \label{fig.suppActSens}
 \end{figure}
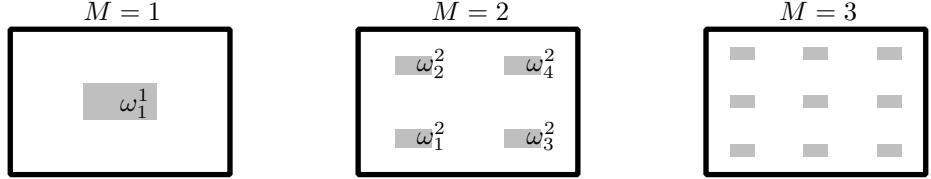

Together with the family~$U_M$ of actuators we will also consider a family of
auxiliary bump-like functions~$\widetilde U_M$ as  follows
\begin{subequations}\label{tildeU_M}
\begin{align}
&\widetilde U_M\coloneqq\left\{\Psi_{j}^M(x)\coloneqq\indf_{\omega_{j}^M}(x)\bigtimes_{j=1}^d\sin^2(\tfrac{ M\pi (x_j-c_j^M+\frac{rL}{2M})}{rL})\mid 1\le j\le M_\sigma\right\}\subseteq \rmD(A),\\
&\widetilde \clU_M\coloneqq\linspan\widetilde U_M,\qquad \dim\widetilde \clU_M=M_\sigma.
\end{align}
\end{subequations}

The orthogonal complement to a given subset~$B\subset \clH$ of a Hilbert space~$\clH$,
with scalar product~$(\Bigcdot,\Bigcdot)_\clH$,  is
denoted~$B^{\perp \clH}\coloneqq\{h\in \clH\mid (h,s)_\clH=0\mbox{ for all }s\in B\}$.
For simplicity, in the case~$\clH=L^2(\Omega)$ we denote
\[
 B^{\perp}\coloneqq B^{\perp L^2(\Omega)}.
\]

\begin{definition}
Given two closed subspaces~$F\subseteq \clH$ and~$G\subseteq \clH$ of the
Hilbert space~$\clH$, we write $\clH=F\oplus G$ if~$F\bigcap G=\{0\}$ and~$\clH=F+ G$ algebraically.
\end{definition}
\begin{definition}
Given two closed subspaces~$F\subseteq \clH$ and~$G\subseteq \clH$ of the
Hilbert space~$\clH=F\oplus G$, we denote by~$P_F^G\in\clL(\clH,F)$
the oblique projection in~$\clH$ onto~$F$ along~$G$. That is, writing $h\in \clH$ as $h=h_F+h_G$
with~$(h_F,h_G)\in F\times G$, we have
\[
 P_F^Gh\coloneqq h_F.
\]
The orthogonal projection in~$L^2(\Omega)$ onto~$F$ is denoted by~$P_F\coloneqq P_F^{F^{\perp }}\in\clL(L^2(\Omega),F)$.
\end{definition}

We can show, by the arguments in~\cite[Sect.~6]{Rod21-aut} (see also~\cite[Lem.~3.4]{KunRodWalter21}), that we have
\[
\widetilde \clU_M\oplus\clU_M^\perp=L^2(\Omega)=\clU_M\oplus\widetilde \clU_M^\perp.
\]

%%%%%%%%%%%%%%%%%%%%%%
\subsection{The main results}\label{sS:mainres-intro}
For simplicity, we define the isomorphism
\begin{equation}\label{UMdiam}
U_M^\diamond\colon\bbR^{M_\sigma}\to\clU_M,\qquad u\mapsto{\textstyle\sum\limits_{i=1}^{M_\sigma}}u_i\indf_{\omega^M_{i}}.
\end{equation}
Next, we observe that we can ``reduce'' the stabilizability to trajectories to the stabilizability to zero, simply by considering the error/difference~$z=y_\ttc-y_\ttt$ between the controlled solution of~\eqref{sys-y-u-Cu} and the targeted solution of~\eqref{sys-haty}.  In this case our goal is to stabilize the dynamics of the difference to zero, and
reads  (cf.~\eqref{goal-exp})
 \begin{equation}\label{goal-diff}
\norm{ z(t,\Bigcdot)}{W^{1,2}(\Omega)}^2\le \varrho\rme^{-\mu (t-s)}\norm{ z(s,\Bigcdot)}{W^{1,2}(\Omega)}^2,\quad\mbox{for all}\quad t\ge s\ge\tau\ge0.
\end{equation}

We shall consider an unconstrained explicit feedback operator, for a given~$\lambda\ge 0$, as
\begin{equation}\label{unc-K}
z\mapsto\clK_M^{\lambda}z\coloneqq-\lambda (U_{M}^\diamond)^{-1}P_{\clU_M}^{\widetilde \clU_M^\perp}(-\nu\Delta+\Id) P_{\widetilde \clU_M}^{\clU_M^\perp} z,
\end{equation}
 where~$\Id$ stands for the identity operator, and take saturated feedback controls as
\begin{equation}\label{const-K}
u=\overline\clK_M^{\lambda}(z)\coloneqq  \fkP^{\dnorm{\Bigcdot}{}}_{C_u}\left(-\lambda (U_{M}^\diamond)^{-1}P_{\clU_M}^{\widetilde \clU_M^\perp}(-\nu\Delta+\Id) P_{\widetilde \clU_M}^{\clU_M^\perp}z\right).
\end{equation}

By direct computations, the dynamics of the error, with this feedback control, reads
\begin{align} \label{sys-z-intro}
 &\tfrac{\p}{\p t} z =\nu\Delta z -f^{y_\ttt}(z)+U_{M}^\diamond u,\quad z(0,\Bigcdot)=z_0,\quad \tfrac{\p}{\p\bfn} z\rest{\p\Omega}=0,
 \end{align}
with~$z_0\in W^{1,2}(\Omega)$ and~$f^{y_\ttt}(z)=z^3+(3y_\ttt+\xi_2) z^2+(3y_\ttt^2+2\xi_2y_\ttt+\xi_1) z$, for suitable constants~$\xi_1$ and~$\xi_2$.

\begin{mainresult}\label{MR:mainStab}
For each~$\mu>0$ there exist large enough
constants~$M\in\bbN_+$, $\lambda\in\bbR_+$, and~$C_u\in\bbR_+$,
such that the solutions of~\eqref{sys-z-intro} with the feedback control~\eqref{const-K}, satisfy~\eqref{goal-diff}, for all $z_0\in W^{1,2}(\Omega)$.
\end{mainresult}

A more precise statement of this result is given in Theorem~\ref{T:mainSchloegl-st}.

We also address the constrained optimal control problem, concerned with the minimization of the classical energy functional
\begin{subequations}\label{Jy0}
\begin{equation}
J_{y_{\ttc0}}^{y_{\ttt0},Q}(y_\ttc-y_\ttt,u)\coloneqq\tfrac12\norm{Q y_\ttc-Qy_\ttt}{L^2(\bbR_+,{L^{2}(\Omega)}
)}^2+\tfrac12\norm{u}{L^2(\bbR_+,\bbR^{M_\sigma})}^2,
\end{equation}
for the solutions of~\eqref{sys-y-u-Cu} and~\eqref{sys-haty} with initial states~$y_{\ttt0}$ and~$y_0$ in $W^{1,2}(\Omega)$, where
\begin{equation}
Q\colon W^{2,2}(\Omega)\to L^2(\Omega),
\end{equation}
is a linear and continuous operator.
\end{subequations}

\begin{mainresult}\label{MR:mainOptim}
Let~$M$ be as in Main Result~\ref{MR:mainStab} and let~$Q=P_{\clE_{M_1}^\rmf}\colon L^2(\Omega)\to \clE_{M_1}^\rmf$, $M_1\in\bbN_+$, be the orthogonal projection in~$L^2(\Omega)$ onto the linear span~$\clE_{M_1}^\rmf$ of the first eigenfunctions of the Laplacian (under the considered Neumann boundary conditions). Then, for large enough~$M_1$, there are controls minimizing~$J_{y_{\ttc0}}^{y_{\ttt0},Q}$  in~\eqref{Jy0}.
\end{mainresult}

A more precise statement of this result is given in Theorem~\ref{T:existOptim}. The orthogonal projection in the statement of the result can be replaced by other operators as, for example, $Q=\Id$ or~$Q=(-\Delta+\Id)^\frac12$.

Further, we shall investigate the first-order optimality conditions associated with optimal controls.

%%%%%%%%%%%%%%%%%%%%%%
\subsection{On previous related works in literature}\label{sS:Intro-liter}
The literature is rich in results concerning the feedback stabilizability of parabolic like equations under no constraints in the magnitude of the control.
For example,  we can mention~\cite{BarbuTri04,BarRodShi11,KunRod19-cocv,PhanRod18-mcss,BreKunRod17,AzouaniTiti14,LunasinTiti17}, \cite[sect.~2.2]{Barbu11}, and references therein.
Though we do not address,
in the present manuscript, the case of boundary controls, we would like to refer the reader to~\cite{BarbuLasTri06,Barbu12,BadTakah11,Raymond19,Rod18,Barbu_TAC13,BalKrstic00,CochranVazquezKrstic06,KrsticMagnVazq09}.

Concerning literature considering an upper bound~$C_u$ for the magnitude
of the control~$u(t)$, for finite-dimensional systems we refer the reader to~\cite{HuLinQiu01,CorradiniCristofaroOrlando10,ZhouLam17,LiuChitourSontag96,Teel92,BarbuChengFreeman97,LauvdalFossen97,SaberiLinTeel96,SussmannSontagYang94,WredenhagenBelanger94}.
For infinite-dimensional systems, namely,  for parabolic equations we refer the reader to~\cite{MironchenkoPrieurWirth21}, and  for wave-like equations we refer the reader to \cite{LasieckaSeidman03,SeidmanLi01} and to~\cite[sect.~8.1]{Slemrod89} with an application to the  beam equation.

Recently, in~\cite{AzmiKunRod21-arx} it is shown that a saturated feedback as
\begin{equation}\label{feed-orthproj}
\widehat\clK_M^{\lambda}z\coloneqq\fkP^{\dnorm{\Bigcdot}{}}_{C_u}\left(-\lambda (U_{M}^\diamond)^{-1}P_{\clU_M}^{\clU_M^\perp}z\right)
\end{equation}
is able to stabilize system~\eqref{sys-y-u-Cu} in the pivot  $L^2(\Omega)$-norm. More precisely, it is shown that for large enough~$M$, $\lambda=\lambda(M)$, and~$C_u=C_u(M,\lambda)$, we have that
\begin{equation}\label{goal-diff-L2}
\norm{ z(t,\Bigcdot)}{L^{2}(\Omega)}^2\le \rme^{-\mu (t-s)}\norm{ z(s,\Bigcdot)}{L^{2}(\Omega)}^2,\quad\mbox{for all}\quad t\ge s\ge0.
\end{equation}
Hence,  Main Result~\ref{MR:mainStab} holds with the $L^2(\Omega)$-norm in place of the Sobolev~$W^{1,2}(\Omega)$-norm.

We shall show that~\eqref{goal-diff-L2} also holds with~$\widehat\clK_M^{\lambda}$ replaced by~$\overline\clK_M^{\lambda}$ in~\eqref{const-K}. With the later feedback,~$\lambda$ can be chosen independent of~$M$. This has advantages in applications, because we do not need  to be concerned with relations between the choices for~$M$ and~$\lambda$.

\medskip
Main Result~\ref{MR:mainStab} states that the stabilizability also holds in the Sobolev~$W^{1,2}(\Omega)$-norm. Note that, for general {\em nonlinear} systems it is not clear whether the stabilizability in~$L^{2}(\Omega)$-norm implies stabilizability in~$W^{1,2}(\Omega)$-norm.
We show the latter as a corollary of a smoothing-like property  of the error dynamics,  which reads as
\begin{equation}\label{diff-smoo}
\norm{ z(s+1,\Bigcdot)}{W^{1,2}(\Omega)}^2\le C\norm{ z(s,\Bigcdot)}{L^{2}(\Omega)}^2,\quad\mbox{for all}\quad  s\ge0.
\end{equation}
with a constant~$C$ independent of~$s$. To verify~\eqref{diff-smoo} we shall need to show appropriate bounds for the solution~$y_\ttt$ of the free dynamics, subject to  a general external forcing satisfying~\eqref{bound-h-intro}.
Subsequently, we will be able to show~\eqref{goal-diff}: for a suitable $\varrho>1$,
\begin{equation}\label{goal-diff-W12}
\norm{ z(t,\Bigcdot)}{W^{1,2}(\Omega)}^2\le \varrho\rme^{-\mu (t-s)}\norm{ z(s,\Bigcdot)}{W^{1,2}(\Omega)}^2,\quad\mbox{for all}\quad t\ge s\ge0.
\end{equation}
Note that comparing to~\eqref{goal-diff-L2}, here we have a ``transient bound'' constant~$\varrho>1$. That is, we could not show the inequality with~$\rho=1$. This means that, though $t\mapsto\norm{ z(t,\Bigcdot)}{L^{2}(\Omega)}^2$ is strictly decreasing, it may hold that~$t\mapsto\norm{ z(t,\Bigcdot)}{W^{1,2}(\Omega)}^2$ is not.

\medskip
Main Result~\ref{MR:mainOptim} states that the existence of optimal controls holds in case the linear mapping~$Q$ is an appropriate orthogonal projection. We are particularly interested in the case that~$Q$ has finite-dimensional range, as in Main Result~\ref{MR:mainOptim}. This can have computational advantages, for example, for nonautonomous linear free dynamics, the combination of the finite-dimensionality of the range of~$Q$ with the finite-dimensionality of the control can enhance the existence of low-rank approximations of the algebraic Riccati equation  associated with the optimal control; \cite{Penzl00}, \cite[sect.~4]{BennerLiPenzl08}.
In general,  for infinite-dimensional systems and nonlinear dynamics, the existence of optimal controls is a nontrivial question. Roughly speaking, following standard arguments, by taking a minimizing sequence of controls we should obtain a sequence of associated states bounded in a suitable norm allowing us to find a limit state solving the dynamics.  Indeed, the derivation of such bound (cf.~Lemma~\ref{L:key-exist-Id}) plays a key role in the proof of Main Result~\ref{MR:mainOptim}.

%%%%%%%%%%%%%%%%%%%%
\subsection{Contents} In section~\ref{S:Schloegl} we prove  Main Result~\ref{MR:mainStab} concerning the stabilizing property of the proposed saturated explicit feedback in~\eqref{const-K}. Section~\ref{S:optimalcontrol} is concerned with the existence of optimal controls stated in Main Result~\ref{MR:mainOptim}.  First order (Karush--Kuhn--Tucker) optimality conditions associated with our optimal control problem are addressed in section~\ref{S:1optcond}, which  are used in section~\ref{S:simul} to compute a suboptimal stabilizing receding horizon control,  whose  performance is compared with that of the saturated explicit feedback~\eqref{const-K}.

%%%%%%%%%%%%%%%%%%%%
\subsection{Notation}
We write~$\bbR$ and~$\bbN$ for the sets of real numbers and nonnegative
integers, respectively. We set $\bbR_+\coloneqq(0,+\infty)$, $\overline\bbR_+\coloneqq[0,+\infty)$, and~$\bbN_+\coloneqq\bbN\setminus\{0\}$.

Given Banach spaces~$X$ and~$Y$, we write $X\xhookrightarrow{} Y$ if the inclusion
$X\subseteq Y$ is continuous.
The space of continuous linear mappings from~$X$ into~$Y$ is denoted by~$\clL(X,Y)$. We
write~$\clL(X)\coloneqq\clL(X,X)$.
The continuous dual of~$X$ is denoted~$X'\coloneqq\clL(X,\bbR)$.
The adjoint of an operator $L\in\clL(X,Y)$ will be denoted $L^*\in\clL(Y',X')$.
The space of continuous functions from~$X$ into~$Y$ is denoted by~$\clC(X,Y)$.

We also denote
$W((0,T),X,Y)\coloneqq\{y\in L^2((0,T),X)\mid \dot y\in L^2((0,T),Y)\}$ and
$W_{\rm loc}(\bbR_+,X,Y)\coloneqq\{y\mid y\in W((0,T),X,Y)
\;\mbox{ for all }\; T>0\}$.

For simplicity, we shall often denote, for a fixed~$\Omega\subset\bbR^d$ the Hilbert spaces
\[
H\coloneqq L^2\coloneqq L^2(\Omega)\quad\mbox{and}\quad
W^{1,2}\coloneqq W^{1,2}(\Omega)=\{y\in L^2\mid \nabla y \in (L^2)^d\}
\]
endowed with their usual scalar products, and
\begin{equation}\label{V}
V\coloneqq W^{1,2}(\Omega)\quad\mbox{with  scalar product}\quad(w,z)_V\coloneqq \nu(\nabla w,\nabla z)_{H^d}+(w,z)_{H},
\end{equation}
 where~$\nu$ is the diffusion parameter in~\eqref{sys-y-u}, whose associated norm is equivalent to the usual~$W^{1,2}$-norm, due to
\[
\min\{\nu,1\}\norm{y}{W^{1,2}}^2\le\norm{y}{V}^2\le\max\{\nu,1\}\norm{y}{W^{1,2}}^2.
\]

As usual we take~$H=L^2$ as a pivot space, that is, we identify it with its continuous dual,~$H'=H$. We consider the (symmetric, rescaled, and shifted) Neumann Laplacian
\begin{equation}\label{A}
A\in\clL(V,V'),\quad \langle Aw,z\rangle_{V',V}\coloneqq(w,z)_V,
\end{equation}
with domain given by
\begin{equation}\label{DA}
\rmD(A)\coloneqq \{w\in H\mid Ay\in H\}=\left.\left\{z\in W^{2,2}(\Omega)\right|\;\; \tfrac{\p}{\p\bfn}z\rest{\p\Omega}=0\right\},
\end{equation}
which we shall assume to be endowed with the scalar product
\[
(w,z)_{\rmD(A)}\coloneqq (Aw,Az)_{H}.
\]
At some points in the text we will also need to refer to the continuous dual~$\rmD(A)'\eqqcolon D(A^{-1})$ of~$\rmD(A)$. In general, we can define the fractional power~$A^s$ of~$A$, for~$s\ge0$, and we write
$\rmD(A^{-s})\coloneqq\rmD(A^{s})'$.

We will also write, for
the more general Lebesgue and Sobolev spaces,
\[
L^p\coloneqq L^p(\Omega),\quad W^{s,p}=W^{s,p}(\Omega),\qquad p\ge1,\quad s\ge0.
\]

By
$\overline C_{\left[a_1,\dots,a_n\right]}$ we denote a nonnegative function that
increases in each of its nonnegative arguments~$a_i$, $1\le i\le n$.

Finally, $C,\,C_i$, $i=0,\,1,\,\dots$, stand for unessential positive constants.

%%%%%%%%%%%%%%%%%%%%%%%%%%%
%%%%%%%%%%%%%%%%%%%%%%%%%%%
 \section{Exponential stabilizability of the Schl\"ogl equation}\label{S:Schloegl}
The results are derived under the following assumption on the  external force.
\begin{assumption}\label{A:h-pers-bound}
The external force $h$ in~\eqref{sys-y-u} is in~$L^6_{\rm loc}(\bbR_+,L^2(\Omega))$ and there exist constants~$\tau_h>0$ and~$C_h\ge0$ such that it  satisfies the persistent bound as
\[
\sup_{s\ge0}\norm{h}{L^6((s,s+\tau_h),L^2(\Omega))}\le C_h.
\]
\end{assumption}
We show that a saturated control allows us to track arbitrary solutions/trajectories~$ y_\ttt$ to the free-dynamics.
Our goal is to construct a control which stabilizes the system to a given~$y_\ttt$. Recalling~\eqref{A}, we write the controlled system~\eqref{sys-y-u} as
\begin{subequations}\label{Schloegl-feed}
\begin{align}
 &\tfrac{\p}{\p t} y_\ttc +Ay_\ttc- y_\ttc+(y_\ttc- \zeta_1)(y_\ttc- \zeta_2)(y_\ttc- \zeta_3)= h+U_M^\diamond\overline\clK_M^\lambda(y_\ttc-y_\ttt),\\
&y_\ttc(0)=y_{\ttc0}\in V,
\intertext{with the saturated feedback input control operator (cf.~\eqref{sys-z-intro})}
 &\overline\clK_M^\lambda(w)\coloneqq  \fkP^{\dnorm{\Bigcdot}{}}_{C_u}\left(-\lambda (U_{M}^\diamond)^{-1}P_{\clU_M}^{\widetilde \clU_M^\perp}A P_{\widetilde \clU_M}^{\clU_M^\perp} w\right).\label{SatKw}
 \end{align}
\end{subequations}

We denote the solution~$y_\ttt$ of the free dynamics~\eqref{sys-haty} and the solution~$y(t)$ of~\eqref{Schloegl-feed} by
\[
\clS(y_{\ttt0};t)\coloneqq y_\ttt(t)\qquad\mbox{and}\qquad \clS_{\rm feed}^{y_\ttt}(y_{\ttc0};t)\coloneqq y_\ttc(t),
\]
with initial states~$y_{\ttt0}\in W^{1,2}(\Omega)$ and~$y_{\ttc0}\in W^{1,2}(\Omega)$.
We also denote
\begin{equation}\label{zetainf}
\norm{\zeta}{\infty}\coloneqq\max\limits_{1\le i\le 3}\norm{\zeta_i}{\bbR}.
\end{equation}

Recall that~$H=L^2(\Omega)$ and~$V=W^{1,2}(\Omega)$.
\begin{theorem}\label{T:mainSchloegl}
For arbitrary~$\mu>0$, there exist $M_*\in\bbN_+$ and~$\lambda_*>0$ such that, for every~$M\ge M_*\in\bbN_+$ and~$\lambda>\lambda_*$ there exists~$C_u^*=C_u^*(M,\lambda)\in\bbR_+$  such that, for all~$C_u>C_u^*$ it holds that: for each~$(y_{\ttt0},y_0)\in V\times V$, the solutions $y_\ttt(t)\coloneqq\clS(y_{\ttt0};t)$ of~\eqref{sys-haty} and $y_\ttc(t)\coloneqq\clS_{\rm feed}^{y_\ttt}(y_{\ttc0};t)$ of~\eqref{Schloegl-feed} satisfy
\begin{align}
\norm{y_\ttc(t)-y_\ttt(t)}{H}&\le \rme^{-\mu (t-s)}\norm{y_\ttc(s)-y_\ttt(s)}{H},
&&\mbox{ for all}\quad t\ge s\ge0;\label{Tmain.expL2}
\end{align}
Furthermore, $M_*\le\ovlineC{\mu,\norm{\zeta}{\infty}}$, $\lambda_*\le\ovlineC{\mu,\norm{\zeta}{\infty}}$.
\end{theorem}

The proof of Theorem~\ref{T:mainSchloegl} is given in section~\ref{sS:proofT:mainSchloegl}.
\begin{theorem}\label{T:mainSchloegl-st}
For arbitrary~$\mu>0$, let $M_*\in\bbN_+$, $\lambda_*>0$, and~$C_u^*\in\bbR_+$ be as in Theorem~\ref{T:mainSchloegl}. If Assumption~\eqref{A:h-pers-bound} holds true, then for each~$(y_{\ttt0},y_0)\in V\times V$, the solutions $y_\ttt(t)\coloneqq\clS(y_{\ttt0};t)$ of~\eqref{sys-haty} and $y_\ttc(t)\coloneqq\clS_{\rm feed}^{y_\ttt}(y_{\ttc0};t)$ of~\eqref{Schloegl-feed} satisfy
\begin{align}
\norm{y_\ttc(t)-y_\ttt(t)}{V}&\le \varrho\rme^{-\mu (t-s)}\norm{y_\ttc(s)-y_\ttt(s)}{V},
&&\mbox{for all}\quad t\ge s\ge0.\label{Tmain.expH1}
\end{align}
Furthermore, $\varrho\le\ovlineC{\mu,\norm{\zeta}{\infty}, \norm{y_{\ttt0}}{L^6},\norm{y_{\ttt0}-y_{\ttc0}}{L^6},\tau_h^{-1},C_h,\norm{U_M^\diamond\overline\clK_M^\lambda}{\clL(H)}}$.
\end{theorem}

The proof of Theorem~\ref{T:mainSchloegl-st} is given in
section~\ref{sS:proofT:mainSchloegl-st}.
Observe that Theorem~\ref{T:mainSchloegl-st} states that every trajectory~$y_\ttt$ of the free-dynamics system~\eqref{sys-haty} can be tracked exponentially fast, in the $V$-norm.
Observe also that~$M\ge M_*$ does not depend on the pair~$(y_{\ttt0},y_0)$ of initial states, which means that the number~$M_\sigma$ of actuators can be chosen independently of both the targeted trajectory~$y_\ttt$ and of the initial error~$y_{\ttc0}-y_{\ttt0}$.

%%%%%%%%%%%%%%%%%%%%%%%%%%%
\subsection{The dynamics of the error}
We consider the dynamics of the error, that is, of the difference~$z\coloneqq y_\ttc-y_\ttt$ between the solution
$y_\ttc(t)$ of system~\eqref{Schloegl-feed} and the targeted solution~$y_\ttt(t)$ of the free dynamics~\eqref{sys-haty}. We find
\begin{subequations}\label{Schloegl-diff}
\begin{align}
 &\tfrac{\p}{\p t} z +Az-z+f(y_\ttc)-f(y_\ttt)= U_M^\diamond\overline\clK_M(z),\qquad z(0)=z_0,
\intertext{with}
&z_0\coloneqq y_{\ttc0}-y_{\ttt0}\quad\mbox{and}\quad
f(w)\coloneqq (w- \zeta_1)(w- \zeta_2)(w- \zeta_3).
\end{align}
\end{subequations}
We start by observing that
\begin{align}
&f(w)=w^3 +\xi_2 w^2+\xi_1 w+\xi_0\notag
\end{align}
with
\begin{equation}\label{xis}
\xi_2\coloneqq -(\zeta_1+ \zeta_2+ \zeta_3),\quad\xi_1\coloneqq  \zeta_1 \zeta_2+ \zeta_1 \zeta_3+ \zeta_2 \zeta_3,\quad
\xi_0\coloneqq  -\zeta_1 \zeta_2 \zeta_3,
\end{equation}
which leads us to
\begin{align}
f(y_\ttc)=f(z+y_\ttt)&=z^3 +3y_\ttt z^2 +3y_\ttt^2 z+y_\ttt^3
+\xi_2( z^2+2y_\ttt z +y_\ttt^2) + \xi_1(z+y_\ttt)+\xi_0\notag\\
&=z^3 +(3y_\ttt+\xi_2) z^2+(3y_\ttt^2+2\xi_2y_\ttt+\xi_1) z+f(y_\ttt),\notag
\end{align}
and, recalling~\eqref{A} and~\eqref{DA}, we arrive at the evolutionary error dynamics
\begin{subequations}\label{sys-z-feedf-hat}
\begin{align}
 &\dot z =-Az -f^{y_\ttt}(z)+U_M^\diamond\overline\clK_M(z),\qquad z(0)=z_0,\\
\intertext{with~$f^{y_\ttt}(z)\coloneqq f(y_\ttc)- f(y_\ttt)+z$, that is,}
&f^{y_\ttt}(z)\coloneqq z^3+(3y_\ttt+\xi_2) z^2+(3y_\ttt^2+2\xi_2y_\ttt+\xi_1-1) z.
 \end{align}
\end{subequations}
\begin{remark}
For every initial state~$y_{\ttc0}\in V$, there exists a unique solution
$y_\ttc\in W_{\rm loc}(\bbR_+,\rmD(A),H)$ for system~\eqref{Schloegl-feed}. It can be derived as a weak limit of Galerkin approximations~$y^N(t)\in\clE_N^\rmf$  in  the span $\clE_N^\rmf=\linspan\{e_i\mid 1\le i\le N\}$ of the first eigenfunctions of~$A$. We refer to the comments in~\cite[sect.~2.1]{AzmiKunRod21-arx} where we find estimates, which also hold for such approximations.
\end{remark}

%%%%%%%%%%%%%%%%%%%%%%%%%%%
\subsection{Auxiliary results for stabilizability in $H$-norm}\label{sS:auxiliary-L2}
We gather auxiliary results we shall use to show the stabilizability in the norm of the pivot space~$H=L^2(\Omega)$.
 Hereafter, $\clU_M$, $\widetilde\clU_M$, $A$,  and~$\overline\clK_M^\lambda$ are as in \eqref{U_M}, \eqref{tildeU_M}, \eqref{A}, and~\eqref{SatKw}.
\begin{lemma}\label{L:poincare}
The sequence~$(\xi_{M})_{M\in\bbN_+}$ of Poincar\'e-like constants
\[
\xi_{M}\coloneqq\inf_{h\in (V\bigcap\clU_M^\perp)\setminus\{0\}}\tfrac{\norm{h}{V}}{\norm{h}{H}}
\]
satisfies $\lim\limits_{M\to+\infty}\xi_{M}=+\infty$.
 \end{lemma}
The proof of Lemma~\ref{L:poincare} can be found in~\cite[Sect.~5]{Rod21-sicon}.
\begin{lemma}\label{L:MlamPoinc}
For every~$\zeta>0$ we can find~$M_*\in\bbN_+$ and~$\lambda_*>0$ large enough such that
\[
\norm{y}{V}^2+\lambda \norm{P_{\widetilde\clU_{M}}^{\clU_{M}^\perp}y}{V}^2
\ge\zeta \norm{y}{H}^2, \qquad\mbox{for all}\quad M\ge M_*,\quad\lambda \ge\lambda_*,\quad\mbox{and}\quad
y\in V.
\]
Furthermore~$M_*=\ovlineC{\zeta}$ and~$\lambda_*=\ovlineC{\zeta}$.
 \end{lemma}
\begin{proof}
We follow the arguments in~\cite[Lemma~3.5]{KunRodWalter21}.
We write,
\[
y=\varTheta+\vartheta,\quad\mbox{with}\quad\vartheta\coloneqq P_{\widetilde\clU_{M}}^{\clU_{M}^\perp}y
 \quad\mbox{and}\quad\varTheta\coloneqq P_{\clU_{M}^\perp}^{\widetilde\clU_{M}}y.
\]
We find that, for each~$y\in V$, we have~$\vartheta\in\rmD(A)\subset V$, $\varTheta\in V$, and
\begin{align*}
\norm{y}{V}^2+\lambda \norm{P_{\widetilde\clU_{M}}^{\clU_{M}^\perp}y}{V}^2
&=\norm{\varTheta+\vartheta}{V}^2+\lambda \norm{\vartheta}{V}^2
=\norm{\varTheta}{V}^2+2(\varTheta,\vartheta)_V+\norm{\vartheta}{V}^2+\lambda \norm{\vartheta}{V}^2\\
&\ge\tfrac{1}{2}\norm{\varTheta}{V}^2-\norm{\vartheta}{V}^2+\lambda \norm{\vartheta}{V}^2
\ge\tfrac{1}{2}\norm{\varTheta}{V}^2+(\lambda-1)\norm{\vartheta}{V}^2,
\end{align*}
With~$\xi_{M}$ as in Lemma~\ref{L:poincare} we denote
\[
\underline\xi_{M}\coloneqq\min_{N\ge M}\xi_{N}.
\]
Hence, for given~$\zeta>0$, by choosing
\[
M_*\coloneqq\min\{M\in\bbN_+\mid\underline\xi_{M}\ge4\zeta\} \quad\mbox{and}\quad\lambda_*\ge 2\zeta+1,
\]
we arrive at
\begin{align*}
\norm{y}{V}^2+\lambda\norm{P_{\widetilde\clU_{M}}^{\clU_{M}^\perp}y}{V}^2
&\ge2\zeta\left(\norm{\varTheta}{H}^2+ \norm{\vartheta}{H}^2\right)
\ge \zeta\norm{\varTheta+\vartheta}{H}^2=\zeta\norm{y}{H}^2,
\end{align*}
for all~$M\ge M_*$ and all~$\lambda\ge\lambda_*$.
\end{proof}

\begin{lemma}\label{L:Ku-monot}
The constrained feedback operator satisfies
\begin{align}
\left(U_M^\diamond\overline\clK_M^\lambda(z(t)),z(t)\right)_{H}&=
\begin{cases}
-\lambda\min\left\{1,\tfrac{C_u}{\dnorm{v(t)}{}}\right\}\norm{P_{\widetilde \clU_M}^{\clU_M^\perp}z(t)}{V}^2
&\quad\mbox{if}\quad P_{\clU_M}z(t)\ne0,\\
0,&\quad\mbox{if}\quad P_{\clU_M}z(t)=0,
\end{cases}\notag
\end{align}
where~$v(t)\coloneqq-\lambda(U_M^\diamond)^{-1}P_{\clU_M}^{\widetilde \clU_M^\perp}A P_{\widetilde \clU_M}^{\clU_M^\perp}z(t)$.
\end{lemma}
\begin{proof}
Recalling~\eqref{rad.proj.min} and the feedback in~\eqref{Schloegl-feed}, 
with~$v(t)$ as in Lemma~\ref{L:Ku-monot}, we find
\begin{align}\notag
\left(U_M^\diamond\overline\clK_M^\lambda(z(t)),z(t)\right)_{H}&=\left(U_M^\diamond\fkP^{\dnorm{\Bigcdot}{}}_{C_u}(v(t)),z(t)\right)_{H}\notag\\
&=\min\left\{1,\tfrac{C_u}{\dnorm{v(t)}{}}\right\}(U_M^\diamond v(t),z(t))_{H},\quad\mbox{if}\quad v(t)\ne0.\notag
\end{align}
By~\cite[Lem.~3.4]{KunRodWalter21} we have the adjoint identity~$(P_{\clU_M}^{\widetilde \clU_M^\perp})^*=P_{\widetilde \clU_M}^{\clU_M^\perp}$, which leads us to
\begin{align}
(U_M^\diamond v(t),z(t))_{H}=-\lambda( P_{\clU_M}^{\widetilde \clU_M^\perp}A P_{\widetilde \clU_M}^{\clU_M^\perp}z(t),z(t))_{H}=-\lambda( AP_{\widetilde \clU_M}^{\clU_M^\perp} z(t),P_{\widetilde \clU_M}^{\clU_M^\perp}z(t))_{H},\notag
\end{align}
which finishes the proof, because~$( AP_{\widetilde \clU_M}^{\clU_M^\perp} z(t),P_{\widetilde \clU_M}^{\clU_M^\perp}z(t))_{H}=( P_{\widetilde \clU_M}^{\clU_M^\perp} z(t),P_{\widetilde \clU_M}^{\clU_M^\perp}z(t))_{V}$.
\end{proof}

\begin{lemma} \label{L:expdec.largez0}
For every~$\mu>0$, there exists a constant~$D\ge1$ such that for all~$(z_0,y_{\ttt0})\in V\times V$ the solution of system~\eqref{sys-z-feedf-hat}
satisfies
\begin{align}
 &\tfrac{\rmd}{\rmd t}\norm{z(t)}{H}\le-\mu\norm{z(t)}{H}\quad\mbox{if}\quad \norm{z(t)}{H}\ge D,\notag
\intertext{and}
&\norm{z(t)}{H}\le D\quad\mbox{for all}\quad t\ge (\mu^2D)^{-\frac1{2}}.\notag
\end{align}
Moreover, if for some~$s\ge0$ we have that~$\norm{z(s)}{H}\le D$,  then~$\norm{z(t)}{H}\le D$ for all~$t\ge s$.
Furthermore
$D\le\ovlineC{\mu,\norm{\zeta}{\infty}}$ is independent of~$(z_0,y_{\ttt0}, C_u)$.
\end{lemma}

The proof of Lemma~\ref{L:expdec.largez0} can be done by following the steps as in~\cite[Proof of Lem.~2.4]{AzmiKunRod21-arx}, and uses uses the structure of the nonlinearity~$f^{y_\ttt}$ in~\eqref{sys-z-feedf-hat} in an essential manner. It is also important to observe that, though our feedback is different from the one in~\cite{AzmiKunRod21-arx}, the essential property at this point is that, as in~\cite{AzmiKunRod21-arx}, it is monotone, as in Lemma~\ref{L:Ku-monot}.

%%%%%%%%%%%%%%%%%%%%%%%%%%%
\subsection{Proof of Theorem~\ref{T:mainSchloegl}}\label{sS:proofT:mainSchloegl}
Following the arguments in~\cite{AzmiKunRod21-arx},  we can conclude that
\begin{align}
 \tfrac{\rmd}{\rmd t} \norm{z}{H}^2&\le-\norm{z}{V} ^2+C \norm{z}{H}^2 +2(U_M^\diamond\overline\clK_M^\lambda(z),z)_{H},\notag
\end{align}
with~$C=\tfrac{128}{15}\norm{\xi_2}{}^2+\tfrac{50}{11}\xi_2^2-2\xi_1+2\le\ovlineC{\norm{\zeta}{\infty}}$; see~\cite[Sect,~2.5]{AzmiKunRod21-arx}.

We define also the constants
\begin{subequations}\label{varpi}
\begin{align}
C_0&\coloneqq \tfrac23\xi_2^2-2\xi_1=\max\{-6r^2 -4\xi_2r-2\xi_1\mid r\in\bbR\},\\ \widehat C_0&\coloneqq\max\{4+C_0,C\},\quad\mbox{and}\quad
\varpi\coloneqq 2\mu+\widehat C_0.
\end{align}
\end{subequations}

We would like to remark that in this proof we can take $\widehat C_0=C$. We introduce~$C_0$,  because it will be convenient,  later on, that this Theorem holds for the larger~$\widehat C_0$ in~\eqref{varpi}; see ~\eqref{dtw-bull4} following from~\eqref{dtw-bull-latime} within the proof of Lemma~\ref{L:ZoweKurcyusz-0inint}.

Let~$M_*=\ovlineC{\varpi}\le\ovlineC{\mu,\norm{\zeta}{\infty}}\in\bbN_+$ and~$\lambda_*>0$ be given by Lemma~\ref{L:MlamPoinc}.
Then
\begin{equation}\label{lamMvarpi}
\norm{z}{V}^2+2\lambda\norm{P_{\widetilde \clU_M}^{\clU_M^\perp}z}{V}^2\ge \varpi\norm{z}{H}^2,\quad\mbox{if}\quad  M\ge M_* \quad\mbox{and}\quad \lambda\ge\lambda_*,
\end{equation}
which leads us to
\begin{align}
 \tfrac{\rmd}{\rmd t} \norm{z}{H}^2
&\le 2(U_M^\diamond\overline\clK_M^\lambda(z),z)_{H}+2\lambda\norm{P_{\widetilde \clU_M}^{\clU_M^\perp}z}{V}^2
-2\mu\norm{z}{H}^2.\label{dtz+2}
\end{align}

Observe that, by Lemma~\ref{L:Ku-monot}, we have that
\begin{subequations}\label{inactCu1}
\begin{align}
\left(U_M^\diamond\overline\clK_M(z(t)),z(t)\right)_{H}=
-\lambda\norm{P_{\widetilde \clU_M}^{\clU_M^\perp}z}{V}^2
&\quad\Longleftarrow\quad \dnorm{-\lambda(U_M^\diamond)^{-1}P_{\clU_M}^{\widetilde \clU_M^\perp}A P_{\widetilde \clU_M}^{\clU_M^\perp}z}{}\le C_u\notag\\
&\quad\Longleftarrow\quad  C_u\ge \lambda\dnorm{\fkU_M}{}\norm{z}{H}
\intertext{where}
\hspace{-12em}\fkU_M\coloneqq (U_M^\diamond)^{-1}P_{\clU_M}^{\widetilde \clU_M^\perp}A P_{\widetilde \clU_M}^{\clU_M^\perp}&\quad\mbox{and}\quad
\dnorm{\fkU_M}{}\coloneqq\max\limits_{w\in L^2\setminus\{0\}}\frac{\dnorm{\fkU_Mw}{}}{\norm{w}{H}}.
\end{align}
\end{subequations}

Recall that, from Lemma~\ref{L:expdec.largez0}, there exists a constant~$D\ge 1$ such that for
\begin{subequations}\label{zhaty.tauD-lt}
\begin{align}
&t_\rma\coloneqq\min\{t\ge0\mid\norm{z(t)}{H}\le D\}\le(\mu^2D)^{-\frac12},\\
\intertext{we have that}
&\norm{z(t)}{H}\le \rme^{-\mu(t-s)}\norm{z(s)}{H},\quad\mbox{for all}\quad 0\le s\le t\le t_\rma,\\
&\norm{z(t)}{H}\le D,\quad\mbox{for all}\quad t\ge t_\rma.
\end{align}
\end{subequations}

Now, motivated by~\eqref{inactCu1} we set
\begin{equation}\label{Cu*}
C_u^*\coloneqq  \lambda\dnorm{\fkU_M}{}D.
\end{equation}

We can conclude that
\begin{align}
\dnorm{-\lambda(U_M^\diamond)^{-1}P_{\clU_M}^{\widetilde \clU_M^\perp}A P_{\widetilde \clU_M}^{\clU_M^\perp}z(t)}{}\le C_u,\qquad\mbox{if}\quad C_u\ge C_u^*
\quad\mbox{and}\quad t\ge t_\rma,\notag
\end{align}
and, by~\eqref{inactCu1},
\begin{align}
(U_M^\diamond\overline\clK(z(t)),z(t))_{H}&=-\lambda\norm{P_{\widetilde \clU_M}^{\clU_M^\perp}z(t)}{V}^2
,\qquad\mbox{if}\quad C_u\ge C_u^*
\quad\mbox{and}\quad t\ge t_\rma.\notag
\end{align}

Therefore, by~\eqref{dtz+2},
\begin{align}
 \tfrac{\rmd}{\rmd t} \norm{z}{H}^2
&\le-2\mu\norm{z}{H}^2,\quad\mbox{for all}\quad t\ge t_\rma, \quad\mbox{if}\quad C_u\ge C_u^*,\notag
\end{align}
which implies that
\begin{align}
 \norm{z(t)}{H}
&\le\rme^{-\mu(t-s)}\norm{z(s)}{H},\quad\mbox{for all}\quad t\ge s \ge t_\rma, \quad\mbox{if}\quad C_u\ge C_u^*.\label{z1+}
\end{align}

Combining~\eqref{zhaty.tauD-lt} and~\eqref{z1+} we find
\begin{align}
&\norm{z(t)}{H}\le\rme^{-\mu(t-t_\rma)} \rme^{-\mu(t_\rma-s)}\norm{z(s)}{H}=\rme^{-\mu(t-s)} \norm{z(s)}{H},\notag\\ &\hspace{5em}\quad\mbox{for all}\quad  t\ge t_\rma\ge s\ge0,\quad\mbox{if}\quad C_u\ge C_u^*,
\end{align}
which finishes the proof.
\qed

\begin{remark}\label{R:Cu-M}
Note that~$C_u^*$ as in~\eqref{Cu*} depends on the pair~$(\lambda,M)$.
\end{remark}

%%%%%%%%%%%%%%%%%%%%%%%%%%%
\subsection{Auxiliary results for stabilizability in $V$-norm}\label{sS:auxiliary-W12}
We gather auxiliary results which we shall use to show the stabilizability in the norm of the space $V=W^{1,2}(\Omega)$ of initial states.
We start with an estimate, for the free dynamics, in the Lebesgue space~$L^6(\Omega)$; recall that~$W^{1,2}(\Omega)\subset L^6(\Omega)$, since $\Omega\subset\bbR^d$ with~$d\in\{1,2,3\}$.
\begin{lemma}\label{L:estL6free}
If Assumption~\ref{A:h-pers-bound} holds for the external force~$h$, then the solution of the free dynamics~\eqref{sys-haty} satisfies, for all~$t\ge s\ge0$,
\begin{align}
&\norm{y_\ttt(t)}{L^6}^6
\le\rme^{-(t-s)}\norm{y_\ttt(s)}{L^6}^6+\breve C_2\left(\norm{h}{L^6((s,t),H)}^6+1\right), 
\notag\\
&\norm{y_\ttt(t)}{L^6}^6
\le\rme^{-(t-s)}\norm{y_\ttt(s)}{L^6}^6+\tfrac{2-\rme^{-\tau_h}}{1-\rme^{-\tau_h}}\breve C_2\left(C_h^6+1\right),
\notag
\end{align}
with
$\breve C_2
\le\ovlineC{\norm{\zeta}{\infty}}$, where~$\tau_h>0$ and~$C_h\ge0$ are as in Assumption~\ref{A:h-pers-bound}.
\end{lemma}
\begin{proof}
Multiplying the free dynamics~\eqref{sys-haty} by~$3y_\ttt^5$, we obtain
\begin{align}
\tfrac{\rmd}{\rmd t}\norm{y_\ttt^3}{H}^2&\le3\nu( \Delta y_\ttt,y_\ttt^5)_{H}
-3( (y_\ttt-\zeta_1)(y_\ttt-\zeta_2)(y_\ttt-\zeta_3)+h,y_\ttt^5)_{H}\notag\\
&=-\tfrac53\nu( 3y_\ttt^2\nabla y_\ttt,3y_\ttt^2\nabla y_\ttt)_{H}
-3( h+\xi_0,y_\ttt^5)_{H}
+3( -y_\ttt^3+\xi_2 y_\ttt^2+\xi_1y_\ttt,y_\ttt^5)_{H}\notag\\
&\le-\tfrac53\norm{y_\ttt^3}{V}^2+\tfrac53\norm{y_\ttt^3}{H}^2+3\norm{h+\xi_0}{H}\norm{y_\ttt^5}{H}
+3(-y_\ttt^3+\xi_2 y_\ttt^2+\xi_1y_\ttt ,y_\ttt^5)_{H}\notag\\
&=-\tfrac53\norm{y_\ttt^3}{V}^2+3\norm{h+\xi_0}{H}\norm{y_\ttt^3}{L^\frac{10}{3}}^\frac{10}{6}
+3(-y_\ttt^3+\xi_2 y_\ttt^2+(\xi_1+\tfrac53)y_\ttt ,y_\ttt^5)_{H}.\notag
\end{align}
Recall that for~$\Omega\subset\bbR^d$ the Sobolev embedding, see~\cite[Thm.~4.57]{DemengelDem12},
\begin{align}\label{sobol-rel}
(0<2s<d\quad\mbox{and}\quad 1\le q\le\tfrac{2d}{d-2s})\quad\Longrightarrow\quad W^{s,2}(\Omega)\xhookrightarrow{}L^q(\Omega)
\end{align}
gives us $W^{\frac{3}5,2}\xhookrightarrow{}W^{\frac{d}5,2}\xhookrightarrow{}L^\frac{10}{3}$, for $d\in\{1,2,3\}$. Then, by suitable interpolation and Young inequalities it follows that
\begin{align}
3\norm{h+\xi_0}{H}\norm{y_\ttt^3}{L^\frac{10}{3}}^\frac{10}{6}
&\le3C_0\norm{h+\xi_0}{H}\norm{y_\ttt^3}{H}^{\frac{10}{6}\frac{2}{5}}\norm{y_\ttt^3}{W^{1,2}}^{\frac{10}{6}\frac{3}{5}}\le3C\norm{h+\xi_0}{H}\norm{y_\ttt^3}{H}^{\frac{2}{3}}\norm{y_\ttt^3}{V}\notag\\
&\le\gamma^{-1}9C^2\norm{h+\xi_0}{H}^2\norm{y_\ttt^3}{H}^{\frac{4}{3}}+\gamma\norm{y_\ttt^3}{W^{1,2}}^{2},\quad\mbox{for all}\quad\gamma>0.\notag
\end{align}
Choosing~$\gamma=\frac13$ we obtain
\begin{align}
\tfrac{\rmd}{\rmd t}\norm{y_\ttt^3}{H}^2
&=-\tfrac43\norm{y_\ttt^3}{V}^2+27C^2\norm{h+\xi_0}{H}^2\norm{y_\ttt^3}{H}^{\frac{4}{3}}
+3(-y_\ttt^3+\xi_2 y_\ttt^2+(\xi_1+\tfrac53)y_\ttt ,y_\ttt^5)_{H}.\notag\\
&=-\tfrac43\norm{y_\ttt^3}{V}^2+27C^2\norm{h+\xi_0}{H}^2\norm{y_\ttt^3}{H}^{\frac{4}{3}}
+3C_1\dnorm{\Omega}{}.\notag
\end{align}
where
\[
C_1\coloneqq\max\{-s^8+\xi_2 s^7+(\xi_1+\tfrac{5}{3}) s^6\mid s\in\bbR\}.
\]

Writing~$-s^8+\xi_2 s^7+(\xi_1+\tfrac{5}{3}) s^6= s^6(-s^2+\xi_2 s+\xi_1+\tfrac{5}{3})$, we observe that \begin{align}
&\mbox{either}\quad C_1\le0\quad\mbox{or}\quad C_2\coloneqq\max\{-s^2+\xi_2 s+\xi_1+\tfrac{5}{3}\mid s\in\bbR\}>0.\notag
\end{align}
Further, we have the implications
$C_2>0 \,\Longrightarrow \, \xi_2^2+4(\xi_1+\tfrac{5}{3})>0$
and
\[
-s^2+\xi_2 s+\xi_1+\tfrac{5}{3}>0\quad\Longrightarrow\quad -C_3<s<C_3\coloneqq\tfrac{\norm{\xi_2}{\bbR}+\sqrt{\xi_2^2+4(\xi_1+\tfrac{5}{3})}}{2}\le
\ovlineC{\norm{\zeta}{\infty}}.
\]

Since, $C_2=\tfrac{\xi_2^2}{4}+\xi_1+\tfrac{5}{3}$ we arrive at~$C_1\le C_3^6C_2\le\ovlineC{\norm{\zeta}{\infty}}$.

Using again the Young inequality,
\begin{align}
\tfrac{\rmd}{\rmd t}\norm{y_\ttt^3}{H}^2
&=-\tfrac43\norm{y_\ttt^3}{V}^2+(\gamma_0^{-1}27C^2)^3\norm{h+\xi_0}{H}^6+\gamma_0^{\frac{3}{2}}\norm{y_\ttt^3}{H}^{2}+3C_1\dnorm{\Omega}{},\quad\mbox{for all}\quad\gamma_0>0.\notag
\end{align}
Choosing 
$\gamma_0=(\frac13)^{\frac{2}{3}}$ and recalling that~$\norm{y_\ttt^3}{H}\le\norm{y_\ttt^3}{V}$
and~$(a+b)^n\le 2^{n-1}(a^n+b^n)$ for~$n\ge1$ (see~\cite[Prop.~2.6]{PhanRod17}), we find
\begin{align}
\tfrac{\rmd}{\rmd t}\norm{y_\ttt^3}{H}^2
&\le-\norm{y_\ttt^3}{V}^2+3^{11}C^6\norm{h+\xi_0}{H}^6+3C_1\dnorm{\Omega}{}\notag\\
&\le-\norm{y_\ttt^3}{V}^2+2^53^{11}C^6\norm{h}{H}^6+2^53^{11}C^6\norm{\xi_0}{H}^6+3C_1\dnorm{\Omega}{}\notag\\
&\le-\norm{y_\ttt^3}{H}^2+2^53^{11}C^6\norm{h}{H}^6+2^53^{11}C^6\norm{\xi_0}{H}^6+3C_1\dnorm{\Omega}{},\label{freedty-L6}
\end{align}
and
\begin{subequations}\label{freey-L6}
\begin{align}
&\norm{y_\ttt(t)}{L^6}^6
\le\rme^{-(t-s)}\norm{y_\ttt(s)}{L^6}^6+\breve C_2\left(\norm{h}{L^6((s,t),H)}^6+1\right),\\
\intertext{for all~$t\ge s\ge0$, with}
&\breve C_2\coloneqq 2^53^{11}C^6(1+\norm{\xi_0}{H}^6)+3C_1\dnorm{\Omega}{}
\le\ovlineC{\norm{\zeta}{\infty}}.
\end{align}
\end{subequations}
This gives us the first claim.

From~\eqref{freey-L6}, in particular, for the integer~$\lfloor r\rfloor$ defined by~$\lfloor r\rfloor\le r< \lfloor r\rfloor+1$,
with~$\tau_h$ and~$C_h$ as in Assumption~\ref{A:h-pers-bound}, we find that for~$t\ge s$,
\begin{align}
&\norm{y_\ttt(t)}{L^6}^6
\le\rme^{-(t-s-\tau_h\lfloor \tfrac{t-s}{\tau_h}\rfloor)}\norm{y_\ttt(s+\tau_h\lfloor \tfrac{t-s}{\tau_h}\rfloor)}{L^6}^6+\breve C_2\left(\norm{h}{L^6((s+\tau_h\lfloor \tfrac{t-s}{\tau_h}\rfloor,t),H)}^6+1\right)\notag\\
&\le\rme^{-(t-s-\tau_h\lfloor \tfrac{t-s}{\tau_h}\rfloor)}\norm{y_\ttt(s+\tau_h\lfloor \tfrac{t-s}{\tau_h}\rfloor)}{L^6}^6+\breve C_2\left(\norm{h}{L^6((s+\tau_h\lfloor \tfrac{t-s}{\tau_h}\rfloor,s+\tau_h\lfloor \tfrac{t-s}{\tau_h}\rfloor+\tau_h),H)}^6+1\right)
\notag\\
&\le\rme^{-(t-s-\tau_h\lfloor \tfrac{t-s}{\tau_h}\rfloor)}\norm{y_\ttt(s+\tau_h\lfloor \tfrac{t-s}{\tau_h}\rfloor)}{L^6}^6+\breve C_2\left(C_h^6+1\right),\label{L6-t-tauh}
\end{align}
because~$s+\tau_h\lfloor \tfrac{t-s}{\tau_h}\rfloor+\tau_h>t\ge s+\tau_h\lfloor \tfrac{t-s}{\tau_h}\rfloor$. From~\eqref{freey-L6} it also follows, for $\lfloor \tfrac{t-s}{\tau_h}\rfloor\ge1$,
\begin{align}
\norm{y_\ttt(s+\tau_h\lfloor \tfrac{t-s}{\tau_h}\rfloor)}{L^6}^6
&\le\rme^{-\tau_h}\norm{y_\ttt(s+\tau_h\lfloor \tfrac{t-s}{\tau_h}\rfloor-\tau_h)}{L^6}^6+\breve C_2\left(C_h^6+1\right)\notag\\
&\le\rme^{-\lfloor \tfrac{t-s}{\tau_h}\rfloor\tau_h}\norm{y_\ttt(s)}{L^6}^6
+\breve C_2\left(C_h^6+1\right){\textstyle\sum\limits_{j=0}^{\lfloor \tfrac{t-s}{\tau_h}\rfloor-1}}\rme^{-j\tau_h}\notag\\
&\le\rme^{-\lfloor \tfrac{t-s}{\tau_h}\rfloor\tau_h}\norm{y_\ttt(s)}{L^6}^6+\tfrac{1}{1-\rme^{-\tau_h}}\breve C_2\left(C_h^6+1\right).\label{L6-tauh+1}
\end{align}
Combining~\eqref{L6-t-tauh} and~\eqref{L6-tauh+1} it follows that
\begin{align}
\norm{y_\ttt(t)}{L^6}^6
&\le\rme^{-(t-s)}\norm{y_\ttt(s)}{L^6}^6+(1+\tfrac{1}{1-\rme^{-\tau_h}}\rme^{-(t-s-\tau_h\lfloor \tfrac{t-s}{\tau_h}\rfloor)})\breve C_2\left(C_h^6+1\right)\notag\\
&\le\rme^{-(t-s)}\norm{y_\ttt(s)}{L^6}^6+\tfrac{2-\rme^{-\tau_h}}{1-\rme^{-\tau_h}}\breve C_2\left(C_h^6+1\right),\notag
\end{align}
which finishes the proof.
\end{proof}
\begin{corollary}\label{C:estL6cont}
If Assumption~\ref{A:h-pers-bound} holds for the external force~$h$, then the solution of the controlled dynamics~\eqref{Schloegl-feed} satisfies, for all~$t\ge s\ge0$,
\begin{align}
&\norm{y_\ttc(t)}{L^6}^6
\le\rme^{-(t-s)}\norm{y_\ttc(s)}{L^6}^6+\breve C_3\left(\norm{h}{L^6((s,t),H)}^6+\norm{y_\ttc(s)-y_\ttt(s)}{H}^6 +1\right), 
\notag\\
&\norm{y_\ttc(t)}{L^6}^6
\le\rme^{-(t-s)}\norm{y_\ttc(s)}{L^6}^6+\tfrac{2-\rme^{-\tau_h}}{1-\rme^{-\tau_h}}\breve C_3\left(C_{h}^6+\norm{y_\ttc(0)-y_\ttt(0)}{H}^6+1\right),
\notag
\end{align}
with
$\breve C_3
\le\ovlineC{\norm{\zeta}{\infty},\norm{U_M^\diamond\overline\clK_M^\lambda}{\clL(H)}}.$
\end{corollary}
\begin{proof}
Replacing~$h$ by~$\overline h\coloneqq h+U_M^\diamond\overline\clK_M^\lambda(y_\ttc-y_\ttt)$, from Lemma~\ref{L:estL6free} we find that
\begin{subequations}\label{estL6barh}
\begin{align}
&\norm{y_\ttc(t)}{L^6}^6
\le\rme^{-(t-s)}\norm{y_\ttc(s)}{L^6}^6+\breve C_2\left(\norm{\overline h}{L^6((s,t),H)}^6+1\right), \\
&\norm{y_\ttc(t)}{L^6}^6
\le\rme^{-(t-s)}\norm{y_\ttc(s)}{L^6}^6+\tfrac{2-\rme^{-\tau_h}}{1-\rme^{-\tau_h}}\breve C_2\left(C_{\overline h}^6+1\right),
\end{align}
\end{subequations}
with~$\breve C_2
\le\ovlineC{\norm{\zeta}{\infty}}$ and~$C_{\overline h}\coloneqq \sup\limits_{r\ge0}\norm{\overline h}{L^6((r,r+\tau_h),H)}$. Observe that
\begin{align}
\norm{\overline h}{L^6((s,t),H)}
&\le\norm{h}{L^6((s,t),H)}+\norm{U_M^\diamond\overline\clK_M^\lambda(y_\ttc-y_\ttt)}{L^6((s,t),H)}\notag\\
&\le\norm{h}{L^6((s,t),H)}+\norm{U_M^\diamond\overline\clK_M^\lambda}{\clL(H)}
\norm{y_\ttc-y_\ttt}{L^6((s,t),H)}\notag
\end{align}
and, by Theorem~\ref{T:mainSchloegl} it follows that
\begin{align}
\norm{\overline h}{L^6((s,t),H)}
&\le\norm{h}{L^6((s,t),H)}+\norm{U_M^\diamond\overline\clK_M^\lambda}{\clL(H)}
\norm{y_\ttc(s)-y_\ttt(s)}{H}\left({\textstyle\int_s^t}\rme^{-6\mu(r-s)}\,\rmd r\right)^\frac16,\notag
\end{align}
which leads us to
\begin{align}
\norm{\overline h}{L^6((s,t),H)}^6
&\le2^5\norm{h}{L^6((s,t),H)}^6+2^5\norm{U_M^\diamond\overline\clK_M^\lambda}{\clL(H)}^6
\norm{y_\ttc(s)-y_\ttt(s)}{H}^6\tfrac1{6\mu},\label{yL6st}
\end{align}
and from~\eqref{estL6barh},
\begin{align}
&\norm{y_\ttc(t)}{L^6}^6
\le\rme^{-(t-s)}\norm{y_\ttc(s)}{L^6}^6+C\left(\norm{h}{L^6((s,t),H)}^6+\norm{y_\ttc(s)-y_\ttt(s)}{H}^6+1\right), \notag
\end{align}
with~$C\coloneqq 2^5\breve C_2(1+\norm{U_M^\diamond\overline\clK_M^\lambda}{\clL(H)}^6
\tfrac1{6\mu})\le \ovlineC{\norm{\zeta}{\infty},\norm{U_M^\diamond\overline\clK_M^\lambda}{\clL(H)}}$.

Next, we observe that~\eqref{yL6st} and Theorem~\ref{T:mainSchloegl} also give us, for given~$r\ge0$,
\begin{align}
\norm{\overline h}{L^6((r,r+\tau_h),H)}^6
&\le2^5\norm{h}{L^6((r,r+\tau_h),H)}^6+2^5\norm{U_M^\diamond\overline\clK_M^\lambda}{\clL(H)}^6
\norm{y_\ttc(r)-y_\ttt(r)}{H}^6\tfrac1{6\mu},\notag\\
&\le2^5\norm{h}{L^6((r,r+\tau_h),H)}^6+2^5\norm{U_M^\diamond\overline\clK_M^\lambda}{\clL(H)}^6
\norm{y_\ttc(0)-y_\ttt(0)}{H}^6\tfrac1{6\mu},
\end{align}
which, together with~\eqref{estL6barh}, lead us to
\begin{align}
&\norm{y_\ttc(t)}{L^6}^6
\le\rme^{-(t-s)}\norm{y_\ttc(s)}{L^6}^6+\tfrac{2-\rme^{-\tau_h}}{1-\rme^{-\tau_h}}C\left(C_{h}^6+\norm{y_\ttc(0)-y_\ttt(0)}{H}^6 +1\right),
\end{align}
and concludes the proof.
\end{proof}

Now, we derive a property for the nonlinearity in the error dynamics as follows.
\begin{lemma}\label{L:esthatf}
The function~$f^{y_\ttt}$ in~\eqref{sys-z-feedf-hat} satisfies
\begin{align}
&\norm{f^{y_\ttt}( z)}{H}\le C_{f^{y_\ttt}}\left(\norm{z}{L^6}^2
+\norm{y_\ttt}{L^6}^2+1\right)\norm{ z}{L^6}\notag
\end{align}
with~$C_{f^{y_\ttt}}=\ovlineC{\norm{\zeta}{\infty}}$.
\end{lemma}
\begin{proof}
By straightforward computations, we find
\begin{align}
\norm{f^{y_\ttt}( z)}{H}&=\norm{z^3+(3y_\ttt+\xi_2) z^2+(3y_\ttt^2+2\xi_2y_\ttt+\xi_1-1) z}{H}\notag\\
&\hspace{0em}=\norm{(z^2+3y_\ttt z+\xi_2z+3y_\ttt^2+2\xi_2y_\ttt+\xi_1-1) z}{L^2}\notag\\
&\hspace{0em}\le\norm{z^2+2\tfrac32y_\ttt z+2\tfrac12\xi_2z+3y_\ttt^2+2\xi_2y_\ttt+\xi_1-1}{L^3}\norm{ z}{L^6}\notag\\
&\hspace{0em}\le\norm{3z^2+((\tfrac{3}{2})^2+4)y_\ttt^2+(\tfrac14+1)\xi_2^2+\xi_1-1}{L^3}\norm{ z}{L^6}\notag\\
&\hspace{0em}\le\ovlineC{\norm{\zeta}{\infty}}
\left(\norm{z}{L^6}^2+\norm{y_\ttt}{L^6}^2+1\right)\norm{ z}{L^6},\notag
\end{align}
which finishes the proof.
\end{proof}

We show next that the $V$-norm of $z(s+1)$, after a period of time~$1$ can be bounded by the $H$-norm of  $z(s)$. For this purpose we start with the following auxiliary result.
\begin{lemma}\label{L:smoo}
There exists~$\overline D>0$ such that the solution of~\eqref{sys-z-feedf-hat} satisfies
\begin{align}
\norm{z(s+1)}{V}^2
&\le \overline D\norm{ z(s)}{H}^2,\quad\mbox{for all}\quad s\ge0.\notag
\end{align}
Furthermore,~$\overline D\le\ovlineC{\norm{\zeta}{\infty},\norm{y_\ttt(0)}{L^6},\norm{z(0)}{L^6},\tau_h^{-1},C_h,\norm{U_M^\diamond\overline\clK_M^\lambda}{\clL(H)}}$.
\end{lemma}
\begin{proof}
Let us fix~$s\ge0$. The function~$w(t)\coloneqq\tau_s(t)  z(t)$
solves, with~$\tau_s(t)\coloneqq t-s$,
\begin{align} \label{sys-tz-feedf-hat}
 &\tfrac{\p}{\p t} w =-A w  -\tau_s f^{y_\ttt}( z)+\tau_s U_M^\diamond\overline\clK_M( z)+ z,\qquad w(s)=0.
 \end{align}
After multiplying the dynamics by~$2Aw$, we find
\begin{align}
 &\tfrac{\rmd}{\rmd t} \norm{w}{V}^2+2\norm{Aw}{H}^2=-2(\tau_sf^{y_\ttt}( z),Aw)_{H}+2(\tau_s U_M^\diamond\overline\clK_M( z)+ z,Aw)_{H}.\notag
\end{align}
With~$C_{f^{y_\ttt}}$ as in Lemma~\ref{L:esthatf}, it follows that
\begin{align}
-2(\tau_sf^{y_\ttt}( z),Aw)_{H}&=-2\tau_s^2(f^{y_\ttt}(z),Az)_{H}
\le 2\tau_s^2 C_{f^{y_\ttt}}\left(\norm{z}{L^6}^2
+\norm{y_\ttt}{L^6}^2+1\right)\norm{ z}{L^6}\norm{A z}{H}\notag\\
&= 2 C_{f^{y_\ttt}}\left(\norm{z}{L^6}^2
+\norm{y_\ttt}{L^6}^2+1\right)\norm{ w}{L^6}\norm{A w}{H}\le 2\Xi\norm{ w}{V}\norm{A w}{H}\notag\\
\mbox{with}\quad\Xi(t)&\coloneqq C_{f^{y_\ttt}}\left(\norm{z(t)}{L^6}^2
+\norm{y_\ttt(t)}{L^6}^2+1\right)\norm{\Id}{\clL(V,L^6)},\notag
\end{align}
which leads us to
\begin{align}
 \tfrac{\rmd}{\rmd t} \norm{w}{V}^2&\le-\tfrac32\norm{Aw}{H} ^2+2\Xi^2\norm{w}{V} ^2 +2(\tau_s U_M^\diamond\overline\clK_M( z)+ z,Aw)_{H},\notag\\
&\le-\norm{Aw}{H} ^2+2\Xi^2\norm{w}{V} ^2 +4\norm{\tau_s (U_M^\diamond\overline\clK_M( z)}{H}^2 +4\norm{z}{H}^2.\label{dtz-smoo}
\end{align}
Together with the Gronwall inequality, and since~$w(s)=0$, we obtain
\begin{align}
\norm{w(t)}{V}^2&\le4\rme^{2\norm{\Xi}{L^2((s,t),\bbR)}^2}
\left(\norm{\tau_s (U_M^\diamond\overline\clK_M^\lambda( z)}{L^2((s,t),H)}^2 +\norm{z}{L^2((s,t),H)}^2\right),\notag
\end{align}
for all~$t\ge s$. Taking~$t=s+1$, we conclude that, for all~$s\ge0$,
\begin{align}
\norm{z(s+1)}{V}^2
&\le4\rme^{2\norm{\Xi}{L^2((s,s+1),\bbR)}^2}
\left(\norm{ (U_M^\diamond\overline\clK_M^\lambda}{\clL(H)}^2
+1\right)\norm{ z}{L^2((s,s+1),H)}^2.\notag
\end{align}

Finally, observe that with $D_0\coloneqq C_{f^{y_\ttt}}\norm{\Id}{\clL(V,L^6)}$ we have
\begin{align}
\norm{\Xi}{L^2((s,s+1),\bbR)}^2&\le D_0\norm{\norm{z}{L^6}^2
+\norm{y_\ttt}{L^6}^2+1}{L^2((s,s+1),\bbR)}^2,\notag\\
&\le 3D_0\left(\norm{z}{L^4((s,s+1),L^6)}^4
+\norm{y_\ttt}{L^4((s,s+1),L^6)}^4+1\right).\notag
\end{align}
By Lemma~\ref{L:estL6free} and Corollary~\ref {C:estL6cont}
we find that
\begin{align}
&\norm{y_\ttt(t)}{L^6}^6\le\norm{y_\ttt(0)}{L^6}^6+\tfrac{2-\rme^{-\tau_h}}{1-\rme^{-\tau_h}}\breve C_2\left(C_h^6+1\right)\eqqcolon D_1,\notag\\
&\norm{y_\ttc(t)}{L^6}^6\le\norm{y_\ttc(0)}{L^6}^6+\tfrac{2-\rme^{-\tau_h}}{1-\rme^{-\tau_h}}\breve C_3\left(C_h^6+\norm{y_\ttc(0)-y_\ttt(0)}{H}^6+1\right)\eqqcolon D_2,\notag
\end{align}
and, recalling that~$z=y_\ttc-y_\ttt$,
\begin{align}
&\norm{\Xi}{L^2((s,s+1),\bbR)}^2
\le 3D_0\left(D_1^\frac46+D_2^\frac46+1\right)\notag\\
&\hspace{1em}\le 3D_0\left(\norm{y_\ttt(0)}{L^6}^4+\norm{y_\ttt(0)+z(0)}{L^6}^4+(\tfrac{2-\rme^{-\tau_h}}{1-\rme^{-\tau_h}})^\frac46
(\breve C_2^\frac46+\breve C_3^\frac46)
\left(C_h^6+\norm{z(0)}{H}^6+1\right)^\frac46\right)\notag\\
&\hspace{1em}
\eqqcolon D_4\le\ovlineC{\norm{\zeta}{\infty},\norm{y_\ttt(0)}{L^6},\norm{z(0)}{L^6},\tau_h^{-1},C_h,\norm{U_M^\diamond\overline\clK_M^\lambda}{\clL(H)}},\label{L2Xi}
\end{align}
which leads us to
\begin{align}
\norm{z(s+1)}{V}^2
&\le4\rme^{2D_4}
\left(\norm{ U_M^\diamond\overline\clK_M^\lambda}{\clL(H)}^2
+1\right)\norm{ z}{L^2((s,s+1),H)}^2.\notag
\end{align}
We finish the proof by observing that
$\norm{z}{L^2((s,s+1),H)}^2\le\norm{z}{L^\infty((s,s+1),H)}^2=\norm{z(s)}{H}^2$,  due to Theorem~\ref{T:mainSchloegl}.
\end{proof}

%%%%%%%%%%%%%%%%%%%%%%%%%%%
\subsection{Proof of Theorem~\ref{T:mainSchloegl-st}}\label{sS:proofT:mainSchloegl-st}
Multiplying the dynamics in~\eqref{sys-z-feedf-hat} by~$2Az$ and proceeding as in the proof of Lemma~\ref{L:smoo} we can arrive at the analog of~\eqref{dtz-smoo},
\begin{align}
 \tfrac{\rmd}{\rmd t} \norm{z}{V}^2 &\le-\tfrac32\norm{Az}{H} ^2+2\Xi^2\norm{z}{V} ^2 +2(U_M^\diamond\overline\clK_M( z),Az)_{H},\notag\\
&\le-\norm{Az}{H} ^2+2\Xi^2\norm{z}{V} ^2 +2\norm{U_M^\diamond\overline\clK_M(z)}{H}^2.
\end{align}
Hence, the Gronwall inequality and~\eqref{L2Xi} give us, for~$0\le s\le t\le s+1$,
\begin{align}
\norm{z(t)}{V}^2&\le\rme^{2D_4}
\left(\norm{z(s)}{V}^2+2\norm{U_M^\diamond\overline\clK_M z}{L^2((s,t),H)}^2\right)\notag\\
&\le\rme^{2D_4+\mu}\rme^{-\mu(t-s)}
\left(\norm{z(s)}{V}^2+2\norm{U_M^\diamond\overline\clK_M}{\clL(H)}^2\norm{z}{L^2((s,t),H)}^2\right)\notag\\
&\le C_1\rme^{-\mu(t-s)}
\norm{z(s)}{V}^2,\quad\mbox{for all}\quad 0\le s\le t\le s+1,\label{expVtsmall}
\end{align}
with~$C_1=\rme^{2D_4+\mu}\left(1+2\norm{U_M^\diamond\overline\clK_M}{\clL(H)}^2\right)\le
\ovlineC{\mu,\norm{\zeta}{\infty},\norm{y_\ttt(0)}{L^6},\norm{z(0)}{L^6},\tau_h^{-1},C_h,\norm{U_M^\diamond\overline\clK_M^\lambda}{\clL(H)}}$.

Next, for~$t\ge s+1\ge1$, by Lemma~\ref{L:smoo} and Theorem~\ref{T:mainSchloegl} we have that
\begin{align}
\norm{z(t)}{V}^2
&\le \overline D\norm{z(t-1)}{H}^2\le  \overline D\rme^{-\mu(t-1-s)}\norm{z(s)}{H}^2\notag\\
&\le \overline D\rme^\mu\norm{\Id}{\clL(V,H)}^2\rme^{-\mu(t-s)}\norm{z(s)}{V}^2
,\quad t\ge s+1,\label{expVtlarge}
\end{align}
with~$\overline D\le\ovlineC{\norm{\zeta}{\infty},\norm{y_\ttt(0)}{L^6},\norm{z(0)}{L^6},\tau_h^{-1},C_h,\norm{U_M^\diamond\overline\clK_M^\lambda}{\clL(H)}}$.

We finish the proof by combining~\eqref{expVtsmall} and~\eqref{expVtlarge},
\begin{align}
\norm{z(t)}{V}^2
&\le \varrho
\rme^{-\mu (t-s)}\norm{z(s)}{V}^2,\quad\mbox{for all}\quad t\ge s\ge0,
\end{align}
with $\varrho\coloneqq \max\{C_1,\overline D\rme^\mu\norm{\Id}{\clL(V,H)}^2\}$.
\qed

%%%%%%%%%%%%%%%%%%%%%%%%%%%
%%%%%%%%%%%%%%%%%%%%%%%%%%%
 \section{Existence of optimal controls}\label{S:optimalcontrol}

The main goal of this section is to show the existence and boundedness of controls minimizing classical energy-like functionals.

Let~$z= y_\ttc-y_\ttt$, where~$y_\ttt$ is the solution of the free dynamics~\eqref{sys-haty} and~$y_\ttc$ is the solution of~\eqref{sys-y-u-Cu}. Then, for an arbitrary open time interval~$I\subseteq\bbR^+$, with~$s_0\coloneqq \inf I$, we have that $z$ satisfies
\begin{align} \label{sys-z-hat}
 &\dot z =-A z -f^{y_\ttt}(z) +U_M^\diamond u,\qquad z(s_0)=z_0,\qquad \dnorm{u(t)}{}\le C_u,
 \end{align}
with$z_0\coloneqq y_{\ttc}(s_0)-y_{\ttt}(s_0)$, $f^{y_\ttt}$ as in~\eqref{sys-z-feedf-hat}, and~$U_M^\diamond$ as in~\eqref{UMdiam}.

We introduce   the energy functional~\eqref{Jy0},
\begin{equation}\label{Jz0}
\clJ_{I}^Q(z,u)\coloneqq\tfrac12\norm{Qz}{L^2(I,{H}
)}^2+\tfrac12\norm{u}{L^2(I,\bbR^{M_\sigma})}^2.
\end{equation}

The operator~$Q$ will need to satisfy an assumption involving the error dynamics with an arbitrary external forcing~$g\in L^2(I,H)$. For this purpose, we consider the system
\begin{align} \label{sys-z-g}
 &\dot w =-A w -f^{y_\ttt}(w) +g,\qquad w(s_0)=w_0\in V.
 \end{align}

\begin{assumption}\label{A:Q}
For~$Q\in\clL(\rmD(A),H)$, there exists a constant~$C_Q>0$ such that  the solutions of~\eqref{sys-z-g} satisfy the estimate
\begin{align}
\norm{z}{L^2(I,V)}^2\le C_Q\left(\norm{w_0}{H}^2+\norm{Qz}{L^2(I,H)}^2+\norm{g}{L^2(I,H)}^2\right).
  \notag 
 \end{align}
with~$C_Q$ independent of~$(w_0,g)$.
\end{assumption}

This assumption will be discussed  below. Now, let us denote the Hilbert spaces
\begin{subequations}\label{opt-setting}
\begin{align}
&\clX_I\coloneqq W(I,{\rmD(A)},H)\times L^2(I,{\bbR^{M_\sigma}}),\\
&\clY_I\coloneqq L^2(I,H)\times V,
\intertext{the convex sets}
&\clC_I\coloneqq\{(z,u)\in\clX_{I}\mid \dnorm{u(t)}{}\le C_u\mbox{ for } t\ge s_0\}
\intertext{and the mapping}
&\clG_I^{z_0}\colon\clX_I\to\clY_I,\qquad(z,u)\mapsto(\dot z +Az +f^{y_\ttt}(z) -U_M^\diamond u, z(s_0)-z_0).
\intertext{Then, we define the set}
&\fkX_{I}^{z_0}\coloneqq\{(z,u)\in\clX_{I}\mid \clG_I^{z_0}(z,u)=(0,0)\mbox{ and } (z,u)\in\clC_I\}
\end{align}
\end{subequations}
and consider the minimization problem as follows.
\begin{problem}\label{Pb:OCPI}
Given~$Q$ as in Assumption~\ref{A:Q} and~$z_0\in V$, find
\[
(z_{I}^{z_0} ,u_{I}^{z_0} )\in\argmin_{(z,u)\in\fkX_{I}^{z_0}}\clJ_I^Q(z,u).
\]
\end{problem}

%%%%%%%%%%%%%%%%%%%%%%%%%%%%
We are particularly interested in the case~$I=\bbR_+$.

Hereafter, pairs as~$(z_{I}^{z_0} ,u_{I}^{z_0})$ are assumed to solve Problem~\ref{Pb:OCPI}. Here, it is also understood that we are referring to the initial state~$z_0$ and time interval~$I$, included in the notation.
The main result of this section is as follows, the proof of which shall be given in section~\ref{sS:proofT:existOptim}.
\begin{theorem}\label{T:existOptim}
Let~$Q\in\clL(\rmD(A),H)$ satisfy Assumption~\ref{A:Q} and an let~$I\subseteq\bbR_+$ be a nonempty open interval. Then, for every~$z_0\in V$ there exists~$(z_{I}^{z_0} ,u_{I}^{z_0} )\in \fkX_{I}^{z_0}$  solving Problem~\ref{Pb:OCPI}. Further, there exists a constant~$\widehat C>0$, independent of~$(z_0,I)$ so that
$\clJ_{I}^Q(z_{I}^{z_0} ,u_{I}^{z_0})\le \widehat C\norm{z_0}{V}^2$.
\end{theorem}

%%%%%%%%%%%%%%%%%%%%%%%%%%%%
\subsection{Remark on Assumption~\ref{A:Q}}\label{sS:RemA:Q}
Clearly, the requirement in Assumption~\ref{A:Q} is satisfied with~$Q=A^\frac12\in\clL(V,H)\subset\clL(\rmD(A),H)$. However, for applications it may be convenient to take~$Q=\Id$ or even~$Q$ with finite-dimensional range.
We show now the satisfiability of the assumption in the case where~$Q=P_{\clE_{M_1}^\rmf}$ is the orthogonal projection in~$H$ onto the linear span~$\clE_{M_1}^\rmf$ of the first~$M_1$ eigenfunctions~$e_i$ of the diffusion operator~$A$,
\[
\clE_{M_1}^\rmf\coloneqq\linspan\{e_i\mid 1\le i\le M_1\}.
\]

\begin{lemma}\label{L:key-exist-Id}
 There exists~$M_{1*}\in\bbN$ such that Assumption~\ref{A:Q} holds with~$Q=P_{\clE_{M_1}^\rmf}$, for all~$M_1\ge M_{1*}$. Furthermore, $C_Q=\ovlineC{\norm{\zeta}{\infty}}$ and~$M_1\le  \ovlineC{\norm{\zeta}{\infty}}$, with~$\norm{\zeta}{\infty}$ as in~\eqref{zetainf}.
\end{lemma}
\begin{proof}
Note that~$Q=P_{\clE_{M_1}^\rmf}\in\clL(H)\subset\clL(\rmD(A),H)$.
Recalling $f^{y_\ttt}$ defined in~\eqref{sys-z-feedf-hat},
in~\cite[Sect.~2.2]{AzmiKunRod21-arx} we find the estimate
\begin{align} \label{estz3hatf-z}
-2(f^{y_\ttt}(z),z)_{H}&\le-\tfrac18\norm{z}{L^4}^4+C_1 \norm{z}{H}^2.
\end{align}
with~$C_1\le\ovlineC{\norm{\zeta}{\infty}}$.
Multiplying the dynamics with~$2z$, we find
\begin{align}
\tfrac{\rmd}{\rmd t}\norm{z}{H}^2
 &\le-2\norm{z}{V}^2+C_1\norm{z}{H}^2+2(g,z)_{H}\le-\norm{z}{V}^2+C_1\norm{z}{H}^2+\norm{g}{H}^2.\notag
 \end{align}
Thus, after time integration, we obtain
\begin{align}
 \norm{z}{L^2(I,V)}^2
 &\le\norm{z(s_0)}{H}^2+C_1\norm{z}{L^2(I,H)}^2+\norm{g}{L^2(I,H)}^2\notag\\
&\le C_2\left(\norm{z_0}{H}^2+\norm{z}{L^2(I,H)}^2+\norm{g}{L^2(I,H)}^2\right), \notag
 \end{align}
with~$C_2\coloneqq 1+C_1=\ovlineC{\norm{\zeta}{\infty}}$.
Now, writing~$z=\theta+\Theta$ with the orthogonal components~$\theta\coloneqq P_{\clE_{M_1}^\rmf}z$ and~$\Theta\coloneqq z-\theta$, we find
\begin{align}
 2\norm{z}{L^2(I,V)}^2
 &\le 2C_2\left(\norm{z_0}{H}^2+\norm{\theta}{L^2(I,H)}^2+\norm{\Theta}{L^2(I,H)}^2+\norm{g}{L^2(I,H)}^2\right)\notag\\
&\le 2C_2\left(\norm{z_0}{H}^2+\norm{\theta}{L^2(I,H)}^2+\alpha_{M_1}^{-1}\norm{\Theta}{L^2(I,V)}^2+\norm{g}{L^2(I,H)}^2\right),\notag
 \end{align}
 where~$\alpha_{M_1}$ is the~$M_1$th eigenvalue of~$A$. Since~$\alpha_{i}\to+\infty$, we conclude that for large enough~$M_1$, with~ $\alpha_{M_1}\ge 2C_2$, we have
\begin{align}
\norm{z}{L^2(I,V)}^2\le 2\norm{\theta}{L^2(I,V)}^2+\norm{\Theta}{L^2(I,V)}^2
&\le 2C_2\left(\norm{z_0}{H}^2+\norm{\theta}{L^2(I,H)}^2+\norm{g}{L^2(I,H)}^2\right).
 \notag
 \end{align}

 Defining~$M_{1*}=\min\{i\in\bbN_+\mid \alpha_{i}\ge 2C_2\}$ we observe that the requirement~$\alpha_{M_1}\ge 2C_2$ is achieved for~$M_1\ge M_{1*}=\ovlineC{C_2}=\ovlineC{\norm{\zeta}{\infty}}$, which finishes the proof.
\end{proof}

%%%%%%%%%%%%%%%%%%%%%%%%%%%%
\subsection{A key auxiliary result}\label{sS:exist-optim-key}
In this section we prove an estimate for the controlled solutions, which we shall use to derive the  existence of optimal controls.
 \begin{lemma}\label{L:key-exist-Q}
Let~$Q\in\clL(\rmD(A),H)$ be as in Assumption~\ref{A:Q}. Then, there exists a constant~$ D_Q>0$ such that, for every interval~$I\subseteq\bbR_+$, $z_0\in V$, and~$u\in L^2(I,\bbR^{M_\sigma})$ we have that the solution~$z$ of the system
\begin{align}
 &\dot z =-A z -f^{y_\ttt}(z) +U_M^\diamond u,\qquad z(s_0)=z_0,\notag
 \end{align}
where~$s_0=\inf I$,
satisfies the estimate
\[
\norm{z}{W(I,\rmD(A),H)}^2\le D_Q\left(\norm{z_0}{V}^2+\norm{Qz}{L^2(I,H)}^2+\norm{u}{L^2(I,\bbR^{M_\sigma})}^2\right).
\]
Furthermore, the constant~$D_Q$ is independent of~$s_0$ and can be taken as
\[
D_Q\le \ovlineC{\norm{\zeta}{\infty},\norm{U_M^\diamond}{\clL(\bbR^{M_\sigma},H)},C_h,\tfrac1{\tau_h},\norm{y_\ttt(0)}{L^6},\norm{z_0}{V},\norm{Qz}{L^2(I,H)},\norm{u}{L^2(I,\bbR^{M_\sigma})}}.
\]
\end{lemma}
\begin{proof}
By Assumption~\ref{A:Q}, we have that
\begin{align}
\norm{z}{L^2(I,V)}^2\le C_Q\left(\norm{z_0}{H}^2+\norm{Qz}{L^2(I,H)}^2+\norm{U_M^\diamond u}{L^2(I,H)}^2\right).
  \label{zL2VQ}
 \end{align}

Multiplying the dynamics with~$2Az$
we find
\begin{align}
\tfrac{\rmd}{\rmd t}\norm{z}{V}^2
 &\le-2\norm{z}{\rmD(A)}^2-2(f^{y_\ttt}(z),Az)_{H}
+2\norm{U_M^\diamond u}{H}\norm{z}{\rmD(A)}\label{dtzDA-1}
 \end{align}

Recalling the expression for the nonlinearity in~\eqref{sys-z-feedf-hat}, we have that
\begin{subequations}\label{hatf-est-exist}
\begin{align}
-(f^{y_\ttt}(z),Az)_{H}
 &=(-z^3-(3y_\ttt+\xi_2) z^2-(3y_\ttt^2+2\xi_2y_\ttt+\xi_1-1) z,Az)_{H}\notag\\
&\hspace{-4em}=(-z^3-\xi_2 z^2+(1-\xi_1)z,Az)_{H} +(-3y_\ttt z^2-(3y_\ttt^2+2\xi_2y_\ttt) z,Az)_{H}.
\end{align}
From direct computations we obtain
\begin{align}
&(-z^3-\xi_2 z^2+(1-\xi_1)z,Az)_{H}\notag\\
&\hspace{1em}=(-z^2-\xi_2 z+1-\xi_1,z^2)_{H}+\nu(-3z^2-2\xi_2 z+1-\xi_1,\norm{\nabla z}{\bbR^d}^2)_{H}\le C_3\norm{z}{V}^2
\end{align}
with~$C_3\coloneqq\max\limits_{r\in\bbR}(-r^2-\xi_2 r+1-\xi_1)+\nu\max\limits_{s\in\bbR}(-3s^2-2\xi_2 s+1-\xi_1)$.
That is,
\[
C_3=\tfrac{(4\nu+3)\xi_2^2}{12}+(1+\nu)(1-\xi_1)\le\ovlineC{\norm{\zeta}{\infty}}.
\]
For the term involving the targeted state~$y_\ttt$ we find
\begin{align}
&\norm{-3y_\ttt z^2-(3y_\ttt^2+2\xi_2y_\ttt) z}{H}
\le3\norm{y_\ttt z^2}{H}+\norm{(3y_\ttt^2+2\xi_2y_\ttt) z}{H}\notag\\
&\hspace{3em}\le3\norm{y_\ttt}{L^6}\norm{z^2}{L^3}+\norm{3y_\ttt^2+2\xi_2y_\ttt}{L^3}\norm{z}{L^6}\le3\norm{y_\ttt}{L^6}\norm{z^2}{L^3}+\norm{4y_\ttt^2+\xi_2^2}{L^3}\norm{z}{L^6}\notag\\
&\hspace{3em}\le3\norm{y_\ttt}{L^6}\norm{z}{L^6}^2+(4\norm{y_\ttt}{L^6}^2+\xi_2^2\dnorm{\Omega}{}^\frac13)\norm{z}{L^6}
\le C_4\norm{z}{V}^2+\Phi\norm{z}{V}
 \end{align}
with~$C_4\coloneqq 3\norm{y_\ttt}{L^6}\norm{\Id}{\clL(V,L^6)}^2$ and~$\Phi\coloneqq (4\norm{y_\ttt}{L^6}^2+\xi_2^2\dnorm{\Omega}{}^\frac13)\norm{\Id}{\clL(V,L^6)}$.
\end{subequations}
Therefore, ~\eqref{dtzDA-1},~\eqref{hatf-est-exist}, and the Young inequality, give us
\begin{align}
\tfrac{\rmd}{\rmd t}\norm{z}{V}^2
 &\le-2\norm{z}{\rmD(A)}^2+2C_3\norm{z}{V}^2
+2\left(C_4\norm{z}{V}^2+\Phi\norm{z}{V}\right)\norm{z}{\rmD(A)}
+2\norm{U_M^\diamond u}{H}\norm{z}{\rmD(A)}\notag\\
&\le-\norm{z}{\rmD(A)}^2+2C_3\norm{z}{V}^2
+2\left(C_4\norm{z}{V}^2+\Phi\norm{z}{V}\right)^2
+2\norm{U_M^\diamond u}{H}^2\notag\\
&\le-\norm{z}{\rmD(A)}^2+(2C_3+4\Phi^2)\norm{z}{V}^2
+4C_4^2\norm{z}{V}^4
+2\norm{U_M^\diamond u}{H}^2.\label{dtzDA-2}
 \end{align}

By Lemma~\ref{L:estL6free} we have that
\begin{align}
\norm{y_\ttt(t)}{L^6}^6
\le\norm{y_\ttt(0)}{L^6}^6+C_5\left(C_h^6+1\right),\quad\mbox{for all}\quad t\ge0,\label{LinfL6haty}
\end{align}
with
$C_5\le\ovlineC{\norm{\zeta}{\infty},C_h,\tfrac1{\tau_h}}$, which gives us~$C_4\le\ovlineC{\norm{\zeta}{\infty},C_h,\tfrac1{\tau_h},\norm{y_\ttt(0)}{L^6}}$ and~$\Phi\le C_6\le\ovlineC{\norm{\zeta}{\infty},C_h,\tfrac1{\tau_h},\norm{y_\ttt(0)}{L^6}}$.

From~\eqref{dtzDA-2}, we have
\begin{subequations}\label{dtzDA-2-Gronwall}
\begin{align}
&\tfrac{\rmd}{\rmd t}\norm{z}{V}^2
 \le
a\norm{z}{V}^2
+b
\intertext{with}
&a\coloneqq4C_4^2\norm{z}{V}^2,\qquad
b\coloneqq C_7\norm{z}{V}^2+2\norm{U_M^\diamond u}{H}^2,
 \end{align}
\end{subequations}
and~$C_7\coloneqq 2C_3+4C_6^2$.
We apply the Gronwall inequality to~\eqref{dtzDA-2-Gronwall}, which implies that, for every~$s_1> s_0$ with~$s_1\le\sup I$ and~$I^{s_1}\coloneqq(s_0,s_1)$,   it holds the estimate
\begin{align}
\norm{z}{L^\infty(I^{s_1},V)}^2
 &\le\rme^{4C_4^2\norm{z}{L^2(I^{s_1},V)}^2}\left(\norm{z(s_0)}{V}^2 +C_7\norm{z}{L^2(I^{s_1},V)}^2+2\norm{U_M^\diamond u}{L^2(I^{s_1},H)}^2\right)\notag\\
&\le\rme^{4C_4^2\norm{z}{L^2(I,V)}^2}\left(\norm{z(s_0)}{V}^2 +C_7\norm{z}{L^2(I,V)}^2+2\norm{U_M^\diamond u}{L^2(I,H)}^2\right),\notag
\end{align}
and thus
\begin{align}
\norm{z}{L^\infty(I,V)}^2
&\le C_8\left(\norm{z(s_0)}{V}^2 +C_7\norm{z}{L^2(I,V)}^2+2\norm{U_M^\diamond u}{L^2(I,H)}^2\right),\label{LinfVz}
 \end{align}
where, using~\eqref{zL2VQ},
\[
C_8\coloneqq \rme^{4C_4^2\norm{z}{L^2(I,V)}^2}\le\ovlineC{\norm{\zeta}{\infty},\norm{U_M^\diamond}{\clL(\bbR^{M_\sigma},H)},C_h,\tfrac1{\tau_h},\norm{y_\ttt(0)}{L^6},\norm{z_0}{H},\norm{Qz}{L^2(I,H)},\norm{u}{L^2(I,\bbR^{M_\sigma})}}.
\]
We underline that the constant~$C_8$ is independent of~$s_0=\inf I$.
By integration of~\eqref{dtzDA-2},
\begin{align}
\norm{z}{L^2(I,\rmD(A))}^2
 &\le \norm{z_0}{V}^2+(C_7+4C_4^2\norm{z}{L^\infty(I,V)}^2)\norm{z}{L^2(I,V)}^2
+2C_U^2\norm{u}{L^2(I,\bbR^{M_\sigma})}^2,\label{zDA}
 \end{align}
where~$C_U\coloneqq\norm{U_M^\diamond}{\clL(\bbR^{M_\sigma},H)}$.
Combining~\eqref{zDA},~\eqref{LinfVz}, and~\eqref{zL2VQ}, allows us to write
\begin{align}
\norm{z}{L^2(I,\rmD(A))}^2
 &\le C_9\left(\norm{z_0}{V}^2+\norm{Qz}{L^2(I,H)}^2+\norm{u}{L^2(I,\bbR^{M_\sigma})}^2\right)\label{zL2DA}
 \end{align}
with~$C_9\le \ovlineC{\norm{\zeta}{\infty},C_U,C_h,\tfrac1{\tau_h},\norm{y_\ttt(0)}{L^6},\norm{z_0}{H},\norm{Qz}{L^2(I,H)},\norm{u}{L^2(I,\bbR^{M_\sigma})}}$.

Finally, from the dynamics we find that
\begin{align}
\norm{\dot z}{L^2(I,H)}^2 \le 3\norm{z}{L^2(I,\rmD(A))}^2 +3\norm{f^{y_\ttt}(z)}{L^2(I,H)}^2+3C_U^2\norm{u}{L^2(I,\bbR^{M_\sigma})}^2,\notag
 \end{align}

 Observe that, using~\eqref{LinfL6haty} and~\eqref{LinfVz}, we have that
\begin{align}
\norm{z}{{ L^\infty(I,L^6)}}^2
+\norm{y_\ttt}{{ L^\infty(I,L^6)}}^2\le \norm{\Id}{\clL(V,L^6)}^2\norm{z}{L^\infty(I,V)}^2+\norm{y_\ttt}{L^\infty(I,L^6)}^2\le C_{10}\notag
\end{align}
with~$C_{10}\le\ovlineC{\norm{\zeta}{\infty},C_U,C_h,\tfrac1{\tau_h},\norm{y_\ttt(0)}{L^6},\norm{z_0}{V},\norm{Qz}{L^2(I,H)},\norm{u}{L^2(I,\bbR^{M_\sigma})}}$. Recalling Lemma~\ref{L:esthatf},
\begin{align}
&\norm{f^{y_\ttt}(z)}{L^2(I,H)}^2\le C_{f^{y_\ttt}}^2\left(\norm{z}{L^\infty(I,L^6)}^2
+\norm{y_\ttt}{L^\infty(I,L^6)}^2+1\right)^2\norm{ z}{L^2(I,L^6)}^2,\notag\\
\intertext{with~$C_{f^{y_\ttt}}=\ovlineC{\norm{\zeta}{\infty}}$, and}
&\norm{\dot z}{L^2(I,H)}^2 \le 3\norm{z}{L^2(I,\rmD(A))}^2 +3C_{f^{y_\ttt}}^2(C_{10}+1)^2\norm{\Id}{\clL(V,L^6)}^2\norm{ z}{L^2(I,V)}^2+3C_U^2\norm{u}{L^2(I,\bbR^{M_\sigma})}^2.\notag
 \end{align}

By~\eqref{zL2VQ} and~\eqref{zL2DA} it follows that
\begin{align}
\norm{\dot z}{L^2(I,H)}^2 &\le C_{11}\left(\norm{z_0}{V}^2 +\norm{Qz}{L^2(I,H)}^2+\norm{u}{L^2(I,\bbR^{M_\sigma})}^2\right),\label{dotzL2}
 \end{align}
with~$C_{11}\le \ovlineC{\norm{\zeta}{\infty},C_U,C_h,\tfrac1{\tau_h},\norm{y_\ttt(0)}{L^6},\norm{z_0}{V},\norm{Qz}{L^2(I,H)},\norm{u}{L^2(I,\bbR^{M_\sigma})}}$.

From~\eqref{zL2DA} and~\eqref{dotzL2} it follows that, with~$D\coloneqq\max\{C_{9},C_{11}\}$ it follows that
\begin{align}
\norm{z}{W(I,\rmD(A),H)}^2 \le D\left(\norm{z_0}{V}^2+\norm{Qz}{L^2(I,H)}^2+\norm{u}{L^2(I,\bbR^{M_\sigma})}^2\right),\notag
 \end{align}
which finishes the proof.
\end{proof}

%%%%%%%%%%%%%%%%%%%%%%%%%%%%
\subsection{Proof of Theorem~\ref{T:existOptim}}\label{sS:proofT:existOptim}
Let~$s_0\coloneqq\inf I$ be the infimum of the interval~$I\subseteq\bbR_+$. By Theorem~\ref{T:mainSchloegl} we have that the pair~$(\overline z,\overline u)\coloneqq(y_\ttc-y_\ttt,\overline\clK_M^\lambda(y_\ttc-y_\ttt))$ satisfies
\begin{subequations}\label{bdd-optimcost}
\begin{align}
\clJ_{I}^Q(\overline z,\overline u)&=\tfrac12\norm{Q\overline z}{L^2(I,H)}^2+\tfrac12\norm{\overline\clK_M^\lambda\overline z}{L^2(I,\bbR^{M_\sigma})}^2
\le\tfrac12\norm{\overline z}{L^2(I,H)}^2+\tfrac12\norm{\clK_M^\lambda\overline z}{L^2(I,\bbR^{M_\sigma})}^2\notag\\
&\le\tfrac12(1+\norm{\clK_M^\lambda}{\clL(L^2,\bbR^{M_\sigma})}^2)\norm{\overline z}{L^2(I,H)}^2
\le \widehat C\norm{z_0}{H}^2,
\intertext{with}
\widehat C&\coloneqq  (2\mu)^{-1}(1+\norm{\clK_M^\lambda}{\clL(L^2,\bbR^{M_\sigma})}^2).
\end{align}
\end{subequations}
Here $y_\ttt$ and $y_\ttc$ solve \eqref{sys-haty} and~\eqref{Schloegl-feed}, respectively.
By Lemma~\ref{L:key-exist-Q} we have~$(\overline z,\overline u)\in\fkX_{I}^{z_0}$.

Now, let us consider a minimizing sequence~$(\overline z_k,\overline u_k)\in\fkX_{I}^{z_0}$,
\begin{align}
\clJ_{I}^Q(\overline z_k,\overline u_k)\to\inf_{(z,u)\in\fkX_I^{z_0}}\clJ_{I}^Q(z,u),\quad\mbox{with}\quad\clJ_{I}^Q(\overline z_k,\overline u_k)\le \widehat C\norm{z_0}{H}^2.\label{min-seq}
\end{align}
Note that~$\clJ_{I}^Q(\overline z_k,\overline u_k)=\norm{(Q\overline z_k,\overline u_k)}{\clZ_I}^2$, with~$\clZ_I\coloneqq L^2(I,H)\times L^2(I,\bbR^{M_\sigma})$.
Using again Lemma~\ref{L:key-exist-Q} we have that
\begin{align}
\norm{(\overline z_k,\overline u_k)}{\clX_I}^2\le D_Q\left(\norm{z_0}{V}^2+\norm{Q\overline z_k}{L^2(I,H)}^2+\norm{\overline u_k}{L^2(I,\bbR^{M_\sigma})}^2\right)\le\widehat D,\label{bddX}
\end{align}
with~$\widehat D\le \ovlineC{\norm{\zeta}{\infty},\norm{U_M^\diamond}{\clL(\bbR^{M_\sigma},H)},C_h,\tfrac1{\tau_h},\norm{y_\ttt(0)}{L^6},\norm{z_0}{V},\widehat C}$. Therefore we can take a subsequence, which we still denote by~$(\overline z_k,\overline u_k)_{k\in\bbN}$, weakly converging to some~$(\overline z_\infty,\overline u_\infty)\in\clX_I$,
\begin{align}
(\overline z_k,\overline u_k)\xrightharpoonup[\clX_I]{}(\overline z_\infty,\overline u_\infty).\label{exist-weakconv}
\end{align}
In particular~$(\overline z_k,\overline u_k)\xrightharpoonup[\clZ_I]{}(\overline z_\infty,\overline u_\infty)$ and due to~\eqref{min-seq} we obtain
\begin{align}
(Q\overline z_k,\overline u_k)\xrightarrow[\clZ_I]{}(Q\overline z_\infty,\overline u_\infty).\label{exist-weakconvZ}
\end{align}

Next for every~$s_0<r\le\sup I$, with~$I^r\coloneqq(s_0,r)$, we also have that
$\norm{(\overline z_k,\overline u_k)}{\clX_{I^r}}^2\le\widehat D$, thus we can assume that
\begin{align}
(\overline z_k,\overline u_k)\rest{I_r}\xrightarrow[\clZ_{I^r}]{}(\overline z_\infty,\overline u_\infty)\rest{I_r}.\label{exist-weakconvr}
\end{align}
Then, we can also assume that
\begin{align}
(\overline z_k,\overline u_k)\rest{I_r}\xrightharpoonup[\clX_{I^r}]{}(\overline z_\infty,\overline u_\infty)\rest{I_r}\quad\mbox{and}\quad \overline z_k\rest{I_r}\xrightarrow[L^2(I^r,V)]{}\overline z_\infty\rest{I_r},\label{exist-weakconvr2}
\end{align}
where for the later strong convergence we have used the compactness of the embedding~$W((s_0,r),\rmD(A),H)
\xhookrightarrow{}L^2((s_0,r),V)$.

 Following (nontrivial) standard arguments (which we skip here) we can show that $(\overline z_\infty,\overline u_\infty)$ satisfies~\eqref{sys-z-hat}, leading us to the conclusion that~$(\overline z_\infty,\overline u_\infty)\in\fkX_I^{z_0}$ solves Problem~\ref{Pb:OCPI}.
Finally, by~\eqref{bdd-optimcost} we necessarily have  the inequality
\[
\clJ_I^Q(\overline z_\infty,\overline u_\infty)\le\clJ_I^Q(\overline z ,\overline u )\le\widehat C\norm{z_0}{H}^2,
\]
which ends the proof.\qed

\begin{remark}
Note that in Theorem~\ref{T:existOptim}, uniqueness of the minimizer may not hold due to the nonlinearity of the problem (nonlinear constraint).
\end{remark}

%%%%%%%%%%%%%%%%%%%%%%%%%%%%
\subsection{Dynamic programming principle}\label{sS:DPP}
The dynamic programming principle can play an important role in the computation/implementation of optimal controls. For a given nonempty open interval~$I=(s_0,s_1)\subseteq\bbR_+$,
we denote the value function by
\begin{equation}\notag
\fkV^I\colon V \to\bbR,\qquad \fkV^I(z_0)\coloneqq\clJ_{I}^{Q}(z_{I}^{z_0} ,u_{I}^{z_0}),
\end{equation}
with~$(z_{I}^{z_0} ,u_{I}^{z_0} )$ solving Problem~\ref{Pb:OCPI}.

For an element~$a\in I$, let us denote the complementary intervals
\[
I^a=(s_0,a)\quad\mbox{and}\quad I_{a}=(a,s_1).
\]

Next, we recall the dynamic programming principle.
 \begin{lemma}\label{L:DPP}
We have the equivalence
\begin{align}\notag
&(\overline z ,\overline u)\in\argmin\limits_{(z,u)\in\fkX_{I}^{z_0}}\clJ_{I}^{Q}(z,u)
\quad\Longleftrightarrow\quad(\overline z ,\overline u)\rest{I^a}\in\argmin\limits_{(w,v)\in\fkX_{I^a}^{z_0}}\Bigl( \clJ_{I^a}^{Q}(w,v)+\fkV^{I_a}(w(a)) \Bigr).\notag
\end{align}
\end{lemma}
\begin{corollary}\label{C:optim-tail}
For every~$z_0\in V$, every open interval~$I\subseteq\bbR_+$, and every~$a\in I$, we have that given a minimizer~$(z_{I}^{z_0} ,u_{I}^{z_0} )$ for~$\clJ_{I}^Q$, then~$(z_{I}^{z_0} ,u_{I}^{z_0} )\rest{I_a}$ is a minimizer of~$\clJ_{I_a}^Q$.
\end{corollary}

The proofs of Lemma~\ref{L:DPP} and Corollary~\ref{C:optim-tail}, follow by standard arguments. We briefly recall the arguments in the Appendix, sections~\ref{ApxProofL:DPP} and~\ref{ApxProofC:optim-tail}.

%%%%%%%%%%%%%%%%%%%%%%%%%%%%
%%%%%%%%%%%%%%%%%%%%%%%%%%%%
\section{First-order optimality conditions}\label{S:1optcond}
For applications and, in particular, for the computation of approximations of the optimal controls, it is useful to know further properties of such controls.
Here we investigate the existence and properties of associated Lagrange mutipliers.
\begin{lemma}\label{L:fyr-est}
 The mapping~$\rmd f^{y_\ttt}\rest{\overline z}$ defined as
 \begin{equation}\label{dfyr.z}
 \rmd f^{y_\ttt}\rest{\overline z}w=3\overline z^2w +(6y_\ttt+2\xi_2)\overline zw+(3y_\ttt^2+2\xi_2y_\ttt+\xi_1-1)w
 \end{equation}
 satisfies the estimates
\begin{align}
\norm{\rmd f^{y_\ttt}\rest{\overline z}w}{H}&\le C_1\left(\norm{\overline z}{L^6}^2+\norm{y_\ttt}{L^6}^2+1\right)\norm{w}{L^6},\label{bdd-f-opt1}\\
\norm{\rmd f^{y_\ttt}\rest{\overline z}w}{L^2(I,H)}^2
&\le C_3\norm{ w}{L^2(\bbR_+,V)}^2\le C_4\norm{w}{W(I,\rmD(A),H)}^2,\label{fkG-wdef3}
\end{align}
with constants as
\begin{subequations}
\begin{align}
C_1&=C+\norm{1}{L^3}\le\ovlineC{\norm{\zeta}{\infty},\dnorm{\Omega}{}},&C&=\norm{2\xi_2^2+\xi_1-1}{\bbR},\\
C_3&\le\ovlineC{\norm{\zeta}{\infty},C_h,\tfrac1{\tau_h},\norm{\overline z}{W(I,\rmD(A),H)}^2},& C_4&\le\ovlineC{\norm{\zeta}{\infty},C_h,\tfrac1{\tau_h},\norm{\overline z}{W(I,\rmD(A),H)}^2}.
\end{align}
\end{subequations}
\end{lemma}
\begin{proof}
By direct computations, we find that
\begin{align}
\norm{\rmd f^{y_\ttt}\rest{\overline z}w}{H}&=\norm{3\overline z^2w +(6y_\ttt+2\xi_2)\overline zw+(3y_\ttt^2+2\xi_2y_\ttt+\xi_1-1)w}{L^2}\notag\\
&\hspace{-3.5em}\le\norm{3\overline z^2+6y_\ttt\overline z+2\xi_2\overline z+3y_\ttt^2+2\xi_2y_\ttt+\xi_1-1}{L^3}\norm{w}{L^6}\notag\\
&\hspace{-3.5em}\le\norm{7\overline z^2+7y_\ttt^2+2\xi_2^2+\xi_1-1}{L^3}\norm{w}{L^6}
\notag\\
&\hspace{-3.5em}\le C(7+\norm{2\xi_2^2+\xi_1-1}{\bbR})\left(\norm{\overline z^2+y_\ttt^2+1}{L^3}\right)\norm{w}{L^6}\le C\left(\norm{\overline z}{L^6}^2+\norm{y_\ttt}{L^6}^2+\norm{1}{L^3}\right)\norm{w}{L^6}\notag\\
&\hspace{-3.5em}\le C_1\left(\norm{\overline z}{L^6}^2+\norm{y_\ttt}{L^6}^2+1\right)\norm{w}{L^6},\notag
\end{align}
with~$C=\norm{2\xi_2^2+\xi_1-1}{\bbR}$ and~
$C_1=C+\norm{1}{L^3}\le\ovlineC{\norm{\zeta}{\infty},\dnorm{\Omega}{}}$.
Hence, we obtain that
\begin{align}
&\norm{\rmd f^{y_\ttt}\rest{\overline z}w}{L^2(I,H)}^2
\le C_1\norm{ \left(\norm{\overline z}{L^6}^2+\norm{y_\ttt}{L^6}^2+1\right)^2\norm{w}{L^6}^2}{L^1(I,\bbR)}\notag
\end{align}
and using Lemma~\ref{L:estL6free} we have that
\[
\norm{y_\ttt(t)}{L^6}^6
\le\norm{y_\ttt(0)}{L^6}^6+\tfrac{2-\rme^{-\tau_h}}{1-\rme^{-\tau_h}}\breve C_2\left(C_h^6+1\right),
\]
with
$\breve C_2
\le\ovlineC{\norm{\zeta}{\infty}}$, where~$\tau_h>0$ and~$C_h\ge0$ are as in Assumption~\ref{A:h-pers-bound}. Thus
\begin{align}
&\norm{\rmd f^{y_\ttt}\rest{\overline z}w}{L^2(I,H)}^2
\le C_2\norm{ \left(\norm{\overline z}{V}^4+1\right)\norm{w}{V}^2}{L^1(I,\bbR)}\notag
\end{align}
with
$C_2
\le\ovlineC{\norm{\zeta}{\infty},C_h,\tfrac1{\tau_h}}$, where we have also used the
Sobolev embedding~$V\xhookrightarrow{}L^6$.
Next, recalling the embedding
\begin{align}\label{fkG-wdef2}
W(I,\rmD(A),H)\xhookrightarrow{}C(\overline I,V),
\end{align}
we arrive at
\begin{align}
&\norm{\rmd f^{y_\ttt}\rest{\overline z}w}{L^2(I,H)}^2
\le C_3\norm{ w}{L^2(I,V)}^2\le C_4\norm{w}{W(I,\rmD(A),H)}^2,\notag
\end{align}
with~$C_3\le\ovlineC{\norm{\zeta}{\infty},C_h,\tfrac1{\tau_h},\norm{\overline z}{W(I,\rmD(A),H)}^2}$ and~$C_4\le\ovlineC{\norm{\zeta}{\infty},C_h,\tfrac1{\tau_h},\norm{\overline z}{W(I,\rmD(A),H)}^2}$.
\end{proof}

\begin{lemma}\label{L:dG}
For every~$z_0\in V$ and an open interval~$I\subseteq\bbR_+$, the mapping~$\clG_{I}^{z_0}$ in~\eqref{opt-setting} is Fr\'echet differentiable, with derivative~$\rmd \clG_{I}^{z_0}\rest{(\overline z,\overline u)}\in\clL(\clX_{I},\clY_{I})$, at~$(\overline z,\overline u)\in\clX_{I}$,  given by
\[
(w,u)\mapsto\rmd \clG_{I}^{z_0}\rest{(\overline z,\overline u)}(w,u)=\Bigl(\dot w +Aw+\rmd f^{y_\ttt}\rest{\overline z}w -U_M^\diamond u, \, w(s_0)\Bigr),
\]
where~$\rmd f^{y_\ttt}\rest{\overline z}$ is as in~\eqref{dfyr.z}.
\end{lemma}
We give the proof in the Appendix, section~\ref{ApxProofL:dG}.

\begin{definition}\label{D:locoptim}
Given~$z_0\in V$, a pair~$(\overline z ,\overline u )\in \fkX_{I}^{z_0}$ is called local solution for Problem~\ref{Pb:OCPI} if $\clJ_I^Q(\overline z,\overline u )=\min\{\clJ_I^Q(z,u)\mid (z,u)\in\fkX_{I}^{z_0}\mbox{ and }\norm{(z,u)-(\overline z ,\overline u )}{\clX_I}<r\}$
for some~$r>0$.
\end{definition}

Next we recall the following result which follows from optimization theory in Banach spaces that we find in~\cite[Thm~3.1; Eqs.~(1.1) and~(1.4)]{ZoweKurcyusz79}.
\begin{lemma}\label{L:ZoweKurcyusz}
If~$(\overline z,\overline u)\in\clX$ is a local solution for Problem~\ref{Pb:OCPI} which satisfies the relation
$
\overline\bbR_+\rmd \clG_I^{z_0}\rest{(\overline z,\overline u)}\left(\clC_I-(\overline z,\overline u)\right)=\clY_I,
$
then there exists a Lagrange multiplier~$\ell\in\clY_I'=L^2(I,H)\times V'$ such that
\begin{align}
&\rmd \clJ_I^Q\rest{(\overline z,\overline u)}-\ell\circ \rmd \clG_I^{z_0}\rest{(\overline z,\overline u)}\in\{\xi\in \clX_I'\mid \langle\xi,c-(\overline z,\overline u)\rangle_{\clX_I',\clX_I}\ge0\mbox{ for all }c\in\clC_I\}.\notag
\end{align}
Above~$\clC_I-(\overline z,\overline u)$ stands for~$\{c-(\overline z,\overline u)\mid c\in \clC_I\}\subseteq\clX_I$.
\end{lemma}

\begin{remark}
In~\cite{ZoweKurcyusz79} a more general constraint~$\clG(z,u)\in K$ for a convex cone with vertex at~$(0,0)$ is considered. In our setting we have the particular case~$K=\{(0,0)\}$.
\end{remark}

The interior of a subset~$S$ of~$\clY_I$ will be denoted~${\bf int}[S]$.
\begin{lemma}[{\cite[sect.~3]{ZoweKurcyusz79}}]\label{L:ZoweKurcyusz-int}
Let~$(\overline z,\overline u)\in\clX_I$. Then, the relation~$
0\in{\bf int} [\rmd \clG_I^{z_0}\rest{(\overline z,\overline u)}\left(\clC_I-(\overline z,\overline u)\right)]
$
implies that
$
\overline\bbR_+\rmd \clG_I^{z_0}\rest{(\overline z,\overline u)}\left(\clC_I-(\overline z,\overline u)\right)=\clY_I.
$
\end{lemma}

We shall use  the following asymptotic behavior results.
\begin{lemma}\label{L:Wwkconv0}
Let~$I=(s_0,+\infty)$. Then, every function~$w\in W(I,V,V')$ satisfies the asymptotic limit~$\lim\limits_{t\to+\infty}\norm{ w (t)}{H}^2=0$, $t\in I$.
\end{lemma}
\begin{proof}
 We follow a slight variation of the arguments in~\cite[Proof of~(2.8), Thm~2.4]{CasasKun17}. From~\cite[Ch.~3, Sect.~1.4, Lem.~1.2]{Temam01} we can write~$\frac{\rmd}{\rmd t}\norm{w}{H}^2=2\langle \dot w,w\rangle_{V',V}$. Since~$ w\in L^2(I,V)$ there is a sequence $(t_n)_{n\in\bbN}$, $t_n\in I$, such that $t_n\xrightarrow{}+\infty$ and~$\norm{w(t_n)}{V}^2\le\frac1n$. We can conclude by following~\cite[Proof of~(2.8), Thm~2.4]{CasasKun17}.
\end{proof}

\begin{corollary}\label{C:Wstconv0}
For every~$\alpha\in\bbR$ and~$z\in W(\bbR_+,\rmD(A^{\alpha+\frac12}),\rmD(A^{\alpha-\frac12}))$ we have that $\lim\limits_{t\to+\infty}\norm{ z (t)}{\rmD(A^{\alpha})}=0$.
\end{corollary}
\begin{proof}
Observe that~$A^{\alpha} z  \in W(\bbR_+,V,V')$, thus we can use Lemma~\ref{L:Wwkconv0} to conclude that
$
\lim\limits_{t\to+\infty}\norm{z (t)}{\rmD(A^{\alpha})}=\lim\limits_{t\to+\infty}\norm{A^{\alpha}  z (t)}{H}=0.
$
\end{proof}

Now we present a first property of optimal pairs.
\begin{lemma}\label{L:ZoweKurcyusz-0inint}
For every~$(z_\bullet,u_\bullet)\in\fkX_{I}^{z_0}$, we have~$
0\in{\bf int} [\rmd \clG_{I}^{z_0}\rest{(z_\bullet,u_\bullet)}\left(\clC_{I}-(z_\bullet,u_\bullet)\right)].
$
\end{lemma}
\begin{proof}
For simplicity let us denote~$\clX\coloneqq\clX_{I}^{z_0}$, $\clY\coloneqq\clY_{I}^{z_0}$, $\clG\coloneqq\clG_{I}^{z_0}$, and~$\clC\coloneqq\clC_I$.
Let us fix~$\varepsilon>0$, and consider the ball~$\fkB_\varepsilon^{\clY}\coloneqq\fkB_\varepsilon^{\clY}(0)
=\{g\in\clY\mid\norm{g}{\clY}\le\varepsilon\}$. We shall show that for small enough~$\varepsilon$ we have that~$\fkB_\varepsilon^{\clY}\subseteq\rmd \clG\rest{( z_\bullet  , u_\bullet )}\left(\clC-( z_\bullet  ,u_\bullet )\right)$.
For arbitrary~$g=(\eta,w_0)\in\fkB_\varepsilon^{\clY}$, we search for~$c=(w_c, v_c)\in\clC$ such that $\rmd \clG\rest{( z_\bullet  ,u_\bullet )}(w_c- z_\bullet  ,v_c-u_\bullet )=(\eta,w_0)$.
That is, we look for~$c=(w_c, v_c)\in\clX$ such that
\begin{align}
 &\tfrac{\rmd}{\rmd t}({w_c}-{ z_\bullet  }) +A(w_c- z_\bullet  )+\rmd f^{y_\ttt}\rest{z_\bullet}(w_c- z_\bullet  )=U_M^\diamond (v_c-u_\bullet )+\eta,\notag\\
&w_c(s_0)- z_\bullet  (s_0)=w_0\in V,\qquad \dnorm{v_c(t)}{}\le C_u,\qquad 0\le s_0=\inf I.\notag
 \end{align}

Denoting~$w\coloneqq w_c- z_\bullet  $, we see that finding~$(w_c,v_c)$ is  equivalent to finding~$( w,v_c)$, which solves
\begin{subequations}\label{sys-w1}
\begin{align}
 &\dot w +Aw+\rmd f^{y_\ttt}\rest{z_\bullet}w -U_M^\diamond v_c+U_M^\diamond u_\bullet =\eta,\\
&w(s_0)=w_0\in V,\qquad \dnorm{v_c(t)}{}\le  C_u.
 \end{align}
\end{subequations}

Hence, it is enough to show that the solution of~\eqref{sys-w1}
satisfies~$w\in W(I,\rmD(A),H)$.
 Multiplying the dynamics in~\eqref{sys-w1} by
$2w$ we obtain
\begin{align}
\tfrac{\rmd}{\rmd t}\norm{w}{H}^2 &=-2\norm{w}{V}^2+2(\rmd f^{y_\ttt}\rest{z_\bullet}w,w)_{H} +2(U_M^\diamond v_c,w)_{H}-2(U_M^\diamond u_\bullet  ,w)_{H}+2(\eta,w)_{H},\notag
 \end{align}
and from straightforward computations we find that
\begin{align}
&-2(\rmd f^{y_\ttt}\rest{z_\bullet}w,w)_{H}=-2(3z_\bullet^2w +(6y_\ttt+2\xi_2) z_\bullet  w+(3y_\ttt^2+2\xi_2y_\ttt+\xi_1-1)w,w)_{H}\notag\\
&\hspace{3em}=2(-3( z_\bullet  +y_\ttt)^2 -2\xi_2( z_\bullet  +y_\ttt)-\xi_1+1,w^2)_{H}\le C_0\norm{w}{H}^2\notag
\end{align}
with $C_0=\ovlineC{\norm{\zeta}{\infty}}$ as in~\eqref{varpi}.
Thus,
\begin{align}
\tfrac{\rmd}{\rmd t}\norm{w}{H}^2 &=-2\norm{w}{V}^2+(3+C_0)\norm{w}{H}^2+\norm{\eta}{H}^2 +2(U_M^\diamond v_c,w)_{H}-2(U_M^\diamond u_\bullet  ,w)_{H}.\label{dtw-vcbullu}
 \end{align}
 Multiplying the dynamics of~\eqref{sys-w1} by~$2Aw$ we find
\begin{align}
 \tfrac{\rmd}{\rmd t}\norm{w}{V}^2&=-2\norm{Aw}{H}^2-2(\rmd f^{y_\ttt}\rest{z_\bullet}w,Aw)_{H} +2(U_M^\diamond v_c-U_M^\diamond u_\bullet+\eta,Aw)_{H}^2\notag\\
&\le-\norm{Aw}{H}^2+2\norm{\rmd f^{y_\ttt}\rest{z_\bullet}w}{H}^2+2\norm{U_M^\diamond v_c-U_M^\diamond u_\bullet+\eta}{H}^2\notag
 \end{align}
and time integration gives us
\begin{align}
 \norm{w}{L^2(I,\rmD(A))}^2&\le\norm{w_0}{V}^2+2\norm{\rmd f^{y_\ttt}\rest{z_\bullet}w}{L^2(I,H)}^2 +2\norm{U_M^\diamond v_c-U_M^\diamond u_\bullet+\eta}{L^2(I,H)}^2\notag
 \end{align}
 and recalling~\eqref{fkG-wdef3},
\begin{align}
 2\norm{w}{L^2(I,\rmD(A))}^2&\le\norm{w_0}{V}^2+2(1+C_3)\norm{w}{L^2(I,V)}^2+2\norm{U_M^\diamond v_c-U_M^\diamond u_\bullet+\eta}{L^2(I,H)}^2\label{boundL2DA-1}
\end{align}
with~$C_3\le\ovlineC{\norm{\zeta}{\infty},C_h,\tfrac1{\tau_h},\norm{\overline z}{W(I,\rmD(A),H)}^2}$.

We consider separately the cases of bounded~$I$ and unbounded~$I$.

\smallskip
\noindent$\bullet$\hspace{.75em}\emph{The case of bounded time interval}. In the case $I=(s_0,s_1)$ with~$0<s_0<s_1\in\bbR_+$, we choose~$v_c(t)=u_\bullet(t)$. Then by~\eqref{dtw-vcbullu}, together with the Gronwall inequality and time integration, it follows that~$\norm{w}{L^2(I,V)}^2<+\infty$, thus we can use~\eqref{boundL2DA-1} to conclude that~$\norm{w}{L^2(I,\rmD(A))}^2<+\infty$. Using the dynamics and~\eqref{sys-w1} and~\eqref{fkG-wdef3} we find
\begin{align}
 \norm{\dot w}{L^2(I,H)} &\le \norm{w}{L^2(I,\rmD(A))}+\norm{w}{L^2(I,H)} +C_3^\frac12\norm{w}{L^2(I,V)}+\norm{\eta}{L^2(I,H)}<+\infty.\notag
 \end{align}
That is, we have~$w\in W(I,\rmD(A),H)$ solving~\eqref{sys-w1}, which finishes the proof for bounded~$I$.

Furthermore, we can see that in this case we can take an arbitrary~$\varepsilon>0$, that is, $\rmd \clG_{I}^{z_0}\rest{(z_\bullet,u_\bullet)}\left(\clC_{I}-(z_\bullet,u_\bullet)\right)=\clY$.

\smallskip
\noindent$\bullet$\hspace{.75em}\emph{The case of unbounded time interval}. In the case $I=(s_0,+\infty)$,
we are going to take~$v_c$ in concatenated form as follows.
We start by defining the constant
\begin{align} \label{delta*}
\delta\coloneqq\tfrac{C_u}{2\lambda}\norm{(U_M^\diamond)^{-1}P_{\clU_M}^{\widetilde \clU_M}A P_{\widetilde \clU_M}^{\clU_M}}{\clL(H)}^{-1}
\end{align}
and, subsequently, we use Corollary~\ref{C:Wstconv0} and choose
\begin{align} \label{s*}
\overline s\ge s_0\quad\mbox{such that}\quad \norm{z_\bullet(t)}{V}^2\le\delta^2\widehat C^{-1}\norm{U_M^\diamond}{\clL(\bbR^{M_\sigma},H)}^{-2}\quad\mbox{for all}\quad t\ge\overline s,
\end{align}
where~$\widehat C$ is as in Theorem~\ref{T:existOptim}. Finally, we choose
\begin{align} \label{eps*}
\varepsilon\coloneqq\delta\left(1+(\overline s-s_0)(4+C_0)\rme^{(2+C_0)(\overline s-s_0) }\right)^{-\frac12},
\end{align}
for the radius of the ball~$\fkB_\varepsilon^\clY$, where~$C_0$ is as in~\eqref{varpi}.

Now, we define the concatenated control
\begin{align}
v_c(t)\coloneqq \fkK_{\overline s}w(t)\coloneqq\begin{cases}
u_\bullet(t),\quad&\mbox{if }t<\overline s;\\
\overline\clK_M^\lambda w(t),\quad&\mbox{if }t\ge \overline s.
\end{cases}\label{fkK}
\end{align}
which, together with~\eqref{dtw-vcbullu} and Lemma~\ref{L:Ku-monot}, give us
\begin{align}
\tfrac{\rmd}{\rmd t}\norm{w}{H}^2
 &\le-2\norm{w}{V}^2+(3+C_0)\norm{w}{H}^2+\norm{\eta}{H}^2,
&&\quad\mbox{for}\quad t<\overline s;\label{dtw-bull-smtime}\\
\tfrac{\rmd}{\rmd t}\norm{w}{H}^2
 &\le-2\norm{w}{V}^2+(4+C_0)\norm{w}{H}^2+\norm{U_M^\diamond u_\bullet}{H}^2+\norm{\eta}{H}^2\notag\\
&\quad-2\lambda\min\left\{1,\tfrac{C_u}{\dnorm{v}{}}\right\}\norm{P_{\widetilde \clU_M}^{\clU_M}w}{V}^2,
&&\quad\mbox{for}\quad t\ge\overline s;\label{dtw-bull-latime}
\end{align}
where~$v(t)\coloneqq-\lambda(U_M^\diamond)^{-1}P_{\clU_M}^{\widetilde \clU_M}A P_{\widetilde \clU_M}^{\clU_M}w(t)$ and
where~$\min\left\{1,\tfrac{C_u}{\dnorm{v}{}}\right\}\coloneqq1$ if~$\dnorm{v}{}=0$. Note also that
$\dnorm{v}{}=0$ is equivalent to~$w=0$.
Inequality~\eqref{dtw-bull-smtime} and~$\norm{w}{H}\le\norm{w}{V}$ imply that
\begin{align}
\tfrac{\rmd}{\rmd t}\norm{w}{H}^2
 &\le(1+C_0)\norm{w}{H}^2+\norm{\eta}{H}^2,\quad\mbox{for}\quad t\in(s_0,\overline s).\notag
 \end{align}
By  Gronwall's inequality and using that ~$(\eta,w_0)\in\fkB_\varepsilon^{\clY}$, we obtain
\begin{align}
\norm{w(t)}{H}^2
 &\le\rme^{(1+C_0)(t-s_0)}\left(\norm{w(s_0)}{H}^2+\norm{\eta}{(I^{\overline s},H)}^2\right)
\le\rme^{(1+C_0)(t-s_0)}\varepsilon^2,\quad\mbox{for}\quad t\in(s_0,\overline s).\notag
 \end{align}

Time integration of~\eqref{dtw-bull-smtime} now implies
\begin{align}
\norm{w(\overline s)}{H}^2+2\norm{w}{L^2((s_0,\overline s),V)}^2&\le \norm{w_0}{H}^2+(3+C_0)\norm{w}{L^2((s_0,\overline s),H)}^2+\varepsilon^2\notag\\
&\hspace{0em}\le \left(1+(\overline s-s_0)(3+C_0)\rme^{(1+C_0)(\overline s-s_0)}\right)\varepsilon^2
\le\delta^2.\label{dtw-bull3}
 \end{align}

From~\eqref{dtw-bull-latime}, and by recalling the choices of~$M\ge M_*$ and~$\lambda\ge \lambda_*$, see ~\eqref{lamMvarpi} and~\eqref{varpi}, we find that, for~$t\ge\overline s$,
\begin{align}
\tfrac{\rmd}{\rmd t}\norm{w}{H}^2
&\le-\norm{w}{V}^2-2\mu\norm{w}{H}^2+2\lambda\norm{P_{\widetilde \clU_M}^{\clU_M}w}{V}^2+\norm{U_M^\diamond u_\bullet  }{H}^2+\norm{\eta}{H}^2,\notag\\
&\quad-2\lambda\min\left\{1,\tfrac{C_u}{\dnorm{v}{}}\right\}\norm{P_{\widetilde \clU_M}^{\clU_M}w}{V}^2,
\notag
\end{align}
 and from the implications
$
\norm{w(t)}{H}\le 2\delta \Rightarrow\dnorm{v(t)}{}\le C_u \Rightarrow
\min\left\{1,\tfrac{C_u}{\dnorm{v}{}}\right\}=1,
$
we obtain
\begin{align}
\tfrac{\rmd}{\rmd t}\norm{w}{H}^2 &\le-\norm{w}{V}^2-2\mu\norm{w}{H}^2+\norm{U_M^\diamond u_\bullet  }{H}^2+\norm{\eta}{H}^2,
\quad\mbox{while }\norm{w(t)}{H}\le 2\delta,\; t\ge\overline s.
\label{dtw-bull4}
 \end{align}

By Theorem~\ref{T:existOptim} and Corollary~\ref{C:optim-tail}, with~$I_{\overline s}=(\overline s,+\infty)$, we obtain
\begin{align}\notag
\norm{u_\bullet}{L^2(I_{\overline s},\bbR^{M_\sigma})}^2
\le \clJ_{I_{\overline s}}^Q(z_{I_{\overline s}}^{z_\bullet(\overline s)},u_{I_{\overline s}}^{z_\bullet(\overline s)})
\le \widehat C\norm{z_\bullet(\overline s)}{V}^2,
\end{align}
which together with~\eqref{s*}, \eqref{dtw-bull4}, and~$\norm{w}{V}\ge\norm{w}{H}$, lead us to
\begin{align}  \notag
\tfrac{\rmd}{\rmd t}\norm{w}{H}^2
&\le-\norm{w}{H}^2+\widehat C\norm{U_M^\diamond
}{\clL(\bbR^{M_\sigma},H)}^2\norm{z_\bullet(\overline s)}{V}^2+\norm{\eta}{H}^2\\
&\le-\norm{w}{H}^2+\delta^2+\norm{\eta}{H}^2,
\quad\mbox{while}\quad \norm{w(t)}{H}\le 2\delta,\qquad t\ge\overline s,\notag
\end{align}
  Next, the Gronwall inequality, \eqref{dtw-bull3}, and~\eqref{eps*}, give us
\begin{align}
\norm{w(t)}{H}^2
&\le\rme^{-(t-\overline s)}\norm{w(\overline s)}{H}^2+\delta^2+\norm{\eta}{L^2(I_{\overline s},H)}^2\notag\\
&\le2\delta^2+\varepsilon^2\le3\delta^2,
\quad\mbox{while}\quad \norm{w(t)}{H}\le 2\delta,\qquad t\ge\overline s.\notag
 \end{align}
Therefore, we can conclude that
\begin{align}
\norm{w(t)}{H}^2\le3\delta^2,
\quad\mbox{for all}\quad t\ge\overline s\notag
 \end{align}
and, by time integration of~\eqref{dtw-bull4}, we also find that
\begin{align}
\norm{w}{L^2(I_{\overline s},V)}^2 &\le\norm{w(\overline s)}{H}^2+\norm{U_M^\diamond u_\bullet  }{L^2(I_{\overline s},H)}^2+\norm{\eta}{L^2(I_{\overline s},H)}^2<+\infty.\label{winL2RV}
 \end{align}

Multiplying the dynamics of~\eqref{sys-w1} by~$2Aw$ we find
\begin{align}
 \tfrac{\rmd}{\rmd t}\norm{w}{V}^2&=-2\norm{Aw}{H}^2-2(\rmd f^{y_\ttt}\rest{z_\bullet}w,Aw)_{H}+2(U_M^\diamond \fkK_{\overline s}w-U_M^\diamond u_\bullet+\eta,Aw)_{H}^2\notag\\
&\le-\norm{Aw}{H}^2+2\norm{\rmd f^{y_\ttt}\rest{z_\bullet}w}{H}^2+2\norm{U_M^\diamond \fkK_{\overline s}w-U_M^\diamond u_\bullet+\eta}{H}^2\notag
 \end{align}
and time integration gives us
\begin{align}
 2\norm{w}{L^2(I,\rmD(A))}^2&\le\norm{w_0}{V}^2+2\norm{\rmd f^{y_\ttt}\rest{z_\bullet}w}{L^2(I,H)}^2 +2\norm{U_M^\diamond \fkK_{\overline s}w-U_M^\diamond u_\bullet+\eta}{L^2(I,H)}^2\notag
 \end{align}
and recalling~\eqref{fkK} and~\eqref{fkG-wdef3},
\begin{align}
 2\norm{w}{L^2(I,\rmD(A))}^2&\le\norm{w_0}{V}^2+2C_3\norm{ w}{L^2(I,V)}^2+6\norm{U_M^\diamond \fkK_{\overline s}w}{L^2(I\setminus(s_0,\overline s],H)}^2\notag\\
&\quad +6\norm{U_M^\diamond u_\bullet}{L^2(I\setminus(s_0,\overline s],H)}^2+6\norm{\eta}{L^2(	I,H)}^2\label{infboundL2DA-1}
\end{align}
with~$C_3\le\ovlineC{\norm{\zeta}{\infty},C_h,\tfrac1{\tau_h},\norm{\overline z}{W(I,\rmD(A),H)}^2}$. By Corollary~\ref{C:optim-tail}, \eqref{dtw-bull3}, and~\eqref{winL2RV}, we obtain
\begin{align}
 2\norm{w}{L^2(I,\rmD(A))}^2
&\le\norm{w_0}{V}^2+(2C_3+6\norm{U_M^\diamond\overline \clK_M^\lambda}{\clL(H)}^2)\norm{w}{L^2(I,V)}^2+6\norm{\eta}{L^2(\bbR_+,H)}^2\notag\\
&\quad+6\widehat C\norm{U_M^\diamond}{\clL(\bbR^{M_\sigma},H)}^2\norm{z_\bullet(\overline s)}{V}^2<+\infty.\label{infboundL2DA}
 \end{align}

Now,  from~\eqref{sys-w1} and~\eqref{fkG-wdef3}, we find
\begin{align}
 \norm{\dot w}{L^2(I,H)} &\le \norm{w}{L^2(I,\rmD(A))}+C_3^\frac12\norm{w}{L^2(I,V)}  +\norm{U_M^\diamond \fkK_{\overline s}w-U_M^\diamond u_\bullet +\eta}{L^2(I,H)}<+\infty.\notag
 \end{align}
Hence, $w\in W(I,\rmD(A),H)$, which finishes the proof in the case of unbounded~$I$.
\end{proof}

%%%%%%%%%%%%%%%%%%%%%%%%%%%%
\subsection{The adjoint equation}\label{sS:Adjoint}
Given a local solution~$(z_\bullet,u_\bullet)=(z_I^{z_0},u_I^{z_0})$ to the optimal control Problem~\ref{Pb:OCPI}, from Lemmas~\ref{L:ZoweKurcyusz}, ~\ref{L:ZoweKurcyusz-int}, and~\ref{L:ZoweKurcyusz-0inint}, it follows that
there exists a~$\ell\in\clY_I'=L^2(I,H)\times V'$ such that~$\rmd \clJ_I^Q\rest{(z_\bullet,u_\bullet)}-\ell\circ \rmd \clG_I^{z_0}\rest{(z_\bullet,u_\bullet)}\in \clX_I'$ and
\begin{align}
&\left\langle\rmd \clJ_I^Q\rest{(z_\bullet,u_\bullet)}-\ell\circ \rmd \clG_I^{z_0}\rest{(z_\bullet,u_\bullet)},c-(z_\bullet,u_\bullet)\right\rangle_{\clX_I',\clX_I}\ge 0\mbox{ for all } c\in\clC_I.\label{lag.gen}
\end{align}
In particular, setting~$c=(v+z_\bullet,u_\bullet)$ we find that
\begin{align}
&\left\langle\rmd \clJ_I^Q\rest{(z_\bullet,u_\bullet)}-\ell\circ \rmd \clG_I^{z_0}\rest{(z_\bullet,u_\bullet)},(v,0)\right\rangle_{\clX_I',\clX_I}\ge 0\mbox{ for all } v\in W(I,\rmD(A),H),\notag
\end{align}
which gives us, for~$\ell\eqqcolon (p,\xi)\in L^2(I,H)\times V'$,  and recalling Lemma~\ref{L:dG},
\begin{align}
0&\le (Qz_\bullet,Qv)_{L^2(I,H)}+(u_\bullet,0)_{L^2(I,\bbR^{M_\sigma})}\notag\\
&\quad-\left( p,\dot v +Av+\rmd f^{y_\ttt}\rest{z_\bullet}v -U_M^\diamond 0\right)_{L^2(I,H)}-\left\langle\xi, \, v(s_0)\right\rangle_{V',V}\notag\\
&= \langle Q^*Qz_\bullet,v\rangle_{L^2(I,\rmD(A)'),L^2(I,\rmD(A))}-\left\langle\xi, \, v(s_0)\right\rangle_{V',V}-\left( p,\dot v +Av+\rmd f^{y_\ttt}\rest{z_\bullet}v \right)_{L^2(I,H)}\notag\\
&\mbox{ for all } v\in W(I,\rmD(A),H),\mbox{ with }v\rest{I_T}=0, \label{lagr.gen-v}
\end{align}
where we recall that~$I_T=\{r\in I\mid r> T\}$, $T> s_0=\inf I$. In particular, we find that
\begin{align}
0&\le -\left\langle -Q^*Qz_\bullet-\dot p +Ap+\rmd f^{y_\ttt}\rest{z_\bullet} p,\pm v \right\rangle_{W(I,\rmD(A),H)',W(I,\rmD(A),H)}\notag\\
&\mbox{ for all } v\in W(I,\rmD(A),H),\mbox{ with } v(s_0)=0\mbox{ and } v\rest{I_T}=0.\notag
\end{align}
which implies that
\begin{align}\label{dotp1}
0&= -\dot p+Ap+\rmd f^{y_\ttt}\rest{z_\bullet} p -Q^*Qz_\bullet \in W(I^T,\rmD(A),H)',\mbox{ for all } T\in I,
\end{align}
where~$I^T=\{r\in I\mid r< T\}$.
 Note that by~\eqref{dfyr.z}, $\rmd f^{y_\ttt}\rest{z_\bullet}pv=\rmd f^{y_\ttt}\rest{z_\bullet}vp$, and by~\eqref{fkG-wdef3}
$\norm{\rmd f^{y_\ttt}\rest{\overline z}p}{L^2(I^T,V')}^2
\le C_3\norm{p}{L^2(I^T,H)}^2.
$
As~$T\to\sup I$, we find~$\norm{\rmd f^{y_\ttt}\rest{\overline z}p}{L^2(I,V')}^2
\le C_3\norm{p}{L^2(I,H)}^2$, which together with~\eqref{dotp1} lead us to~$p\in W(I,H,\rmD(A^{-1}))$. Hence, Corollary~\ref{C:Wstconv0} gives us~$\lim\limits_{t\to+\infty}\norm{p(t)}{\rmD(A^{-\frac12})}=0$ in the case~$\sup I=(s_0,+\infty)$.

Combining~\eqref{lagr.gen-v} with~\eqref{dotp1} we obtain
\[
0\le-\left\langle\xi, \,  v(s_0)\right\rangle_{V',V}+\left\langle p(s_0), \,  v(s_0)\right\rangle_{V',V}
\mbox{ for all } v\in W(I,\rmD(A),H),\mbox{ with }v\rest{I_T}=0,
\]
which gives us~$p(s_0)=\xi$. In the case~$s_1=\sup I<+\infty$ from~\eqref{lagr.gen-v} and~\eqref{dotp1} we
also find, with~$T=s_1$,
\[
0\le-\left\langle p(s_1),v(s_1)\right\rangle_{V',V}
\mbox{ for all } v\in W(I,\rmD(A),H),
\]
which gives us~$p(s_1)=0$.

Next, for simplicity, let us denote the convex set
\[
\bfC\coloneqq\{u\in L^2(I,\bbR^{M_\sigma})\mid \dnorm{u(t)}{}\le C_u,\;t\in I\}.
\]
Then, we set~$c=(z_\bullet,u)$ in~\eqref{lag.gen}, with~$u\in\bfC$, which leads us to
\begin{align}
0&\le\left\langle\rmd \clJ_I^Q\rest{(z_\bullet,u_\bullet)}-\ell\circ \rmd \clG_I^{z_0}\rest{(z_\bullet,u_\bullet)},(0,u-u_\bullet)\right\rangle_{\clX_I',\clX_I}\notag\\
&=(u_\bullet,u-u_\bullet)_{L^2(I,\bbR^{M_\sigma})}+(p, U_M^\diamond (u-u_\bullet))_{L^2(I,H)}\notag\\
&
=( u_\bullet+(U_M^\diamond)^*p,u-u_\bullet)_{L^2(I,\bbR^{M_\sigma})},
\quad\mbox{for all}\quad u\in\bfC,\notag
\end{align}
which implies that
\begin{align}
(-(U_M^\diamond)^*p-u_\bullet,u-u_\bullet)_{L^2(I,\bbR^{M_\sigma})}\le0,
\quad\mbox{for all}\quad u\in \bfC,\notag
\end{align}

By~\cite[Lem.~1.10, Sect.~1.7]{HinzePinUlbrUlbr09} (see also~\cite[Def.~3.8 and Thm.~3.16]{BauschkeComb17}) it follows that
\begin{align}
u_\bullet=\argmin_{\overline u\in\bfC}\norm{\overline u+(U_M^\diamond)^*p}{L^2(I,\bbR^{M_\sigma})}.\label{ubullet-p}
\end{align}

Next we will explore the fact that our control constraint is pointwise in time, namely,
$\dnorm{u(t)}{}\le C_u$ for (almost) every $t\in I$. For this purpose  we denote the convex set
\[
\bfC_u\coloneqq\{v\in\bbR^{M_\sigma}\mid \dnorm{v}{}\le C_u\}.
\]
and the projection
\begin{align}
\bfP_{\bfC_u}\colon\bbR^{M_\sigma}\to\bfC_u,\qquad
\bfP_{\bfC_u}(v)\coloneqq\argmin_{\overline v\in\bfC_u}\norm{\overline v-v}{\bbR^{M_\sigma}}.\label{convProj}
\end{align}

Now, standard arguments lead us to
\begin{align}
u_\bullet(t)=-\bfP_{\bfC_u}\!\left((U_M^\diamond)^*p(t)\right),\quad\mbox{for a.e. } t\in I.\label{ubullet-pt}
\end{align}
Therefore,
every local solution~$(z_\bullet,u_\bullet)$ of Problem~\ref{Pb:OCPI} in the time interval~$I=(s_0,s_1)\subseteq(0,+\infty)$ satisfies, together with an associated adjoint state~$p$, the coupled system
\begin{equation}\label{optimality-system}\boxed{
\begin{array}{ll}
 \dot z_\bullet =-A z_\bullet -f^{y_\ttt}(z_\bullet) +U_M^\diamond u_\bullet,&\quad z_\bullet(s_0)=z_0,\vspace*{.5em}\\
\dot p=Ap+\rmd f^{y_\ttt}\rest{z_\bullet} p -Q^*Qz_\bullet, &\quad p(s_0)=\xi,\qquad \norm{p(s_1)}{V'}=0,\vspace*{.5em}\\
u_\bullet(t)=-\bfP_{\bfC_u}\!\left((U_M^\diamond)^*p(t)\right).
 \end{array}
 }
\end{equation}
where, in the case~$s_1=+\infty$, the relation
\[
\norm{p(+\infty)}{V'}=0\quad\mbox{is to be understood as}\quad \lim\limits_{t\to+\infty}\norm{p(t)}{V'}=0.
\]

%%%%%%%%%%%%%%%%%%%%%%%%%%%%
\subsection{The adjoint equation for~$Q\in\clL(H)$}\label{sS:Adjoint-st}
Note that with~$Q\in\clL(\rmD(A),H)$ we will have that~$Q^*Q\in\clL(\rmD(A),\rmD(A)')$ and we will not, in general, be able to increase the spatial regularity of~$p\in L^2(I,H)$. For that we will need to take~$Q$ in a smaller space. Here, we consider the case that~$Q$  as in Assumption~\ref{A:Q} satisfies in addition~$Q\in\clL(H)$. This extra requirement  allows us to derive strong regularity for the solutions of the adjoint equation.
First of all note that by~\eqref{fkG-wdef3},
 for~$q\in L^2(I,V)$,
\begin{align}
\norm{\rmd f^{y_\ttt}\rest{z_\bullet}q}{L^2(I,H)}\le C_3\norm{q}{L^2(I,V)}.\label{bdd-f-opt1Phi}
\end{align}

Next, observe that, since~$p\in L^2(I,H)$,  for every~$T\in I$, there exists~$t_T\in I$ such that~$t_T>T$ and~$p(t_T)\in L^2$.
Since~$Q^*Q\in\clL(H)$, standard estimates and the Gronwall inequality allow us to conclude the existence of a weak solution, that is, $p\in W((s_0,t_T),V,V')$. Furthermore,  we can combine the smoothing property for linear parabolic equations to conclude that~$p(T)\in V$, and subsequently we can conclude the existence of a strong solution
\begin{subequations}\label{p-regular-loc}
\begin{align}
&p\in W((s_0,T),\rmD(A),H),\quad\mbox{for all}\quad T\in I,\label{p-regular-loc-reg}
\intertext{and}
&0= -\dot p+Ap+\rmd f^{y_\ttt}\rest{z_\bullet} p -Q^*Qz_\bullet \in L^2((s_0,T),H).
\end{align}
\end{subequations}
Taking  an arbitrary~$v_0\in V$ and~$v\in W(I,\rmD(A),H)$ with~$v(s_0)=v_0$ and~$v\rest{I_T}=0$, by~\eqref{lagr.gen-v} and~\eqref{p-regular-loc}, it follows that
\begin{align}
0&\le (p(s_0),v_0)_{H}-\left\langle\xi, \, v_0\right\rangle_{V',V},\label{arbv-loc}
\end{align}
which gives us~$0\le\left\langle p(s_0)-\xi,v_0\right\rangle_{V',V}$ for all~$v_0\in V$, and implies
\begin{align}
\xi=p(s_0)\in V.\label{xiinV}
\end{align}

We know, from~\eqref{optimality-system} that in the case~$s_1=\sup I<+\infty$ we have~$p(s_1)=0$, and in the case~$s_1=+\infty$ we have~$\lim\limits_{t\to+\infty}\norm{p(t)}{V'}=0$.
We show next that this asymptotic limit also holds in $V$-norm.
For this purpose, let~$s>s_0$ and denote~$\breve v(t)\coloneqq v(s+1-t)$, with~$t\in[0,1]$. The dynamics of~$q(t)\coloneqq t\breve  p(t)$ leads us to
\begin{align}
\dot q&=-Aq-\rmd f^{\breve y_\ttt}\rest{\breve z_\bullet}q +tQ^*Q\breve z_\bullet+\breve p,\notag\\
\tfrac{\rmd}{\rmd t}\norm{q}{V}^2&=-2\norm{q}{\rmD(A)}^2-2(\rmd f^{\breve y_\ttt}\rest{\breve z_\bullet}q -tQ^*Q\breve z_\bullet+\breve p,Aq)_{H}\notag\\
&\le\tfrac12\norm{\rmd f^{\breve y_\ttt}\rest{\breve z_\bullet}q -tQ^*Q\breve z_\bullet+\breve p}{H}^2
\le C_3\varPhi\norm{q}{V}^2 +\norm{-tQ^*Q\breve z_\bullet+\breve p}{H}^2\notag
\end{align}
with~$\varPhi$ as in\eqref{bdd-f-opt1Phi}, and the Gronwall inequality gives us
\begin{align}
\norm{p(s)}{V}^2=\norm{q(1)}{V}^2&\le D_\varPhi\norm{\breve p-tQ^*Q\breve z_\bullet}{L^2((0,1),H)}^2\notag\\
&\le 2(D_\varPhi+\norm{Q^*Q}{\clL(H)}^2)\left(\norm{\breve p}{L^2((0,1),H)}^2+\norm{\breve z_\bullet}{L^2((0,1),H)}^2\right)\notag\\
&=2(D_\varPhi+\norm{Q^*Q}{\clL(H)}^2)\left(\norm{p}{L^2((s,s+1),H)}^2+\norm{z_\bullet}{L^2((s,s+1),H)}^2\right),\notag
 \end{align}
with
\begin{align}\notag
D_\varPhi\coloneqq\rme^{C_3\norm{\varPhi}{L^\infty(I,\bbR)}}\ge\sup\limits_{s\ge0}\rme^{\int_s^{s+1}C_3\varPhi(t)\,\rmd t}.
 \end{align}
 Then since both~$p$ and~$z_\bullet$ are in~$\in L^2(I,H)$, we can conclude that
\begin{align}
&\lim_{s\to+\infty}\norm{p}{L^2((s,s+1),H)}^2=0,\qquad\lim_{s\to+\infty}\norm{z_\bullet}{L^2((s,s+1),H)}^2=0,\notag
 \end{align}
 which allow us to derive the asymptotic limit
 \begin{align}
\lim_{s\to+\infty}\norm{p(s)}{V}=0.\label{pVconv0}
 \end{align}

Finally, we show that~$p\in W((s_0,+\infty),\rmD(A),H)$.
We start by considering the finite time interval~$I^s=(s_0,s)$, $s_0<s<+\infty$. Note that by~\eqref{p-regular-loc} we already know that~$p\in W(I^s,\rmD(A),H)$, and
satisfies
\begin{align}
&\dot{ p}=A  p+\rmd f^{y_\ttt}\rest{z_\bullet}  p -Q^*Q\overline  z,\notag
 \end{align}
and after multiplication with~$-2 p$ we find, recalling~\eqref{bdd-f-opt1},
\begin{align}
-\tfrac{\rmd}{\rmd t}\norm{p}{H}^2&\le-2\norm{p}{V}^2+2 C_1\left(\norm{z_\bullet}{L^6}^2+\norm{y_\ttt}{L^6}^2+1\right)\norm{p}{L^6}\norm{p}{H} +2\norm{Q^*Qz_\bullet}{H}\norm{p}{H}\notag\\
&\le-\norm{p}{V}^2+2\Phi\norm{p}{H}^2 +2\norm{Q^*Qz_\bullet}{H}^2\notag,
 \end{align}
with~$\Phi(t)\coloneqq \norm{\Id}{\clL(V,L^6)}^2C_1^2\left(\norm{z_\bullet(t)}{L^6}^2+\norm{y_\ttt(t)}{L^6}^2+1\right)^2$. By Lemma~\ref{L:estL6free}, it follows that~$\Phi(t)\le \ovlineC{\norm{\zeta}{\infty},\dnorm{\Omega}{},\norm{z_\bullet}{W(I,\rmD(A),H)}}\eqqcolon C_2$, where~$C_2$ is independent of~$s$. Then, time integration gives us
\begin{align}
\norm{p(s_0)}{H}^2+\norm{p}{L^2(I^s,V)}^2\le \norm{p(s)}{H}^2+2C_2\norm{p}{L^2(I^s,H)}^2 +2\norm{Q^*Qz_\bullet}{L^2(I^s,H)}^2,\notag
 \end{align}
which gives us, by taking the limit as $s$ goes to~$+\infty$,
\begin{align}
\norm{p(s_0)}{H}^2+\norm{p}{L^2(I,V)}^2\le 2C_2\norm{p}{L^2(I,H)}^2 +2\norm{Q^*Qz_\bullet}{L^2(I,H)}^2<+\infty.\notag
 \end{align}

Next, we multiply the dynamics with~$-2A p$ and find
\begin{align}
-\tfrac{\rmd}{\rmd t}\norm{p}{V}^2&\le-2\norm{p}{\rmD(A)}^2 +2 C_1\left(\norm{z_\bullet}{L^6}^2+\norm{y_\ttt}{L^6}^2+1\right)\norm{p}{L^6}\norm{p}{\rmD(A)} +2\norm{Q^*Qz_\bullet}{H}\norm{p}{\rmD(A)}\notag\\
&\le-\norm{p}{\rmD(A)}^2+2\Phi\norm{p}{V}^2 +2\norm{Q^*Qz_\bullet}{H}^2\notag,
 \end{align}
and
\begin{align}
\norm{p(s_0)}{V}^2+\norm{p}{L^2(I,\rmD(A))}^2\le 2C_2\norm{p}{L^2(I,V)}^2 +2\norm{Q^*Q\overline  z}{L^2(I,H)}^2<+\infty.\notag
 \end{align}

Using the dynamics once more we can see that
\begin{align}
\norm{\dot p}{L^2(I,H)}&\le \norm{p}{L^2(I,\rmD(A))} +\norm{\rmd f^{y_\ttt}\rest{z_\bullet} p}{L^2(I,H)} +\norm{Q^*Qz_\bullet}{L^2(I,H)}\\
&\le \norm{p}{L^2(I,\rmD(A))}+\norm{Q^*Qz_\bullet}{L^2(I,H)}\\
&\quad+C_1\left(\norm{z_\bullet(t)}{L^\infty(I,L^6)}^2
+\norm{y_\ttt(t)}{L^\infty(I,L^6)}^2
+1\right)
\norm{p}{L^2(I,L^6)} <+\infty
,\notag
 \end{align}

Therefore, we can conclude that~$p\in W(I,\rmD(A),H)$.

%%%%%%%%%%%%%%%%%%%%%
%%%%%%%%%%%%%%%%%%%%%
\section{Numerical simulations.}\label{S:simul}
We compare simulations for the nonlinear system defining the error dynamics~\eqref{sys-z-hat} with two different kinds of control. The first one is given by the explicit feedback as in~\eqref{const-K},  the second one by an approximation to the optimal control given   by receding horizon control. For the latter we sucessively minimize
\begin{equation}\notag
\clJ_{\clI_i}^Q(z,u)\coloneqq\tfrac12\norm{Qz}{L^2(\clI_i,{H}
)}^2+\tfrac12\norm{u}{L^2(\clI_i,\bbR^{M_\sigma})}^2,
\end{equation}
 in the time intervals~$\clI_i\coloneqq(i\delta_{\rm rh},i\delta_{\rm rh}+T_{\rm rh})$, seeking for optimal pairs
\[
(z_{\clI_i}^{z_{i}},u_{\clI_i}^{z_{i}}),\qquad  z_{i+1}\coloneqq z_{\clI_{i}}^{z_{i}}((i+1)\delta_{\rm rh}), \quad i\in\bbN,
\]
solving Problem~\ref{Pb:OCPI}; for~$i=0$, $z_0$ is the initial error in~\eqref{sys-z-hat}.
 We refer the reader to the recent works~\cite{AzmiKun19,KunPfeiffer20,GruneSchaSchi22} for more details on this receding horizon strategy in the context of parabolic-like equations.

 For the computation of the finite-horizon optimal pairs~$(z_{\clI_i}^{z_{i}},u_{\clI_i}^{z_{i}})$ we use a gradient based algorithm proposed in~\cite{AzmiKun20} and~\cite[sect.~4]{AzmiKun21} with Barzilai--Borwein time steps; see~\cite{BarzBorw88}.

The simulations have been run in the finite-horizon time interval~$(0,T)$, with
\[
T=15,\qquad T_{\rm rh}=1,\qquad\delta_{\rm rh}=0.5.
\]

The truncated finite time-horizon cost functional
\begin{equation}\notag
\bfJ_T^Q\coloneqq\clJ_{(0,T)}^Q(z,u)\coloneqq\tfrac12\norm{Qz}{L^2((0,T),{H}
)}^2+\tfrac12\norm{u}{L^2((0,T),\bbR^{M_\sigma})}^2,
\end{equation}
is seen as an approximation of the infinite time-horizon one~\eqref{Jz0}. For~$Q$ (motivated by Lemma~\ref{L:key-exist-Id}, cf. discussion at the end of sect.~\ref{sS:Intro-liter} on  the Main Result~\ref{MR:mainOptim}),
we have chosen the orthogonal projection onto the first~$M_1$ eigenfunctions of the Laplacian~$\Delta$ ,
\[
Q=P_{\clE_{M_1}^\rmf},\qquad M_1=20.
\]

We consider the evolution of our model~\eqref{Schloegl-feed} in the one-dimensional spatial domain
$\Omega=(0,1)$,
with the parameters chosen as~$
\nu=0.1$ and~$(\zeta_1,\zeta_2,\zeta_3)=(-1,0,2).$

We present results for two targeted trajectories, namely,
\[
y_\ttt(t,x)=0\quad\mbox{and}\quad y_\ttt(t,x)=\sin(3t)\cos(\pi x),\quad\mbox{with}\quad (t,x)\in[0,+\infty)\times\Omega.
\]
These trajectories  solve the free dynamics system~\eqref{sys-haty} for appropriate external forcings~$h= h_{y_\ttt}(t,x)$.
Finally, for the initial error we set
\[
z_0=-4+8\cos(2\pi x^2),\quad x\in(0,1).
\]

\begin{remark}
In the case~$y_\ttt(t,x)=0$ we must take~$h_{y_\ttt}(t,x)=0$. Observe that in this case, under Neumann boundary conditions (cf.~\eqref{DA}), the constant functions~$\overline\zeta_i(x)\coloneqq\zeta_i$ are steady states, where $\overline\zeta_2(x)=0$ is locally exponentially unstable and~$\overline\zeta_1(x)=-1$ and $\overline\zeta_3(x)=2$ are locally exponentially stable.
\end{remark}

Fig.~\ref{Fig:Free} shows that the free-dynamics error do not converge to zero as time increases.
\begin{figure}[ht]
\centering
\subfigure[Target $y_\ttt(t,x)=0$.]
{\includegraphics[width=0.45\textwidth,height=0.20\textheight]{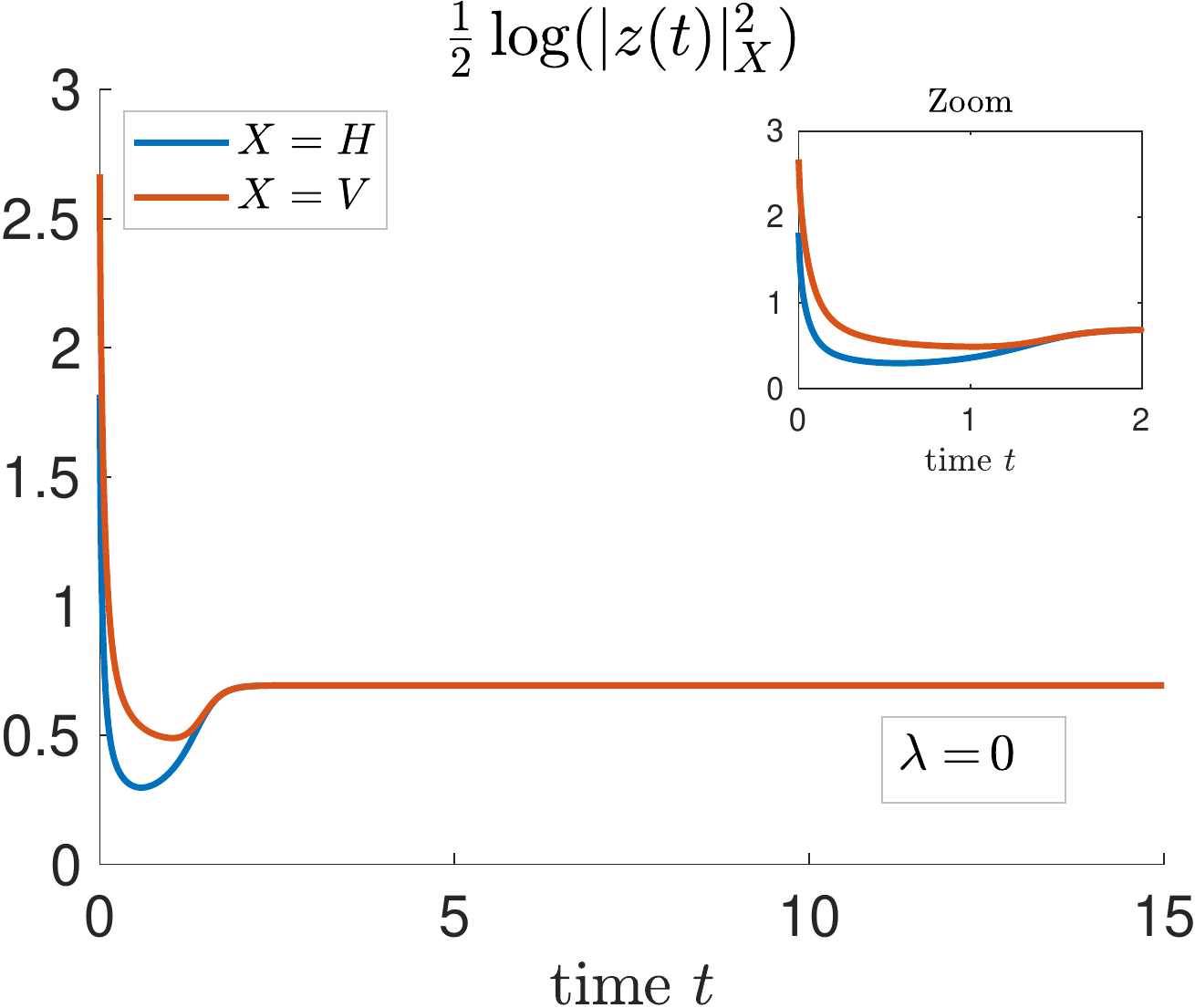}}
\qquad
\subfigure[Target $y_\ttt(t,x)=\sin(3t)\cos(\pi x)$.]
{\includegraphics[width=0.45\textwidth,height=0.20\textheight]{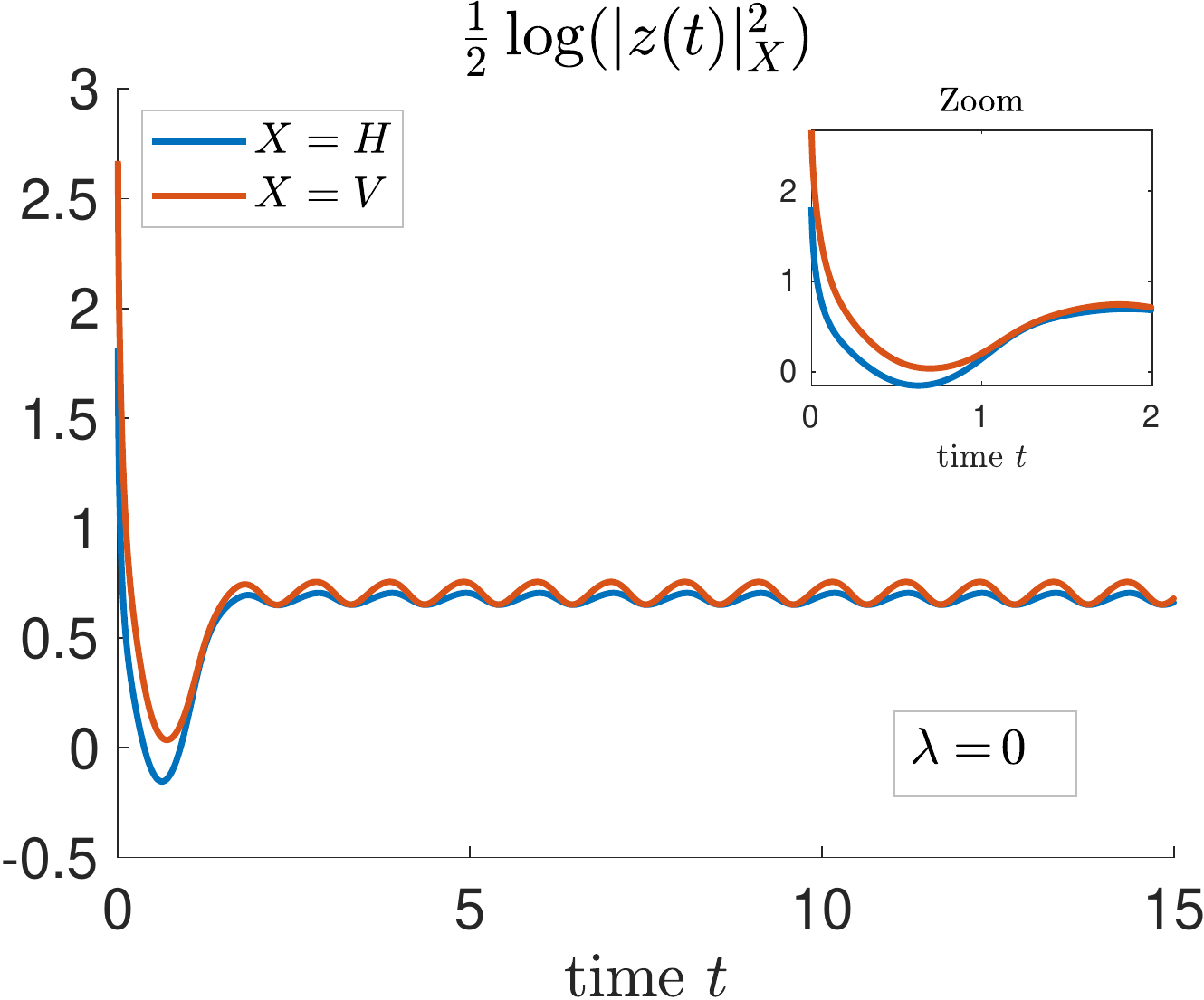}}
\caption{Free dynamics error ($u=0$, i.e., with $\lambda=0$ in~\eqref{SatKw}).}
\label{Fig:Free}
\end{figure}

 Hence a control is needed in order to reach our goal~\eqref{goal-diff-W12} and/or~\eqref{goal-diff-L2}.
 We  utilize four indicator function actuators~$\indf_{\omega_i}$ as in~\eqref{U_M} in section~\ref{sS:actuators}, with
\[
\omega_i=(\tfrac{(2i-1)}{8}-\tfrac{r}{8},\tfrac{(2i-1)}{8}+\tfrac{r}{8}),\qquad r=0.1,\qquad 1\le i\le4,
\]
covering~$10\%$ of the spatial domain.
We take control input magnitude constraints as
\[
\dnorm{u(t)}{}\le C_u,\quad\mbox{with}\quad\dnorm{\Bigcdot}{}\coloneqq\norm{\Bigcdot}{\ell^\infty},
\]
where~$\norm{(v_1,v_2,\dots,v_M)}{\ell^\infty}\coloneqq\max\{\norm{v_i}{\bbR}\mid 1\le i\le M\}$ is the usual~$\ell^\infty$-norm in~$\bbR^M$. We have considered two cases for the magnitude control bound, namely,
\[
C_u=30\quad\mbox{and}\quad C_u=15.
\]

 Concerning the discretization, for the spatial variable we used piecewise linear finite-elements in a regular partition of the spacial domain with 1001 points.  For the temporal variable we used  a Crank--Nicolson--Adams--Bashforth scheme with time step~$10^{-4}$.

\subsection{Targeting zero}\label{sS:numexa-0}
In Figs.~\ref{Fig:Explicit0} and Figs.~\ref{Fig:rhc0}
\begin{figure}[t]
\centering
\subfigure%[]
{\includegraphics[width=0.45\textwidth,height=0.20\textheight]{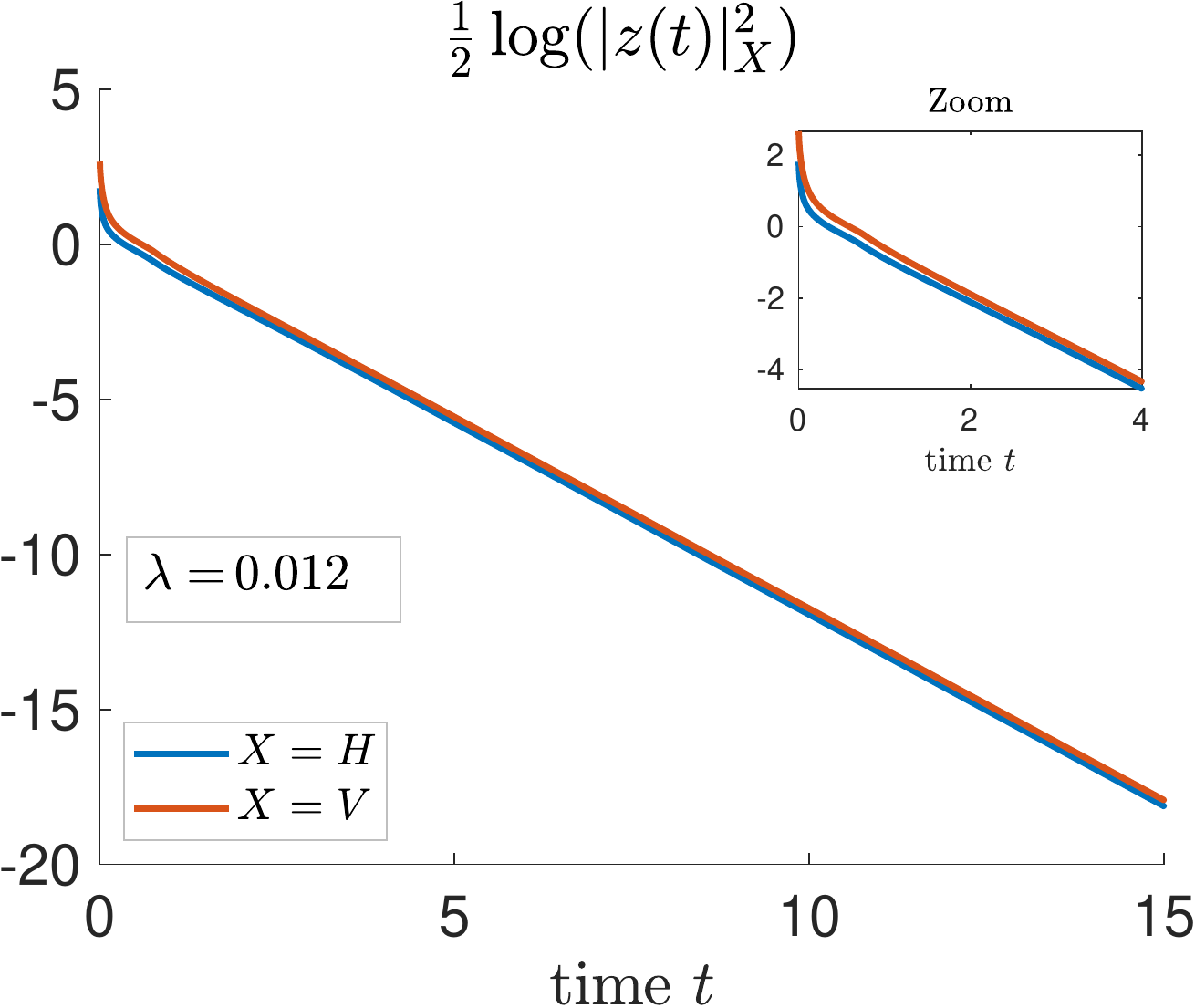}}
\quad
\subfigure%[]
{\includegraphics[width=0.45\textwidth,height=0.20\textheight]{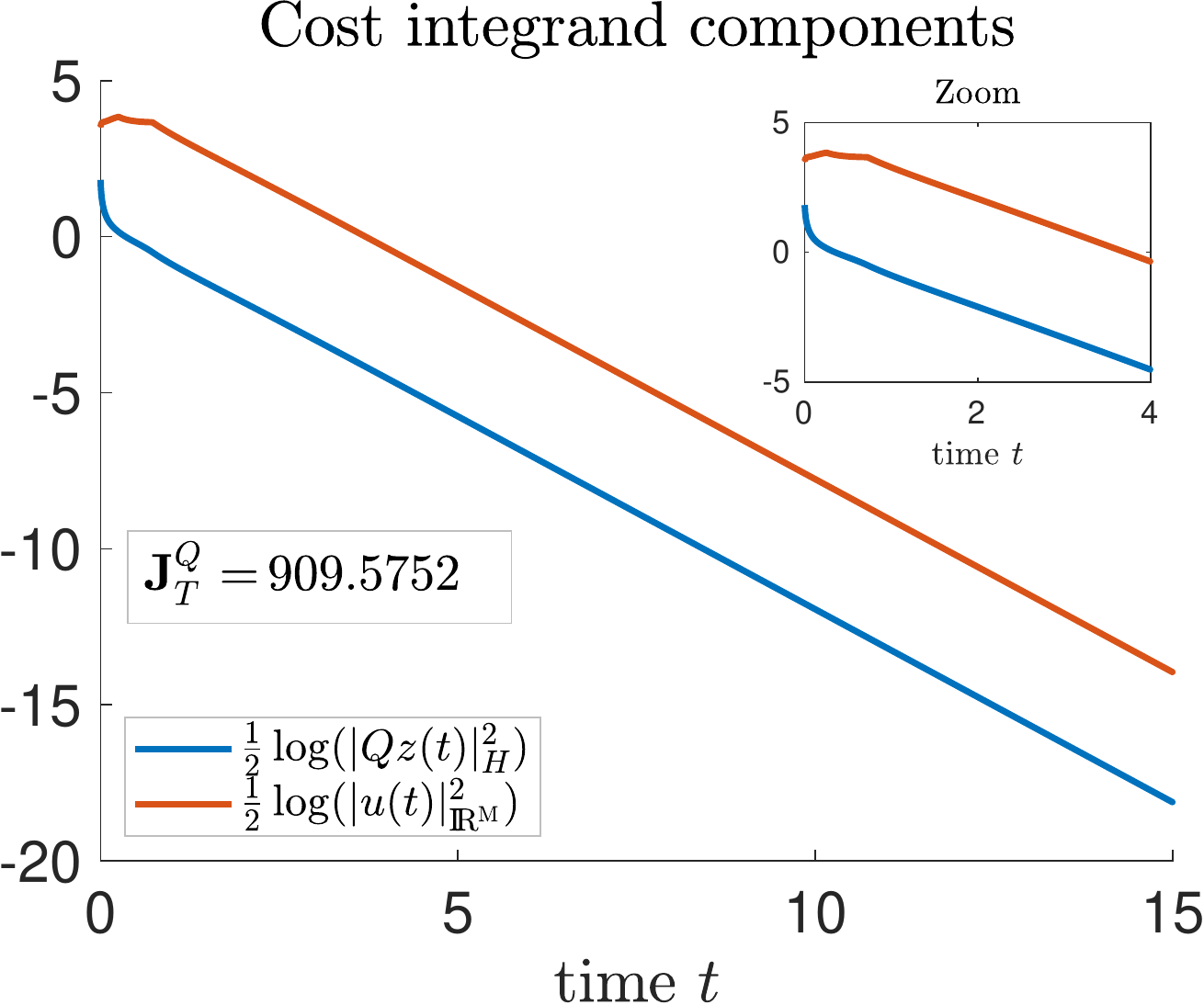}}
\\
\subfigure%[]
{\includegraphics[width=0.45\textwidth,height=0.20\textheight]{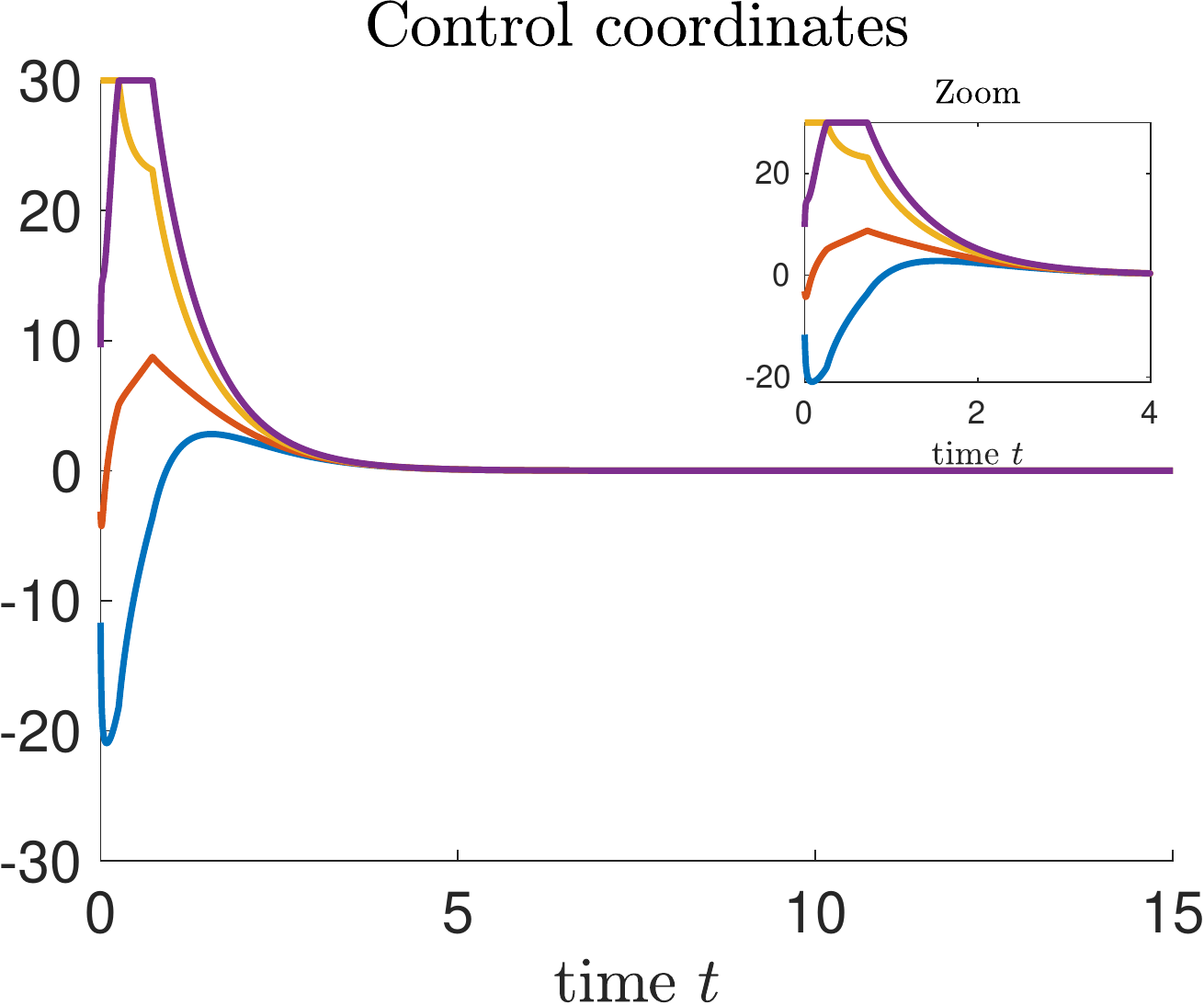}}
\caption{Explicit feedback. $C_u=30$.   $y_\ttt(t,x)=0$.\newline}
\label{Fig:Explicit0}
\subfigure%[]
{\includegraphics[width=0.45\textwidth,height=0.20\textheight]{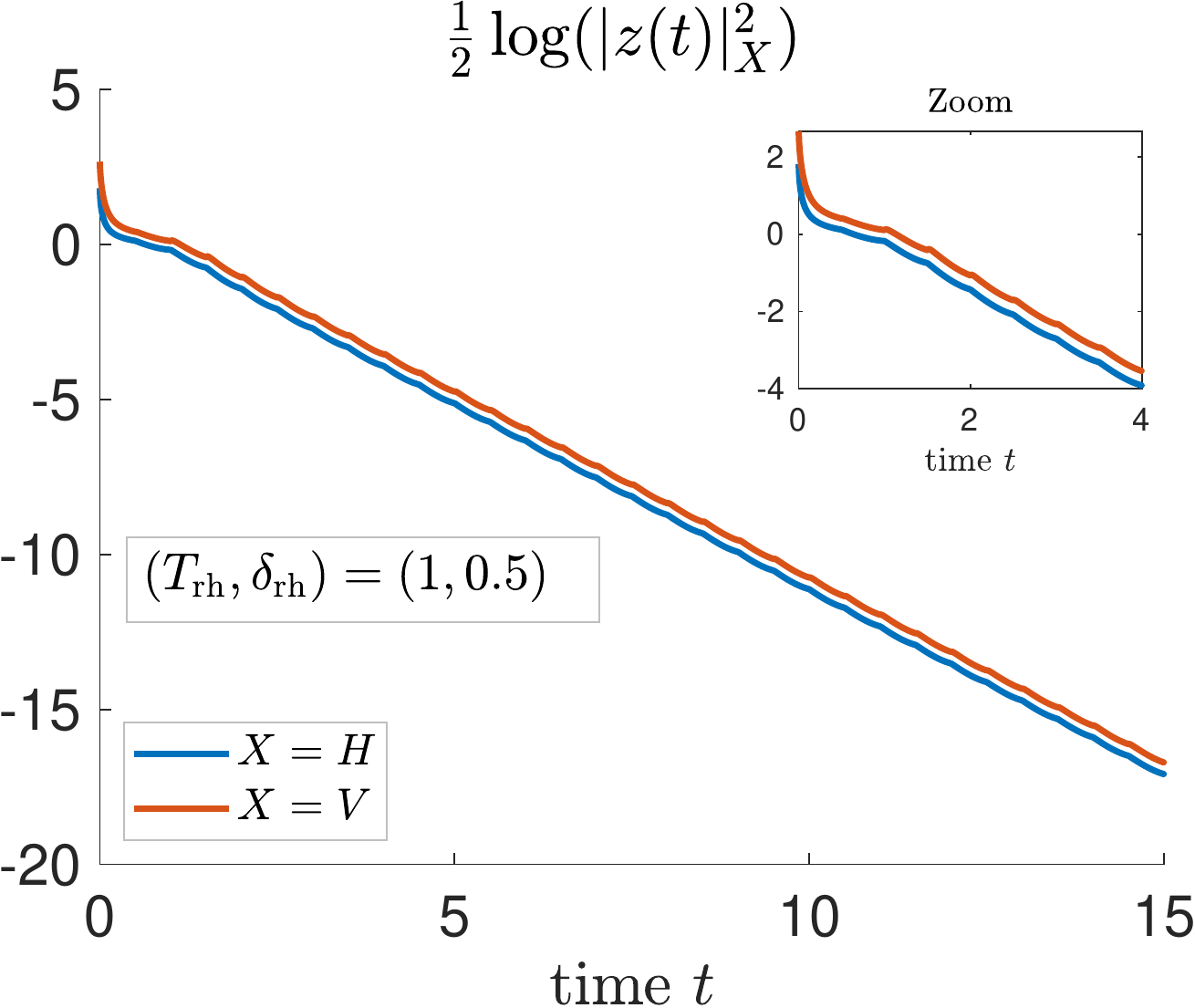}}
\quad
\subfigure%[]
{\includegraphics[width=0.45\textwidth,height=0.20\textheight]{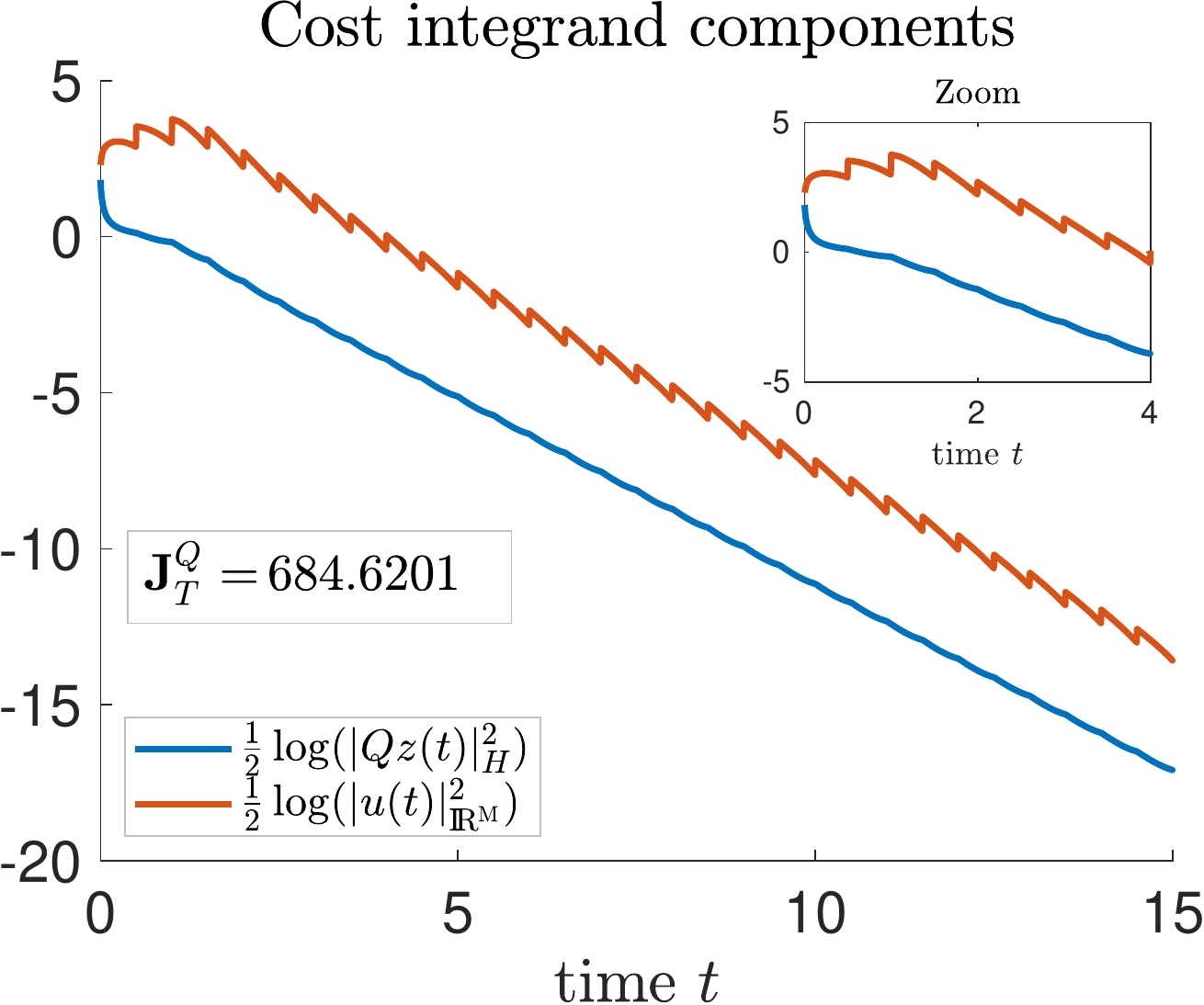}}
\\
\subfigure%[]
{\includegraphics[width=0.45\textwidth,height=0.20\textheight]{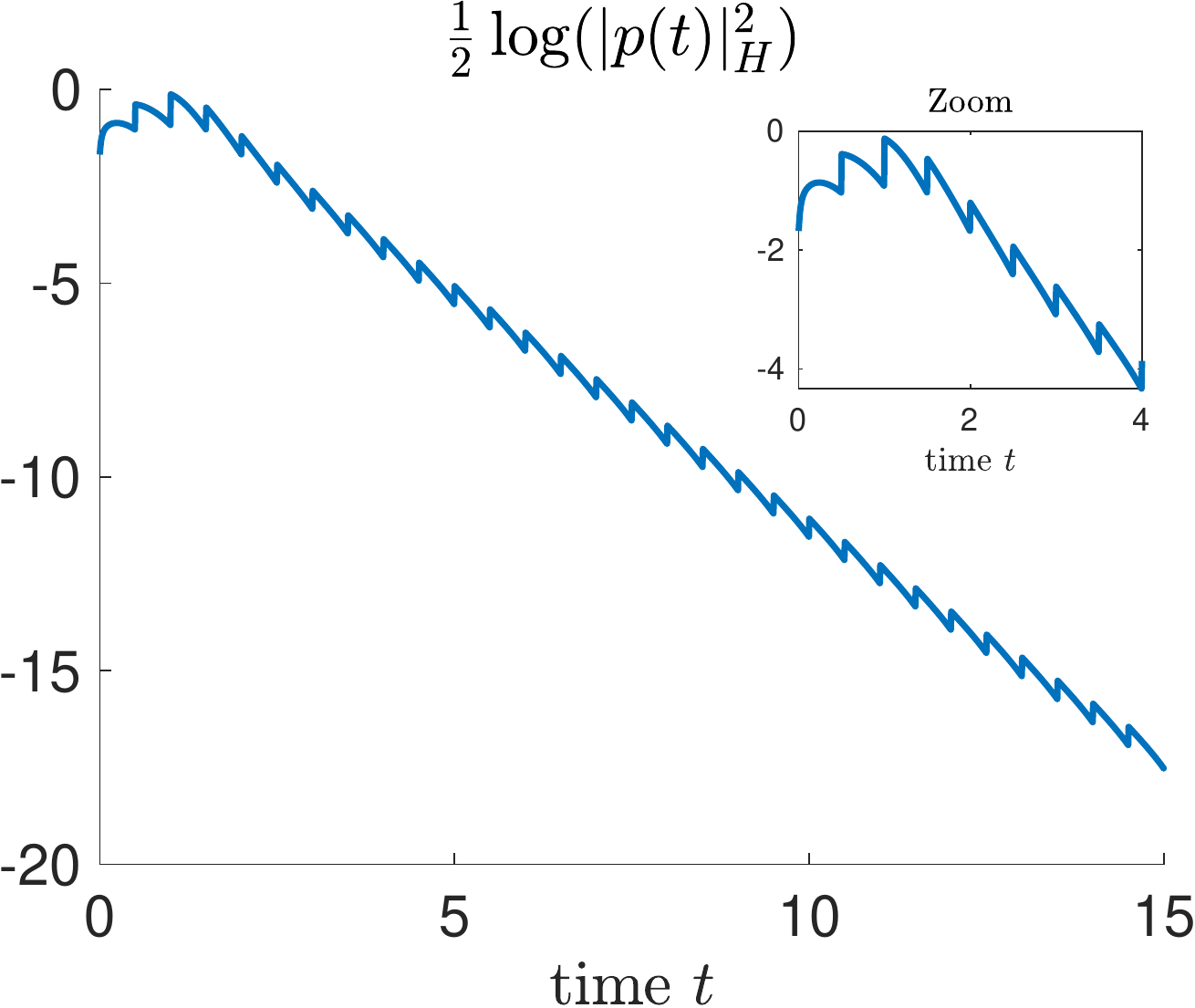}}
\quad
\subfigure%[]
{\includegraphics[width=0.45\textwidth,height=0.20\textheight]{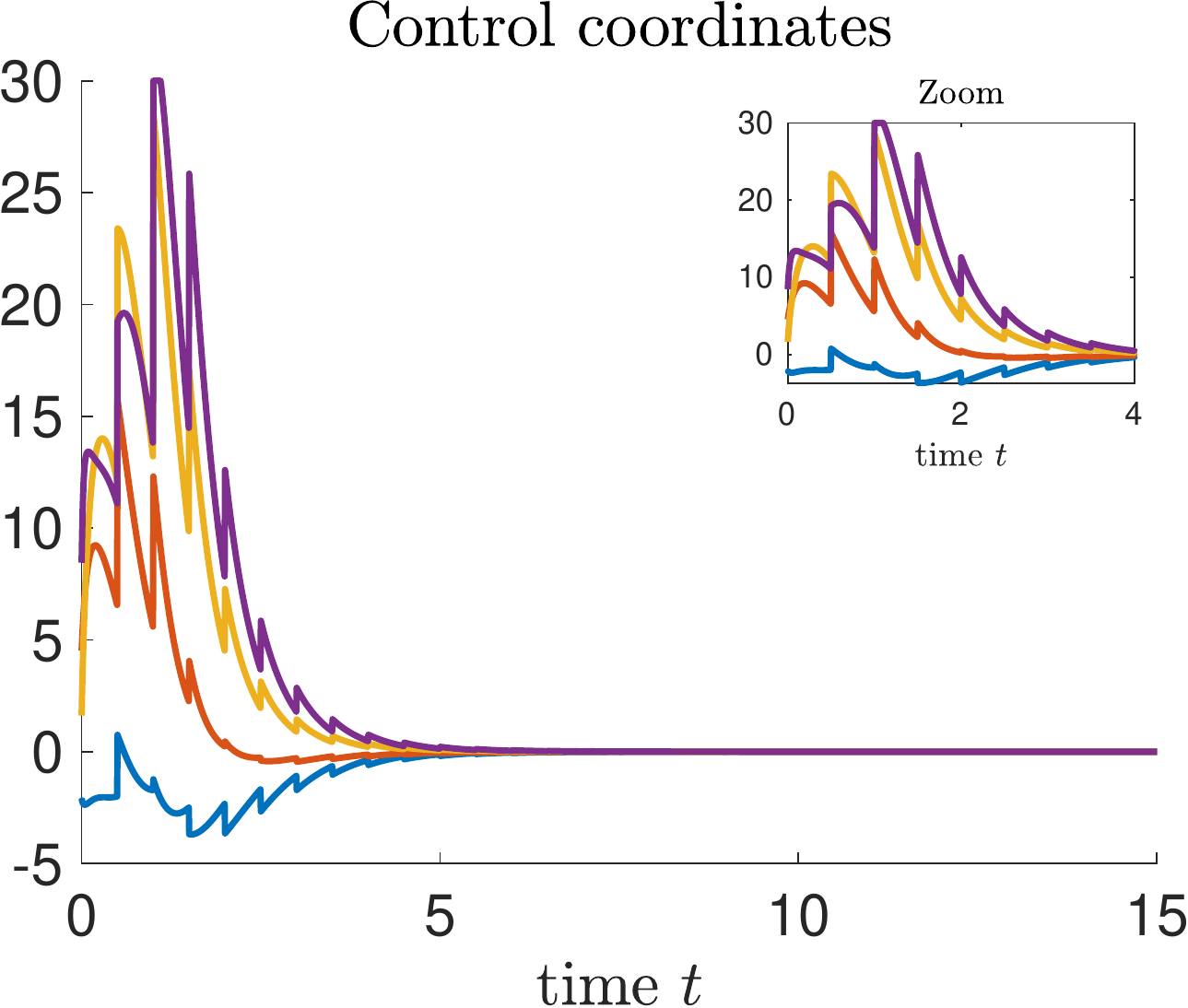}}
\caption{Receding horizon. $C_u=30$.   $y_\ttt(t,x)=0$.}
\label{Fig:rhc0}
\vspace{-.5em}
\end{figure}
 we present the results concerning the case
~$y_\ttt(t,x)=0$ with~$C_u=30$. We see that both the explicit feedback control and the receding horizon control are able to stabilize the system.  The computation of the former is much cheaper, since  the later involves the computation of thirty optimal control problems in time intervals of length~$T_{\rm rh}=1$. Though, receding horizon controls are more difficult to compute, they  play an important role in case we are interested in the minimization  the total energy spent during the stabilization process. Indeed, comparing~$\bfJ_T^Q$ in Figs.~\ref{Fig:Explicit0} and~\ref{Fig:rhc0}, we observe that the spent energy associated with the receding horizon control is smaller than that spent with the explicit feedback.

\subsection{A time-dependent targeted trajectory}\label{sS:numexa-gen}
In Figs.~\ref{Fig:Explicit-tx} and~\ref{Fig:rhc-tx} we present the results concerning the case
of the targeted trajectory~$y_\ttt(t,x)=\sin(3t)\cos(\pi x)$ depending on both spatial and temporal variables  with~$C_u=30$. Again, both the explicit feedback control and the receding horizon control are able to stabilize the system and the spent energy~\eqref{Jz0} associated with the latter is smaller.

Note that, for the receding horizon strategy, the discontinuities of  the controls and  adjoint states occur at the concatenation  time instants~$t=i\delta_{\rm rh}$.
\begin{remark}
 We  first solved for the receding-horizon control solution. Then we  solved the system with the explicit feedback for several values~$\lambda$. In Figs.~\ref{Fig:Explicit0} and~\ref{Fig:Explicit-tx} we present  the results for one of such~$\lambda$s, for which the norm of the error at final time  is close to the corresponding norm at the final time reached with the receding-horizon control. For this reason the values of~$\lambda$ differ in Figs.~\ref{Fig:Explicit0} and~\ref{Fig:Explicit-tx}.
\end{remark}
\begin{figure}[t!]
\centering
\subfigure%[]
{\includegraphics[width=0.45\textwidth,height=0.20\textheight]{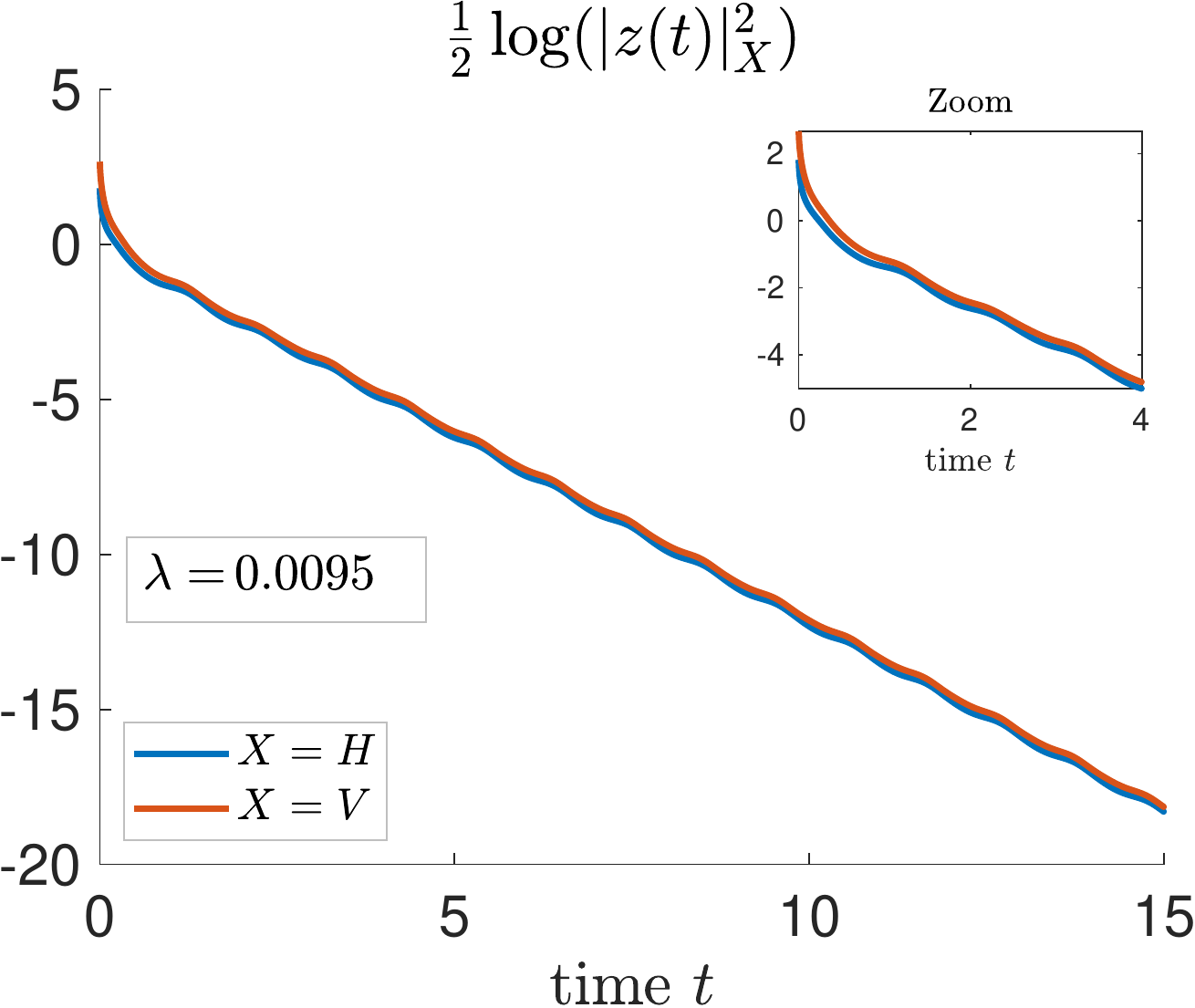}}
\quad
\subfigure%[]
{\includegraphics[width=0.45\textwidth,height=0.20\textheight]{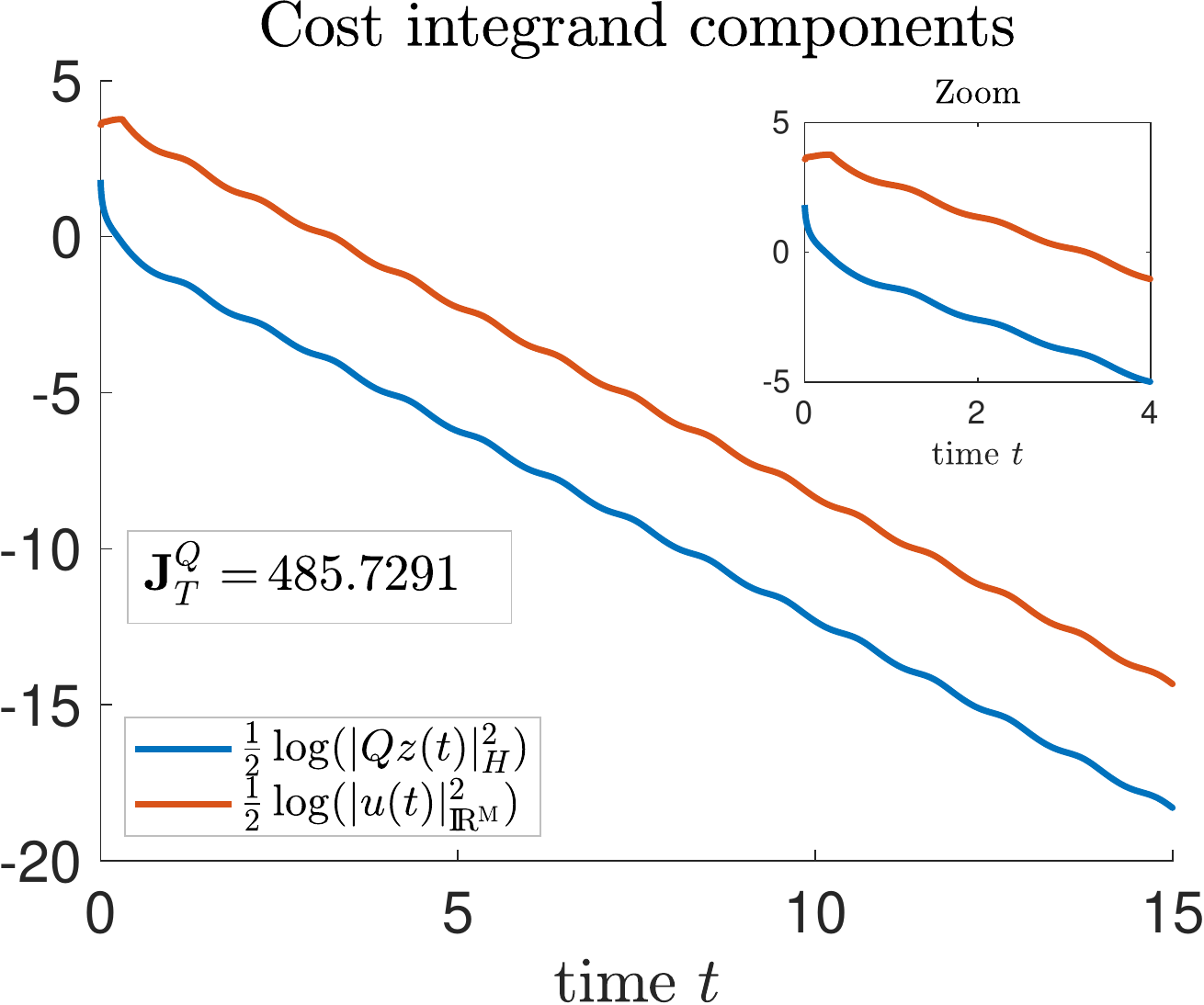}}
\\
\subfigure%[]
{\includegraphics[width=0.45\textwidth,height=0.20\textheight]{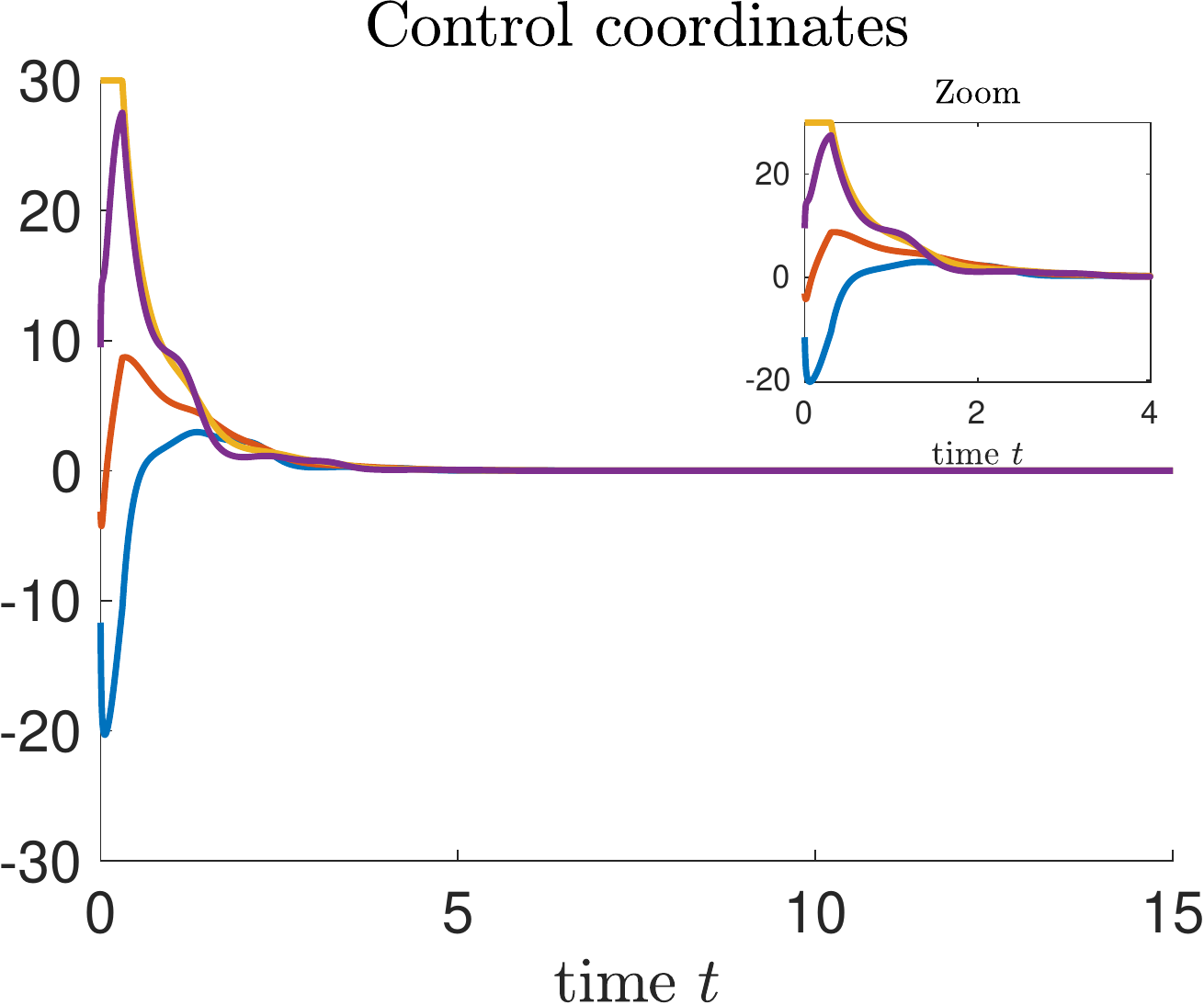}}
\caption{Explicit feedback. $C_u=30$. $y_\ttt(t,x)=\sin(3t)\cos(\pi x)$.\newline}
\label{Fig:Explicit-tx}
\subfigure%[]
{\includegraphics[width=0.45\textwidth,height=0.20\textheight]{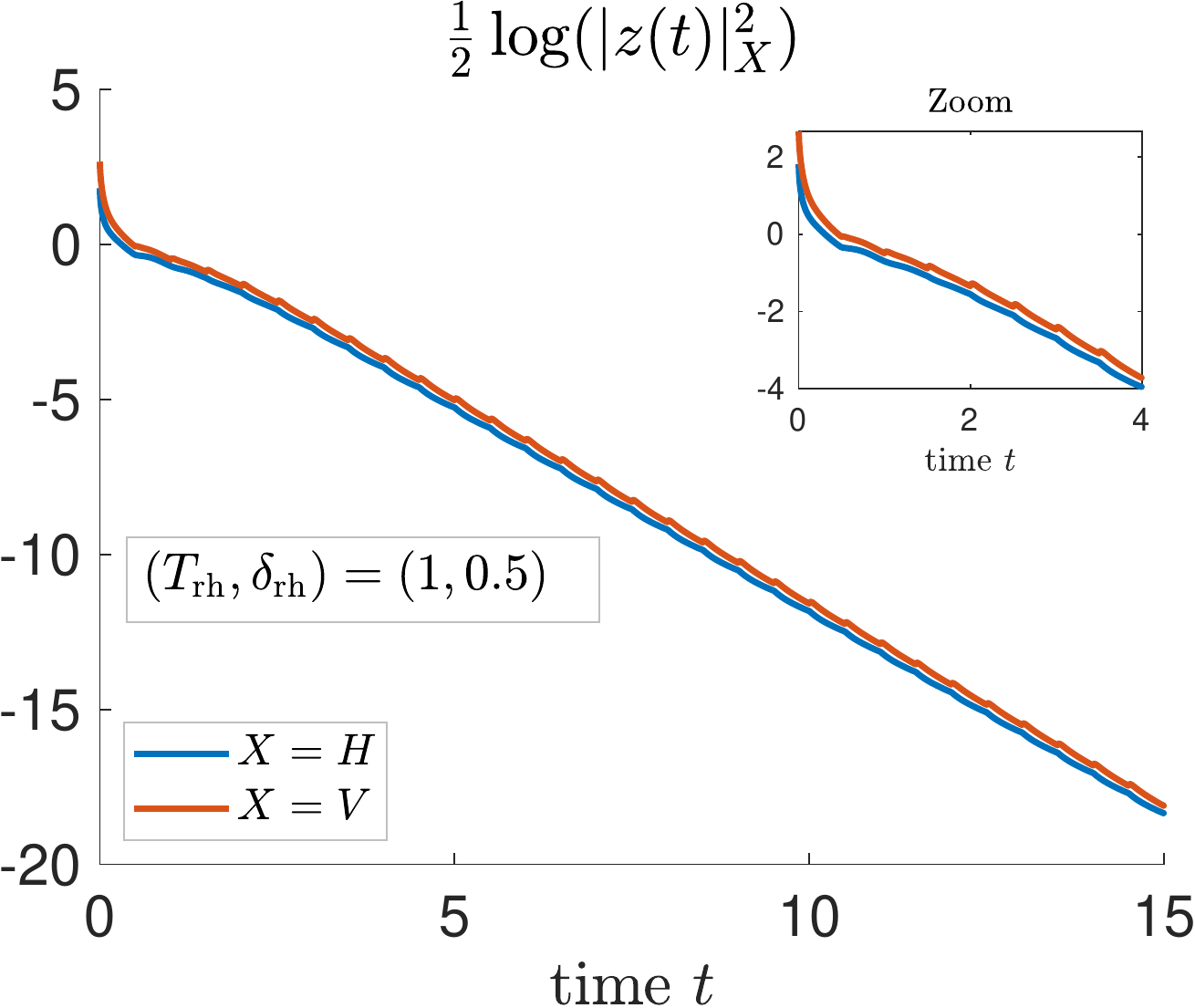}}
\quad
\subfigure%[]
{\includegraphics[width=0.45\textwidth,height=0.20\textheight]{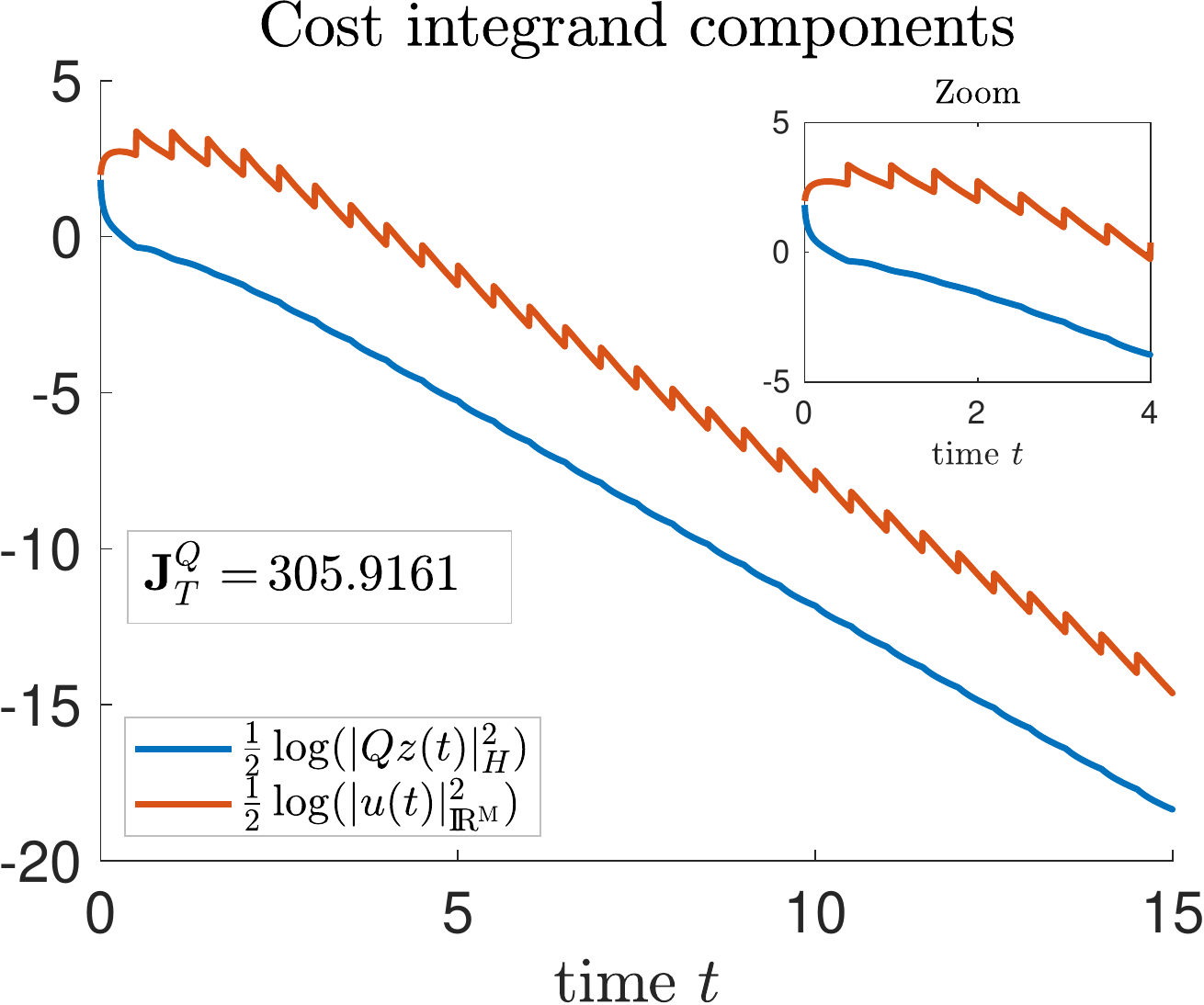}}
\\
{\includegraphics[width=0.45\textwidth]{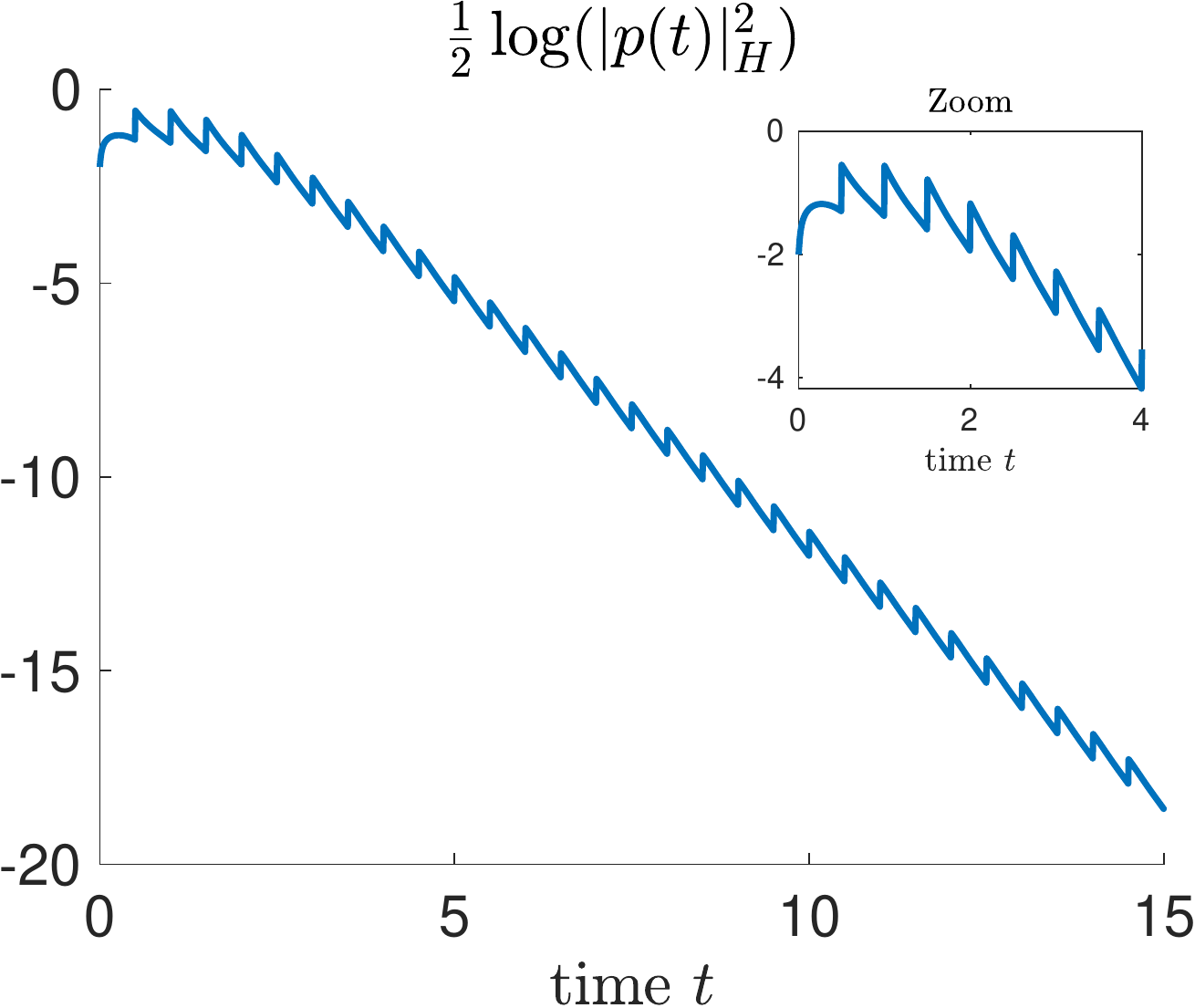}}
\quad
\subfigure%[]
{\includegraphics[width=0.45\textwidth,height=0.20\textheight]{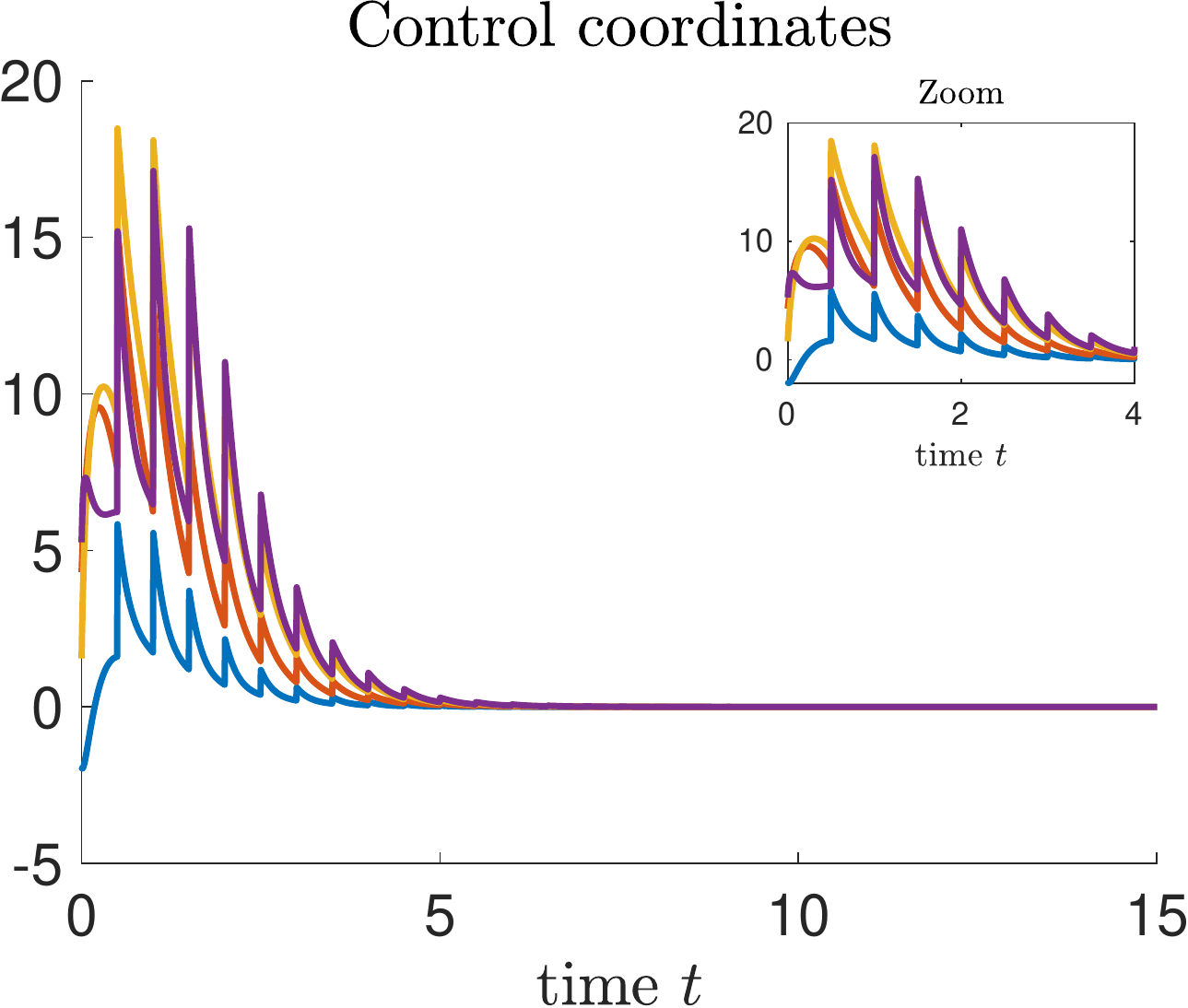}}
\caption{Receding horizon. $C_u=30$. $y_\ttt(t,x)=\sin(3t)\cos(\pi x)$.}
\label{Fig:rhc-tx}
\vspace{-.5em}
\end{figure}

\subsection{On smaller magnitude control bounds}\label{sS:numexa-lim}

 We have already commented on the fact that the explicit feedback control is
significantly simpler to realize in practice
than the receding horizon control, at the prize that the spent energy is
higher.
Here we present an example, where the explicit
feedback control fails to stabilize the system, while the receding horizon control succeeds to.
To illustrate this point, we focus on the magnitude control bound by taking the smaller~$C_u=15$. In Figs.~\ref{Fig:Explicit-Cus} and~\ref{Fig:rhc-Cus} we can see that the explicit feedback control is not stabilizing, while the receding horizon one is.
\begin{figure}[t!]
\centering
\subfigure%[]
{\includegraphics[width=0.45\textwidth,height=0.20\textheight]{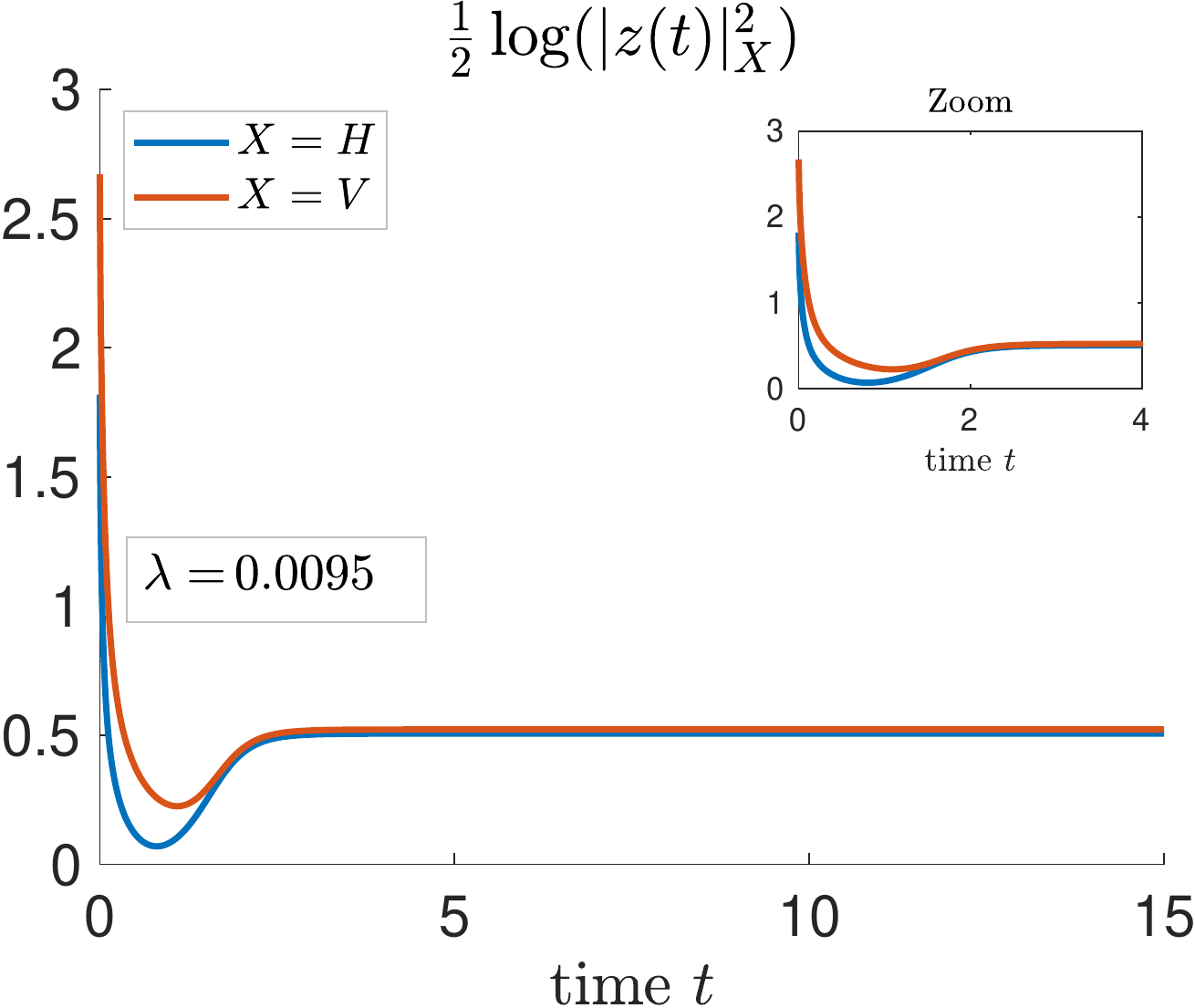}}
\quad
{\includegraphics[width=0.45\textwidth,height=0.20\textheight]{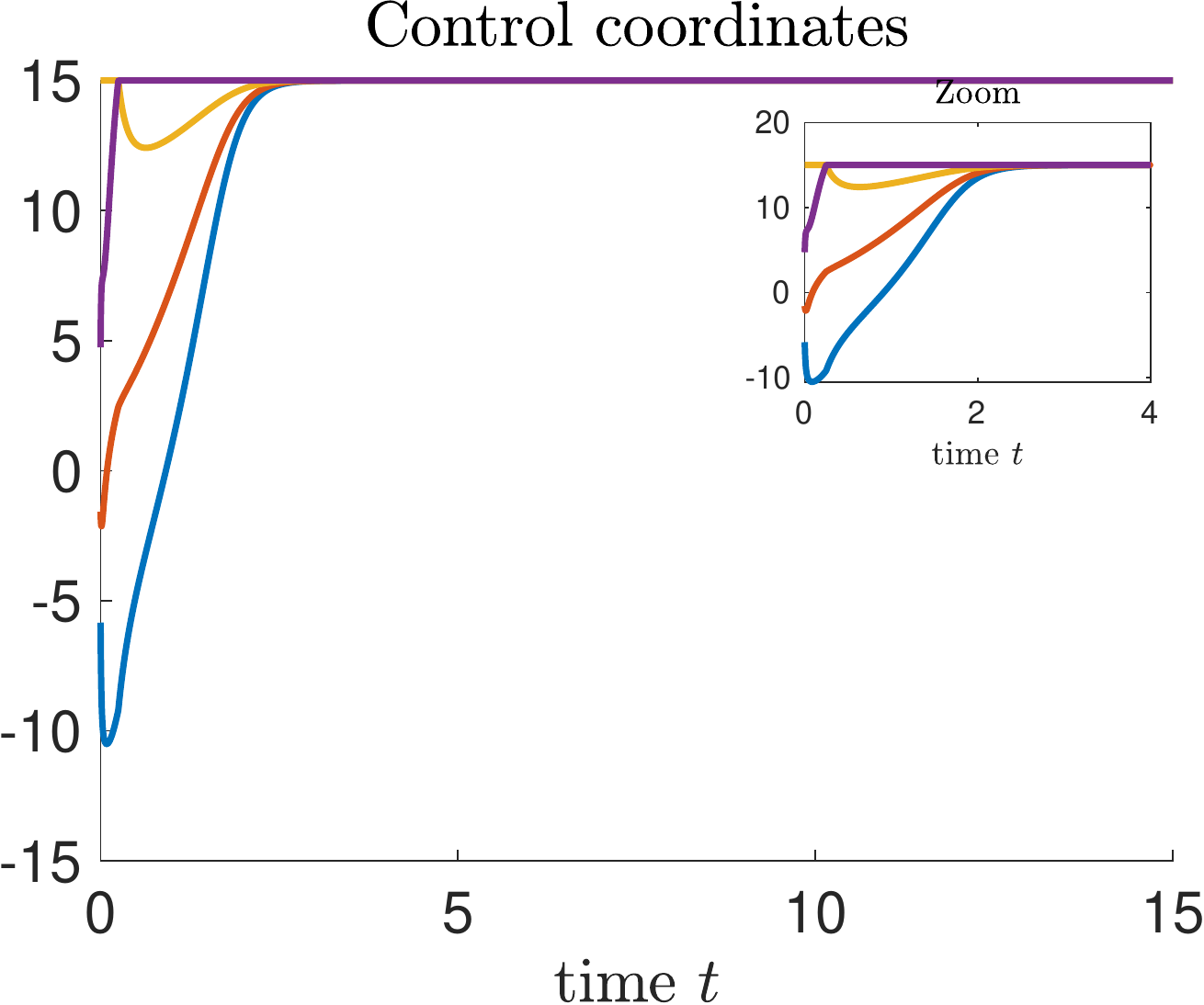}}
\caption{Explicit feedback. $C_u=15$. Target $y_\ttt(t,x)=0$.\newline}
\label{Fig:Explicit-Cus}
\subfigure%[]
{\includegraphics[width=0.45\textwidth,height=0.20\textheight]{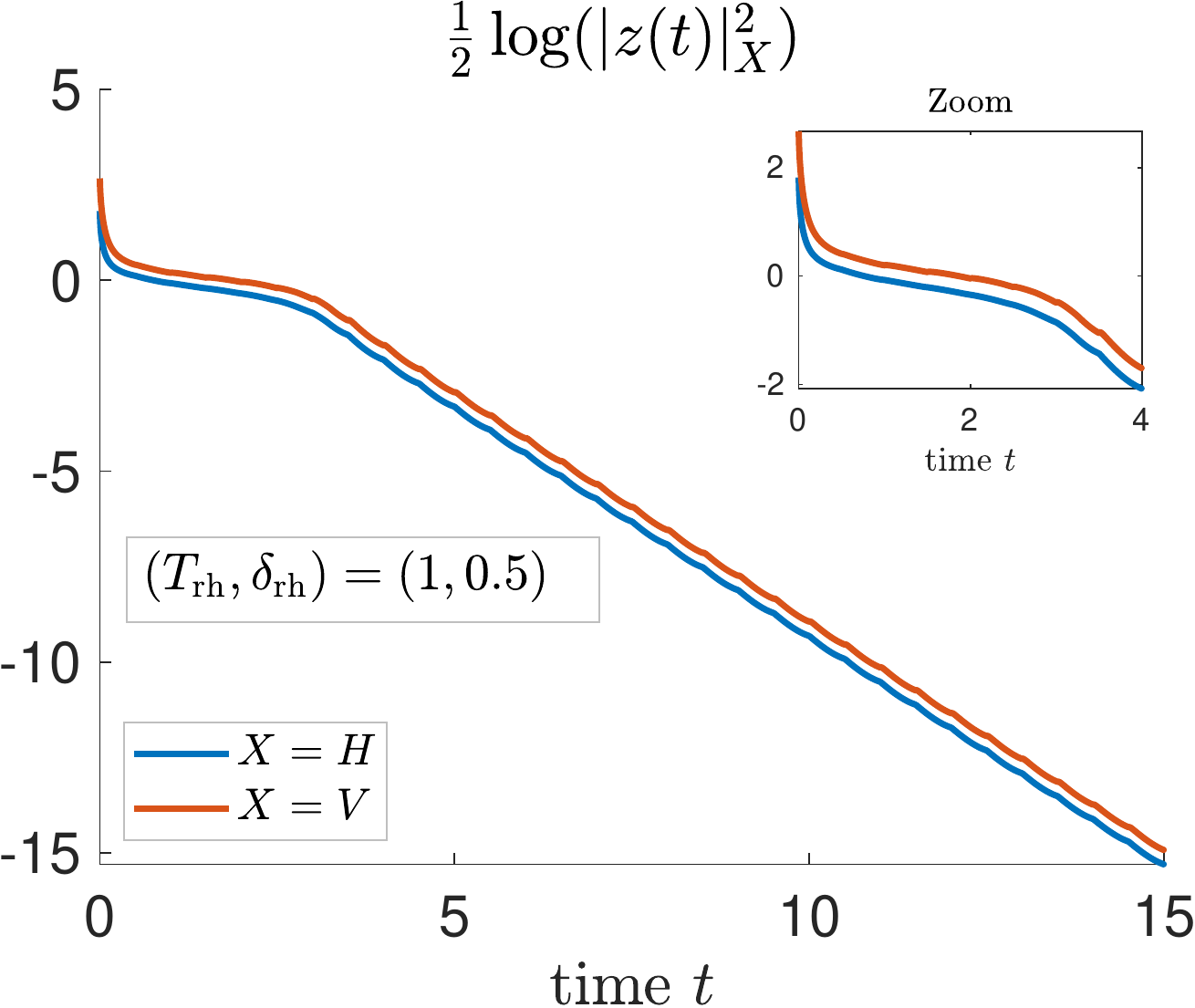}}
\quad
\subfigure%[]
{\includegraphics[width=0.45\textwidth,height=0.20\textheight]{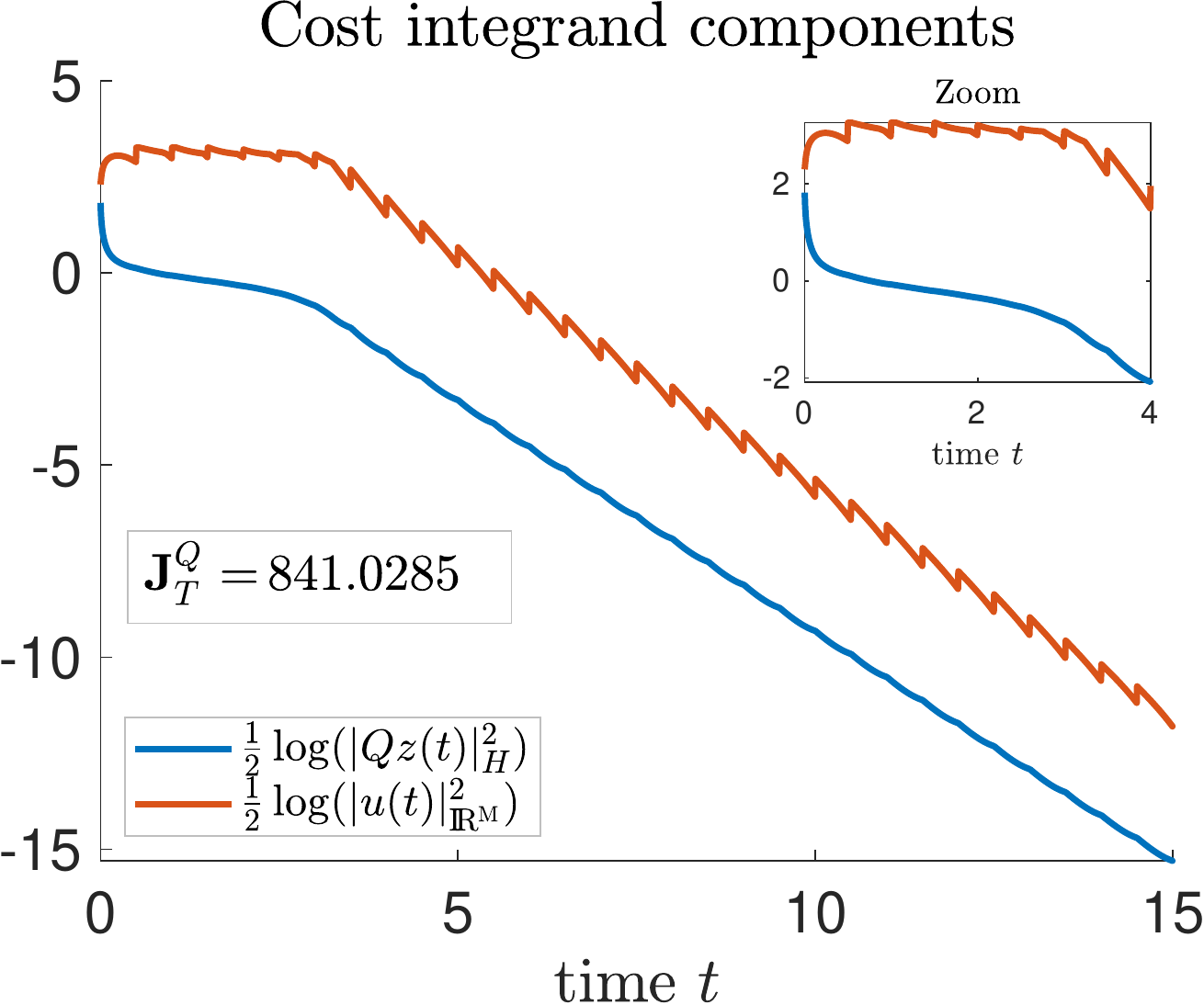}}\\
{\includegraphics[width=0.45\textwidth,height=0.20\textheight]{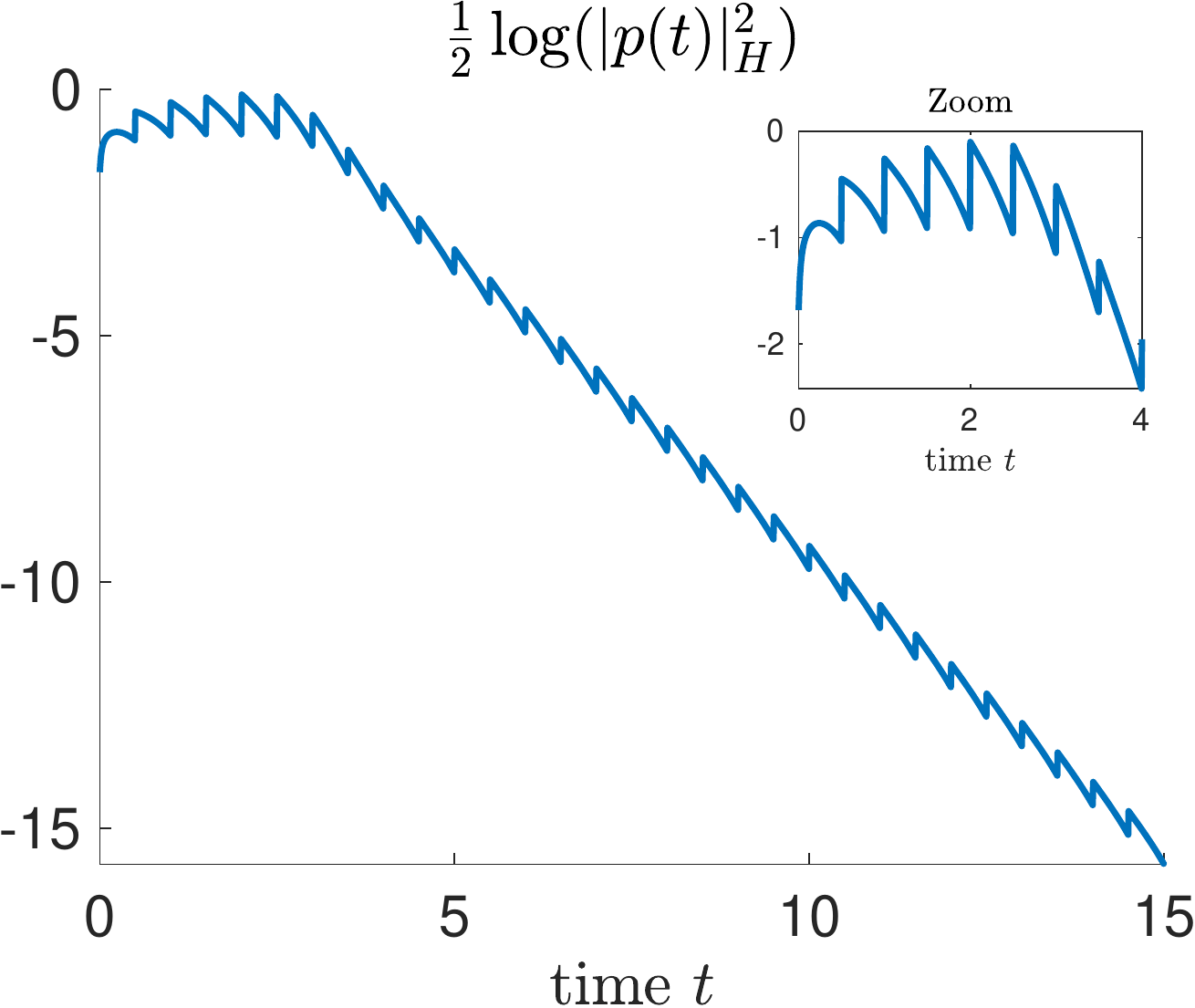}}
\qquad
\subfigure%[]
{\includegraphics[width=0.45\textwidth,height=0.20\textheight]{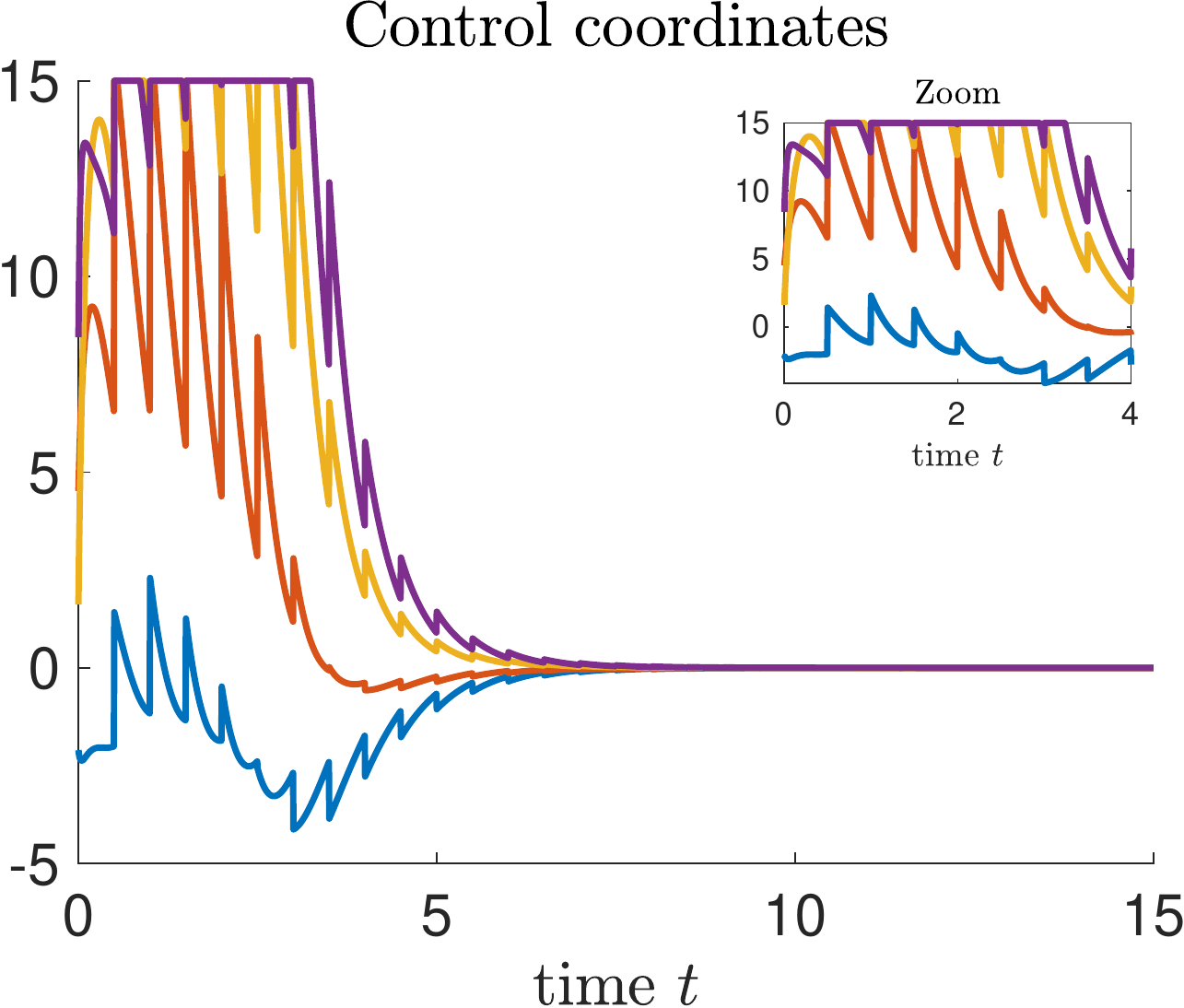}}
\caption{Receding horizon. $C_u=15$. Target $y_\ttt(t,x)=0$.}
\label{Fig:rhc-Cus}
\end{figure}

%%%%%%%%%%%%%%%%%%%%%
\bigskip\noindent
{\bf Aknowlegments.} S. Rodrigues acknowledges partial support from
the State of Upper Austria and the Austrian Science
Fund (FWF): P 33432-NBL.

 %%%%%%%%%%%%%%%%Appendix%%%%%%%%%%%%%%%%%
\bigskip
\appendix
\addcontentsline{toc}{section}{Appendix}
\newpage
\section*{Appendix}
%reset the counters
\setcounter{section}{1}%to restart with A (~1 in Alph)
\setcounter{theorem}{0} \setcounter{equation}{0}
\numberwithin{equation}{section}

 \subsection{Proof of Lemma~\ref{L:DPP}}\label{ApxProofL:DPP}
Let~$(\overline z ,\overline u)\in\fkX_{\bbR_s}^{z_0}$ minimize~$\clJ_{I}^{Q}(z , u)$ and~$(\overline w ,\overline v)\in\fkX_{I^a}^{z_0}$ minimize~$\clJ_{I^a}^Q(w,v)+\fkV^{I_a}(w(a))$.

Let us be given two pairs~$(z_{I_a}^{\overline w(a)} ,u_{I_a}^{\overline w(a)})$ and~$(z_{I_a}^{\overline z(a)} ,u_{I_a}^{\overline z(a)})$ solving Problem~\ref{Pb:OCPI}.

Note that~$(z_1,u_1)\ddag (z_2,u_2)\in\clX_{I}$ if, and only if, $z_1(a)=z_2(a)\in V$.
From
\[
(\overline w ,\overline v)\ddag( y_{I_a}^{\overline w(a)} ,u_{I_a}^{\overline w(a)})\in\fkX_{I}^{z}\quad\mbox{and}\quad
(\overline z ,\overline u)\rest{I^a}\ddag( y_{I_a}^{\overline z(a)} ,u_{I_a}^{\overline z(a)})\in\fkX_{I}^{z}
\]
and  optimality we find
\begin{align}
\clJ^Q_{I}(\overline z ,\overline u)&\le\clJ^Q_{I^a}(\overline w ,\overline v)+\clJ^Q_{I_a}( z_{I_a}^{\overline w(a)} ,u_{I_a}^{\overline w(a)})=\clJ^Q_{I^a}(\overline w ,\overline v)+\fkV^{I_a}(\overline w(a))\notag\\
&\le\clJ^Q_{I^a}(\overline z ,\overline u)+\fkV^{I_a}(\overline z(a))=\clJ^Q_{I^a}(\overline z ,\overline u)+\clJ^Q_{I_a}( z_{I_a}^{\overline z(a)} ,u_{I_a}^{\overline z(a)})\notag\\
&\le\clJ^Q_{I^a}(\overline z ,\overline u)+\clJ^Q_{I_a}(\overline z ,\overline u)=\clJ^Q_{I}(\overline z ,\overline u),\label{dpp=}
\end{align}
which finishes the proof.
\qed

\subsection{Proof of Corollary~\ref{C:optim-tail}}\label{ApxProofC:optim-tail}

By~\eqref{dpp=}, taking~$(\overline z,\overline u)\coloneqq(z_{I}^{z_0} ,u_{I}^{z_0} )$ we observe that $\fkV^{I_a}(\overline z(a))=\clJ^Q_{I_a}(\overline z ,\overline u)$.
\qed

 \subsection{Proof of Lemma~\ref{L:dG}}\label{ApxProofL:dG}
We show that indeed~$\rmd \clG_{I}\rest{(\overline z,\overline u)}\in\clL(\clX_{I},\clY_{I})$.
Recalling that~$\clY_{I}=L^2(I,H)\times V$, it is clear that the mapping
\begin{align} \label{fkG-wdef1}
&\varPsi(w,u)\coloneqq (\dot w +Aw  -U_M^\diamond u, w(0))
\quad\mbox{satisfies}\quad
\varPsi\in\clL(\clX_{I},\clY_{I}),
\end{align}
Next, recalling~\eqref{fkG-wdef3} and~\eqref{fkG-wdef2},
we conclude that~$\rmd \clG_{I}^{z_0}\rest{(\overline z,\overline u)}\in\clL(\clX_{I},\clY_{I})$.

It remains to show that
\[
\fkQ\coloneqq\lim_{\norm{(z,u)-(\overline z,\overline u)}{\clX_I}\to0}\tfrac{\norm{\clG_{I}^{z_0}(z,u)-\clG_{I}^{z_0}(\overline z,\overline u)-\rmd \clG_{I}^{z_0}\rest{(\overline z,\overline u)}\bigl((z,u)-(\overline z,\overline u)\bigr)}{\clY_I}}{\norm{(z,u)-(\overline z,\overline u)}{\clX_I}}=0.
\]
That is, that~$\rmd \clG_{I}^{z_0}\rest{(\overline z,\overline u)}$ is the Fr\'echet derivative of~$\clG_{I}^{z_0}$ at~$(\overline z,\overline u)$.
Indeed, recalling the definition of~$\clG_{I}^{z_0}$ in~\eqref{opt-setting}, we find
\begin{align}
&\clG_{I}^{z_0}(z,u)-\clG_{I}^{z_0}(\overline z,\overline u)=(\dot w +Aw -U_M^\diamond v,w(0))+(f^{y_\ttt}(z)-f^{y_\ttt}(\overline z),0)\notag
\end{align}
with~$(w,v)\coloneqq (z,u)-(\overline z,\overline u)$ and, recalling~$f^{y_\ttt}$ defined in~\eqref{sys-z-feedf-hat}, we have
\begin{align}
f^{y_\ttt}(z)-f^{y_\ttt}(\overline z)&=w^3 +3\overline z w^2+3\overline z^2w+(3y_\ttt+\xi_2) w^2\notag\\
&\quad+2(3y_\ttt+\xi_2)\overline zw+(3y_\ttt^2+2\xi_2y_\ttt+\xi_1-1)w\\
&=w^3 +(3\overline z +3y_\ttt+\xi_2) w^2+\rmd f^{y_\ttt}\rest{\overline z}w.
\end{align}
Then, with~$\bfW\coloneqq W(I,\rmD(A),H)$. we obtain
\begin{align}
\fkQ&=\lim_{\norm{(w,v)}{\clX_I}\to0}\tfrac{\norm{(w^3 +(3\overline z+3y_\ttt+\xi_2) w^2,0)}{\clY_I}}{\norm{(w,v)}{\clX_I}}=\lim_{\norm{w}{\bfW}\to0}\tfrac{\norm{w^3 +(3\overline z+3y_\ttt+\xi_2) w^2}{L^2(I,H)}}{\norm{w}{\bfW}}.\notag
\end{align}
Next we observe that, with~$\psi\coloneqq w +3\overline z+3y_\ttt+\xi_2$,
\begin{align}\notag
\norm{w^3 +(3\overline z+3y_\ttt+\xi_2) w^2}{H}
\le\norm{\psi}{L^6}\norm{ w^2}{L^3}
\le\norm{\psi}{L^6}\norm{ w}{L^6}^2\le\norm{\Id}{\clL(V,L^6)}\norm{\psi}{L^6}\norm{ w}{V}^2
\end{align}
and
\begin{align}
\fkQ&\le \norm{\Id}{\clL(V,L^6)}\lim_{\norm{w}{\bfW}\to0}
\tfrac{\norm{\psi}{L^2(I,L^6)}\norm{w}{C(\overline I,V)}^2}{\norm{w}{\bfW}}\notag\\
&\le \norm{\Id}{\clL(V,L^6)}\norm{\Id}{\clL(\bfW,C(\overline I,V))}\lim_{\norm{w}{\bfW}\to0}
\norm{w +3\overline z+3y_\ttt+\xi_2}{L^2(I,L^6)}\norm{w}{\bfW}=0,\notag
\end{align}
which finishes the proof.
\qed

%%%%%%%%%%%%%%%%%%%%%%%%%%%%%%
%%%%%%%%%%%%%%%%%%%%%%%%%%%%%%
\bibliography{ParabSaturCont}
\bibliographystyle{plainurl}

\end{document}

%% file: Mathcommands.tex
\newcommand{\linspan}{\mathop{\rm span}\nolimits}

\newcommand{\rest}{\left.\kern-2\nulldelimiterspace\right|_}
\newcommand{\norm}[2]{\left|#1\right|_{#2}}
\newcommand{\dnorm}[2]{\left\|#1\right\|_{#2}}

\newcommand{\vol}{\mathop{\rm vol}\nolimits}

\newcommand{\Id}{{\mathbf1}}
\newcommand{\indf}{1}

\newcommand{\p}{\partial}

\newcommand*{\Bigcdot}{\raisebox{-.25ex}{\scalebox{1.25}{$\cdot$}}}

\newcommand{\clC}{{\mathcal C}}

\newcommand{\clE}{{\mathcal E}}

\newcommand{\clG}{{\mathcal G}}
\newcommand{\clH}{{\mathcal H}}
\newcommand{\clI}{{\mathcal I}}
\newcommand{\clJ}{{\mathcal J}}
\newcommand{\clK}{{\mathcal K}}
\newcommand{\clL}{{\mathcal L}}

\newcommand{\clS}{{\mathcal S}}

\newcommand{\clU}{{\mathcal U}}

\newcommand{\clX}{{\mathcal X}}
\newcommand{\clY}{{\mathcal Y}}
\newcommand{\clZ}{{\mathcal Z}}

\newcommand{\bbN}{{\mathbb N}}

\newcommand{\bbR}{{\mathbb R}}

\newcommand{\bfC}{{\mathbf C}}

\newcommand{\bfJ}{{\mathbf J}}

\newcommand{\bfP}{{\mathbf P}}

\newcommand{\bfW}{{\mathbf W}}

\newcommand{\fkB}{{\mathfrak B}}

\newcommand{\fkK}{{\mathfrak K}}

\newcommand{\fkP}{{\mathfrak P}}
\newcommand{\fkQ}{{\mathfrak Q}}

\newcommand{\fkU}{{\mathfrak U}}
\newcommand{\fkV}{{\mathfrak V}}

\newcommand{\fkX}{{\mathfrak X}}

\newcommand{\rmD}{{\mathrm D}}

\newcommand{\bfn}{{\mathbf n}}

\newcommand{\rma}{{\mathrm a}}

\newcommand{\rmd}{{\mathrm d}}
\newcommand{\rme}{{\mathrm e}}
\newcommand{\rmf}{{\mathrm f}}

\newcommand{\ttc}{{\mathtt c}}

\newcommand{\ttt}{{\mathtt t}}

\newcommand{\ovlineC}[1]{\overline C_{\left[#1\right]}}

\definecolor{LightGray}{rgb}{0.75,0.75,0.75}
\definecolor{DarkBlue}{rgb}{0,0.08,0.45}
\definecolor{DarkRed}{rgb}{.65,0,0}
\definecolor{applegreen}{rgb}{0.55, 0.71, 0.0}

\newcounter{mymac@matlab}
\newcommand{\matlab}{MATLAB%
   \ifnum\value{mymac@matlab}<1%
   \textregistered%
   \setcounter{mymac@matlab}{1}%
   \fi%
  }

\providecommand{\argmin}{\operatorname*{argmin}}

%% file: SaturatedOC-OL-arx.bbl
\begin{thebibliography}{10}

\bibitem{AzmiKun19}
B.~Azmi and K.~Kunisch.
\newblock A hybrid finite-dimensional {RHC} for stabilization of time-varying
  parabolic equations.
\newblock {\em SIAM J. Control Optim.}, 57(5):3496--3526, 2019.
\newblock \href {http://dx.doi.org/10.1137/19M1239787}
  {\path{doi:10.1137/19M1239787}}.

\bibitem{AzmiKun20}
B.~Azmi and K.~Kunisch.
\newblock Analysis of the {B}arzilai--{B}orwein step-sizes for problems in
  {H}ilbert spaces.
\newblock {\em J. Optim. Theory Appl.}, 185:819--844, 2020.
\newblock \href {http://dx.doi.org/10.1007/s10957-020-01677-y}
  {\path{doi:10.1007/s10957-020-01677-y}}.

\bibitem{AzmiKun21}
B.~Azmi and K.~Kunisch.
\newblock On the convergence and mesh-independent property of the
  {B}arzilai--{B}orwein method for {\sc pde}-constrained optimization.
\newblock {\em IMA J. Numer. Anal.}, 2021.
\newblock \href {http://dx.doi.org/10.1093/imanum/drab056}
  {\path{doi:10.1093/imanum/drab056}}.

\bibitem{AzmiKunRod21-arx}
B.~Azmi, K.~Kunisch, and S.~S. Rodrigues.
\newblock Saturated feedback stabilizability to trajectories for the
  {S}chl\"{o}gl parabolic equation.
\newblock arXiv:2111.01329 [math.OC] (preprint), 2021.
\newblock \href {http://dx.doi.org/10.48550/arXiv.2111.01329}
  {\path{doi:10.48550/arXiv.2111.01329}}.

\bibitem{AzouaniTiti14}
A.~Azouani and E.~S. Titi.
\newblock Feedback control of nonlinear dissipative systems by finite
  determining parameters -- a reaction-diffusion paradigm.
\newblock {\em Evol. Equ. Control Theory}, 3(4):579--594, 2014.
\newblock \href {http://dx.doi.org/10.3934/eect.2014.3.579}
  {\path{doi:10.3934/eect.2014.3.579}}.

\bibitem{BadTakah11}
M.~Badra and T.~Takahashi.
\newblock Stabilization of parabolic nonlinear systems with finite dimensional
  feedback or dynamical controllers: Application to the {N}avier--{S}tokes
  system.
\newblock {\em SIAM J. Control Optim.}, 49(2):420--463, 2011.
\newblock \href {http://dx.doi.org/10.1137/090778146}
  {\path{doi:10.1137/090778146}}.

\bibitem{BalKrstic00}
A.~Balogh and M.~Krstic.
\newblock {B}urgers' equation with nonlinear boundary feedback:
  {$H^1$}~stability, well-posedness and simulation.
\newblock {\em Math. Probl. Engineering}, 6:189--200, 2000.
\newblock \href {http://dx.doi.org/10.1155/S1024123X00001320}
  {\path{doi:10.1155/S1024123X00001320}}.

\bibitem{BarbuChengFreeman97}
C.~Barbu, J.~J. Cheng, and R.~A. Freeman.
\newblock Achieving maximum regions of attraction for unstable linear systems
  with control constraints.
\newblock In {\em Proceedings of the American Control Conference (ACC),
  Albuquerque, New Mexico}, pages 848--852, 1997.
\newblock \href {http://dx.doi.org/10.1109/ACC.1997.611924}
  {\path{doi:10.1109/ACC.1997.611924}}.

\bibitem{Barbu11}
V.~Barbu.
\newblock {\em Stabilization of {N}avier--{S}tokes Flows}.
\newblock Comm. Control Engrg. Ser. Springer-Verlag London, 2011.
\newblock \href {http://dx.doi.org/10.1007/978-0-85729-043-4}
  {\path{doi:10.1007/978-0-85729-043-4}}.

\bibitem{Barbu12}
V.~Barbu.
\newblock Stabilization of {N}avier--{S}tokes equations by oblique boundary
  feedback controllers.
\newblock {\em SIAM J. Control Optim.}, 50(4):2288--2307, 2012.
\newblock \href {http://dx.doi.org/10.1137/110837164}
  {\path{doi:10.1137/110837164}}.

\bibitem{Barbu_TAC13}
V.~Barbu.
\newblock Boundary stabilization of equilibrium solutions to parabolic
  equations.
\newblock {\em IEEE Trans. Automat. Control}, 58(9):2416--2420, 2013.
\newblock \href {http://dx.doi.org/10.1109/TAC.2013.2254013}
  {\path{doi:10.1109/TAC.2013.2254013}}.

\bibitem{BarbuLasTri06}
V.~Barbu, I.~Lasiecka, and R.~Triggiani.
\newblock Abstract settings for tangential boundary stabilization of
  {N}avier--{S}tokes equations by high- and low-gain feedback controllers.
\newblock {\em Nonlinear Anal.}, 64(12):2704--2746, 2006.
\newblock \href {http://dx.doi.org/10.1016/j.na.2005.09.012}
  {\path{doi:10.1016/j.na.2005.09.012}}.

\bibitem{BarRodShi11}
V.~Barbu, S.~S. Rodrigues, and A.~Shirikyan.
\newblock Internal exponential stabilization to a nonstationary solution for
  {3D} {N}avier--{S}tokes equations.
\newblock {\em SIAM J. Control Optim.}, 49(4):1454--1478, 2011.
\newblock \href {http://dx.doi.org/10.1137/100785739}
  {\path{doi:10.1137/100785739}}.

\bibitem{BarbuTri04}
V.~Barbu and R.~Triggiani.
\newblock Internal stabilization of {N}avier--{S}tokes equations with
  finite-dimensional controllers.
\newblock {\em Indiana Univ. Math. J.}, 53(5):1443--1494, 2004.
\newblock \href {http://dx.doi.org/10.1512/iumj.2004.53.2445}
  {\path{doi:10.1512/iumj.2004.53.2445}}.

\bibitem{BarzBorw88}
J.~Barzilai and J.~M. Borwein.
\newblock Two-point step size gradient methods.
\newblock {\em IMA J. Numer. Anal.}, 8(1):141--148, 1988.
\newblock \href {http://dx.doi.org/10.1093/imanum/8.1.141}
  {\path{doi:10.1093/imanum/8.1.141}}.

\bibitem{BauschkeComb17}
H.~H. Bauschke and P.~L. Combettes.
\newblock {\em Convex Analysis and Monotone Operator Theory in {H}ilbert
  Spaces}.
\newblock CMB Books in Mathematics. Springer New York, 2nd edition, 2017.
\newblock \href {http://dx.doi.org/10.1007/978-1-4419-9467-7}
  {\path{doi:10.1007/978-1-4419-9467-7}}.

\bibitem{BennerLiPenzl08}
P.~Benner, J.-R. Li, and T.~Penzl.
\newblock Numerical solution of large-scale {L}yapunov equations, {R}iccati
  equations, and linear-quadratic optimal control problems.
\newblock {\em Numer. Linear Algebra Appl.}, 15(9):755--777, 2008.
\newblock \href {http://dx.doi.org/10.1002/nla.622}
  {\path{doi:10.1002/nla.622}}.

\bibitem{BreKunRod17}
T.~Breiten, K.~Kunisch, and S.~S. Rodrigues.
\newblock Feedback stabilization to nonstationary solutions of a class of
  reaction diffusion equations of {F}itz{H}ugh--{N}agumo type.
\newblock {\em SIAM J. Control Optim.}, 55(4):2684--2713, 2017.
\newblock \href {http://dx.doi.org/10.1137/15M1038165}
  {\path{doi:10.1137/15M1038165}}.

\bibitem{CasasKun17}
E.~Casas and K.~Kunisch.
\newblock Stabilization by sparse controls for a class of semilinear parabolic
  equations.
\newblock {\em SIAM J. Control Optim.}, 55(1):512--532, 2017.
\newblock \href {http://dx.doi.org/10.1137/16M1084298}
  {\path{doi:10.1137/16M1084298}}.

\bibitem{CochranVazquezKrstic06}
J.~Cochran, R.~Vazquez, and M.~Krstic.
\newblock Backstepping boundary control of {N}avier--{S}tokes channel flow: A
  {3D} extension.
\newblock In {\em Proceedings of the 2006 American Control Conference,
  Minneapolis, Minnesota, USA}, pages 769--774, 6 2006.
\newblock URL: \url{10.1109/ACC.2006.1655449}.

\bibitem{CorradiniCristofaroOrlando10}
M.L. Corradini, A.~Cristofaro, and G.~Orlando.
\newblock Robust stabilization of multi input plants with saturating actuators.
\newblock {\em IEEE Trans. Automat. Control}, 55(2):419--425, 2010.
\newblock \href {http://dx.doi.org/10.1109/TAC.2009.2036308}
  {\path{doi:10.1109/TAC.2009.2036308}}.

\bibitem{DemengelDem12}
F.~Demengel and G.~Demengel.
\newblock {\em Functional Spaces for the Theory of Elliptic Partial
  Differential Equations}.
\newblock Universitext. Springer, 2012.
\newblock \href {http://dx.doi.org/10.1007/978-1-4471-2807-6}
  {\path{doi:10.1007/978-1-4471-2807-6}}.

\bibitem{GruneSchaSchi22}
L.~Gr\"{u}ne, M.~Schaller, and A.~Schiela.
\newblock Efficient model predictive control for parabolic {\sc pde}s with goal
  oriented error estimation.
\newblock {\em SIAM J. Sci. Comput.}, 44(1):A471 -- A500, 2022.
\newblock \href {http://dx.doi.org/10.1137/20M1356324}
  {\path{doi:10.1137/20M1356324}}.

\bibitem{HinzePinUlbrUlbr09}
M.~Hinze, R.~Pinnau, M.~Ulbrich, and S.~Ulbrich.
\newblock {\em Optimization with {PDE} Constraints}.
\newblock Number~23 in Math. Modelling: Theory and Appl. Springer Netherlands,
  2009.
\newblock \href {http://dx.doi.org/10.1007/978-1-4020-8839-1}
  {\path{doi:10.1007/978-1-4020-8839-1}}.

\bibitem{HuLinQiu01}
T.~Hu, Z.~Lin, and L.~Qiu.
\newblock Stabilization of exponentially unstable linear systems with
  saturating actuators.
\newblock {\em IEEE Trans. Automat. Control}, 46(6):973--979, 2001.
\newblock \href {http://dx.doi.org/10.1109/9.928610}
  {\path{doi:10.1109/9.928610}}.

\bibitem{KrsticMagnVazq09}
M.~Krstic, L.~Magnis, and R.~Vazquez.
\newblock Nonlinear control of the viscous {B}urgers equation: Trajectory
  generation, tracking, and observer design.
\newblock {\em J. Dyn. Syst. Meas. Control}, 131(2):\{021012\}, 2009.
\newblock \href {http://dx.doi.org/10.1115/1.3023128}
  {\path{doi:10.1115/1.3023128}}.

\bibitem{KunRod19-cocv}
K.~Kunisch and S.~S. Rodrigues.
\newblock Explicit exponential stabilization of nonautonomous linear
  parabolic-like systems by a finite number of internal actuators.
\newblock {\em ESAIM Control Optim. Calc. Var.}, 25:\{67\}, 2019.
\newblock \href {http://dx.doi.org/10.1051/cocv/2018054}
  {\path{doi:10.1051/cocv/2018054}}.

\bibitem{KunRodWalter21}
K.~Kunisch, S.~S. Rodrigues, and D.~Walter.
\newblock Learning an optimal feedback operator semiglobally stabilizing
  semilinear parabolic equations.
\newblock {\em Appl. Math. Optim.}, 2021.
\newblock \href {http://dx.doi.org/10.1007/s00245-021-09769-5}
  {\path{doi:10.1007/s00245-021-09769-5}}.

\bibitem{KunPfeiffer20}
Karl Kunisch and Laurent Pfeiffer.
\newblock The effect of the terminal penalty control for a class of
  stabilization problems.
\newblock {\em ESAIM: Control Optim. Calc. Var.}, 26:58, 2020.
\newblock \href {http://dx.doi.org/10.1051/cocv/2019037}
  {\path{doi:10.1051/cocv/2019037}}.

\bibitem{LasieckaSeidman03}
I.~Lasiecka and T.~I. Seidman.
\newblock Strong stability of elastic control systems with dissipative
  saturating feedback.
\newblock {\em Systems Control Lett.}, 48(3-4):243--252, 2003.
\newblock \href {http://dx.doi.org/10.1016/S0167-6911(02)00269-4}
  {\path{doi:10.1016/S0167-6911(02)00269-4}}.

\bibitem{LauvdalFossen97}
T.~Lauvdal and T.I. Fossen.
\newblock Stabilization of linear unstable systems with control constraints.
\newblock In {\em Proceedings of the 36th IEEE Conference on Decision and
  Control}, pages 4504--4509, 1997.
\newblock \href {http://dx.doi.org/10.1109/CDC.1997.649680}
  {\path{doi:10.1109/CDC.1997.649680}}.

\bibitem{LiuChitourSontag96}
W.~Liu, Y.~Chitour, and E.~Sontag.
\newblock On finite-gain stabilizability of linear systems subject to input
  saturation.
\newblock {\em SIAM J. Control Optim.}, 34(4):1190--1219, 1996.
\newblock \href {http://dx.doi.org/10.1137/S0363012994263469}
  {\path{doi:10.1137/S0363012994263469}}.

\bibitem{LunasinTiti17}
E.~Lunasin and E.S. Titi.
\newblock Finite determining parameters feedback control for distributed
  nonlinear dissipative systems -- a computational study.
\newblock {\em Evol. Equ. Control Theory}, 6(4):535--557, 2017.
\newblock \href {http://dx.doi.org/10.3934/eect.2017027}
  {\path{doi:10.3934/eect.2017027}}.

\bibitem{MironchenkoPrieurWirth21}
A.~Mironchenko, C.~Prieur, and F.~Wirth.
\newblock Local stabilization of an unstable parabolic equation via saturated
  controls.
\newblock {\em IEEE Trans. Automat. Control}, 66(5):2162--2176, 2021.
\newblock \href {http://dx.doi.org/10.1109/TAC.2020.3007733}
  {\path{doi:10.1109/TAC.2020.3007733}}.

\bibitem{Penzl00}
T.~Penzl.
\newblock Eigenvalue decay bounds for solutions of {L}yapunov equations: the
  symmetric case.
\newblock {\em Syst. Control Lett.}, 40(2):139--144, 2000.
\newblock \href {http://dx.doi.org/10.1016/S0167-6911(00)00010-4}
  {\path{doi:10.1016/S0167-6911(00)00010-4}}.

\bibitem{PhanRod17}
D.~Phan and S.~S. Rodrigues.
\newblock {G}evrey regularity for {N}avier--{S}tokes equations under {L}ions
  boundary conditions.
\newblock {\em J. Funct. Anal.}, 272(7):2865--2898, 2017.
\newblock \href {http://dx.doi.org/10.1016/j.jfa.2017.01.014}
  {\path{doi:10.1016/j.jfa.2017.01.014}}.

\bibitem{PhanRod18-mcss}
D.~Phan and S.S. Rodrigues.
\newblock Stabilization to trajectories for parabolic equations.
\newblock {\em Math. Control Signals Syst.}, 30(2):\{11\}, 2018.
\newblock \href {http://dx.doi.org/10.1007/s00498-018-0218-0}
  {\path{doi:10.1007/s00498-018-0218-0}}.

\bibitem{Raymond19}
J.-P. Raymond.
\newblock Stabilizability of infinite-dimensional systems by finite-dimensional
  controls.
\newblock {\em Comput. Methods Appl. Math.}, 19(4):797--811, 2019.
\newblock \href {http://dx.doi.org/10.1515/cmam-2018-0031}
  {\path{doi:10.1515/cmam-2018-0031}}.

\bibitem{Rod18}
S.S. Rodrigues.
\newblock Feedback boundary stabilization to trajectories for {3D}
  {N}avier--{S}tokes equations.
\newblock {\em Appl. Math. Optim.}, 2018.
\newblock \href {http://dx.doi.org/10.1007/s00245-017-9474-5}
  {\path{doi:10.1007/s00245-017-9474-5}}.

\bibitem{Rod20-eect}
S.S. Rodrigues.
\newblock Semiglobal exponential stabilization of nonautonomous semilinear
  parabolic-like systems.
\newblock {\em Evol. Equ. Control Theory}, 9(3):635--672, 2020.
\newblock \href {http://dx.doi.org/10.3934/eect.2020027}
  {\path{doi:10.3934/eect.2020027}}.

\bibitem{Rod21-sicon}
S.S. Rodrigues.
\newblock Oblique projection exponential dynamical observer for nonautonomous
  linear parabolic-like equations.
\newblock {\em SIAM J. Control Optim.}, 59(1):464--488, 2021.
\newblock \href {http://dx.doi.org/10.1137/19M1278934}
  {\path{doi:10.1137/19M1278934}}.

\bibitem{Rod21-aut}
S.S. Rodrigues.
\newblock Oblique projection output-based feedback stabilization of
  nonautonomous parabolic equations.
\newblock {\em Automatica J. IFAC}, 129:\{109621\}, 2021.
\newblock \href {http://dx.doi.org/10.1016/j.automatica.2021.109621}
  {\path{doi:10.1016/j.automatica.2021.109621}}.

\bibitem{SaberiLinTeel96}
A.~Saberi, Z.~Lin, and A.R. Teel.
\newblock Control of linear systems with saturating actuators.
\newblock {\em IEEE Trans. Automat. Control}, 41(3):368--378, 1996.
\newblock \href {http://dx.doi.org/10.1109/9.486638}
  {\path{doi:10.1109/9.486638}}.

\bibitem{SeidmanLi01}
T.I. Seidman and H.~Li.
\newblock A note on stabilization with saturating feedback.
\newblock {\em Discrete Contin. Dyn. Syst.}, 7(2):319--328, 2001.
\newblock \href {http://dx.doi.org/10.3934/dcds.2001.7.319}
  {\path{doi:10.3934/dcds.2001.7.319}}.

\bibitem{Slemrod89}
M.~Slemrod.
\newblock Feedback stabilization of a linear control system in {H}ilbert space
  with an a priori bounded control.
\newblock {\em Math. Control Signals Syst.}, 2(3):265--285, 1989.
\newblock \href {http://dx.doi.org/10.1007/BF02551387}
  {\path{doi:10.1007/BF02551387}}.

\bibitem{SussmannSontagYang94}
H.J. Sussmann, E.D. Sontag, and Y.~Yang.
\newblock A general result on the stabilization of linear systems using bounded
  controls.
\newblock {\em IEEE Trans. Automat. Control}, 39(12):2411--2425, 1994.
\newblock \href {http://dx.doi.org/10.1109/9.362853}
  {\path{doi:10.1109/9.362853}}.

\bibitem{Teel92}
A.R. Teel.
\newblock Global stabilization and restricted tracking for multiple integrators
  with bounded controls.
\newblock {\em Systems Control Lett.}, 18(3):165--171, 1992.
\newblock \href {http://dx.doi.org/10.1016/0167-6911(92)90001-9}
  {\path{doi:10.1016/0167-6911(92)90001-9}}.

\bibitem{Temam01}
R.~Temam.
\newblock {\em {N}avier--{S}tokes Equations: Theory and Numerical Analysis}.
\newblock AMS Chelsea Publishing, Providence, RI, {reprint of the 1984}
  edition, 2001.
\newblock URL: \url{https://bookstore.ams.org/chel-343-h}.

\bibitem{WredenhagenBelanger94}
G.F. Wredenhagen and P.R. B\'elanger.
\newblock Piecewise-linear {LQ} control for systems with input constraints.
\newblock {\em Automatica J. IFAC}, 30(3):403--416, 1994.
\newblock \href {http://dx.doi.org/10.1016/0005-1098(94)90118-X}
  {\path{doi:10.1016/0005-1098(94)90118-X}}.

\bibitem{ZhouLam17}
B.~Zhou and J.~Lam.
\newblock Global stabilization of linearized spacecraft rendezvous system by
  saturated linear feedback.
\newblock {\em IEEE Trans. Control Syst. Tech.}, 25(6):2185--2193, 2017.
\newblock \href {http://dx.doi.org/10.1109/TCST.2016.2632529}
  {\path{doi:10.1109/TCST.2016.2632529}}.

\bibitem{ZoweKurcyusz79}
J.~Zowe and S.~Kurcyusz.
\newblock Regularity and stability for the mathematical programming problem in
  {Banach} spaces.
\newblock {\em Appl. Math. Optim.}, 5(1):49--62, 1979.
\newblock \href {http://dx.doi.org/10.1007/BF01442543}
  {\path{doi:10.1007/BF01442543}}.

\end{thebibliography}
